\documentclass{amsart}
  
\usepackage[utf8]{inputenc}
\usepackage{calc}
\usepackage{graphbox}
\usepackage[T1]{fontenc}
\usepackage{ textcomp }
  \usepackage{amsthm} 
  \usepackage{amsmath} 
  \usepackage{amssymb} 
  \usepackage{ marvosym }
  \usepackage[tmargin=0.75in,bmargin=0.75in,lmargin=0.75in,rmargin=0.75in]{geometry}
\usepackage{tikz}
\usepackage{float}
\usepackage{standalone}
 \usepackage{graphicx}
  \graphicspath{{pictures/}}
  \usepackage{xspace}
  \usepackage[all]{xy}
  \usepackage{pinlabel}
  \usepackage{enumerate}
  \usepackage[font=small,labelfont=bf]{caption}
  \usepackage{wrapfig}
\usepackage{lipsum}
\usepackage[percent]{overpic}
\usepackage{caption}
\usepackage{subcaption}
\usepackage{stackrel}
\usepackage{accents}

\newlength{\dhatheight}

 \newcommand\scalemath[2]{\scalebox{#1}{\mbox{\ensuremath{\displaystyle #2}}}} 

 \usepackage{hyperref}
  
  \hypersetup{
    colorlinks=true,
    linkcolor=blue,
    filecolor=magenta,      
    urlcolor=cyan,
}

  \newcommand{\calA}{\mathcal{A}}
  
  \newcommand{\calC}{\mathcal{C}}

  \newcommand{\calF}{\mathcal{F}}
  \newcommand{\calG}{\mathcal{G}}

  \newcommand{\calL}{\mathcal{L}}
  \newcommand{\calM}{\mathcal{M}}
  \newcommand{\calN}{\mathcal{N}}
  
  \newcommand{\calP}{\mathcal{P}}

  \newcommand{\calT}{\mathcal{T}}

  \newcommand{\HH}{\mathbb{H}}

  \newcommand{\NN}{\mathbb{N}}
  
  \newcommand{\PP}{\mathbb{P}}
  
  \newcommand{\RR}{\mathbb{R}}

  \newcommand{\ZZ}{\mathbb{Z}}

  \newtheorem{theorem}{Theorem}[section]
  \newtheorem{proposition}[theorem]{Proposition}
  \newtheorem{corollary}[theorem]{Corollary}
  \newtheorem{lemma}[theorem]{Lemma}

  \theoremstyle{definition}
  
  \newtheorem{claim}[theorem]{Claim}
  
  \newtheorem*{claim*}{Claim}

  \newtheorem*{question*}{Question}
  \newtheorem*{answer*}{Answer}
  \newtheorem*{application*}{Application}

  \newtheorem{remark}[theorem]{Remark}
  \newtheorem*{remark*}{Remark}
  
  \newcommand{\secref}[1]{Section~\ref{#1}}
  \newcommand{\thmref}[1]{Theorem~\ref{#1}}
  \newcommand{\corref}[1]{Corollary~\ref{#1}}
  \newcommand{\lemref}[1]{Lemma~\ref{#1}}
  \newcommand{\propref}[1]{Proposition~\ref{#1}}
  \newcommand{\clmref}[1]{Claim~\ref{#1}}

  \newcommand{\figref}[1]{Figure~\ref{#1}}

  \newcommand{\eqnref}[1]{Equation~\eqref{#1}}

  \DeclareMathOperator{\arccosh}{arccosh}
  \DeclareMathOperator{\arcsinh}{arcsinh}

  \newcommand{\emul}{\mathrel{\ooalign{\raisebox{1.4\height}{$\mkern0.2mu\scriptstyle\ast$}\cr\hidewidth$\asymp$\hidewidth\cr}}}
  
  \newcommand{\lmul}{\stackrel{{}_\ast}{\prec}}
  \newcommand{\eadd}{\mathrel{\ooalign{\raisebox{1.6\height}{$\mkern0.2mu\scriptscriptstyle+$}\cr\hidewidth$\asymp$\hidewidth\cr}}}
  
  \newcommand{\ladd}{\stackrel{{}_+}{\prec}}
  
  \newcommand{\emuladd}{\mathrel{\ooalign{\raisebox{1.4\height}{$\mkern1.2mu\scriptstyle\ast$}\cr\hidewidth$\asymp$\hidewidth\cr\raisebox{-1\height}{$\mkern0.2mu\scriptscriptstyle+$}\cr\hidewidth}}}

  \newcommand{\lmuladd}{\mathrel{\ooalign{\raisebox{1.4\height}{$\mkern1.2mu\scriptstyle\ast$}\cr\hidewidth$\prec$\hidewidth\cr\raisebox{-1\height}{$\mkern0.2mu\scriptscriptstyle+$}\cr\hidewidth}}}

  \newcommand{\MF}{\ensuremath{\mathcal{MF}}\xspace} 
   
  \newcommand{\ML}{\ensuremath{\mathcal{ML}}\xspace} 
   
  \newcommand{\EL}{\ensuremath{\mathcal{EL}}\xspace} 
    
  \newcommand{\Teich}{{Teichm\"uller }} 
  \newcommand{\param}{{\mathchoice{\mkern1mu\mbox{\raise2.2pt\hbox{$
  \centerdot$}}
  \mkern1mu}{\mkern1mu\mbox{\raise2.2pt\hbox{$\centerdot$}}\mkern1mu}{
  \mkern1.5mu\centerdot\mkern1.5mu}{\mkern1.5mu\centerdot\mkern1.5mu}}}

\begin{document}


  \title[Masur's criterion does not hold in the Thurston metric]   
  {Masur's criterion does not hold in the Thurston metric}
  

\author   {Ivan Telpukhovskiy}
 \address{Leonhard Euler International Mathematical Institute, 14th Line 29B, Vasilyevsky Island, Saint Petersburg, 199178, Russia}
 \email{ivantelp@math.toronto.edu}

\begin{abstract}  We construct a counterexample for an analogue of Masur's criterion in the setting of \Teich space equipped with the Thurston metric. For that, we find a minimal, filling, non-uniquely ergodic lamination $\lambda$ on the seven-times punctured sphere with uniformly bounded annular projection distances. Then we show that a geodesic in the corresponding \Teich space that converges to $\lambda$, stays in the thick part for the whole time.
\end{abstract}
  
\maketitle

\section{Introduction}
\label{section: introduction}

The Thurston metric is an asymmetric Finsler metric on \Teich space 
that was first introduced by Thurston in \cite{Th86}. The distance between marked hyperbolic surfaces $X$ and $Y$ is defined as the log of the infimum over the Lipschitz constants of maps from $X$ to $Y$, homotopic to the identity. Thurston showed that when $S$ has no boundary, the distance can be computed by taking the ratios of the hyperbolic lengths of the geodesic representatives of simple closed curves (s.c.c.):
\begin{equation}
\label{eqn: Th}
    d_{Th}(X,Y) = \sup_{\alpha \text{\,\,- s.c.c.}} \log \frac{\ell_{\alpha}(Y)}{\ell_{\alpha}(X)}.
\end{equation}
A class of oriented geodesics for this metric called \textit{stretch paths} was introduced in \cite{Th86}. Given a maximal geodesic lamination $\nu$ on a hyperbolic surface $X$, a stretch path starting from $X$ is obtained by stretching the leaves of $\nu$ 
and extending this deformation to the whole surface. The stretch path is controlled by the \textit{horocyclic foliation}, obtained by foliating the ideal triangles in the complement of $\nu$ by horocyclic arcs and endowed with the transverse measure that agrees with the hyperbolic length along the leaves of $\nu$. That is, the projective class of the horocyclic foliation is invariant along the stretch path.

Thurston showed that there exists a geodesic between any two points in \Teich space equipped with this metric that is a finite concatenation of stretch path segments. In general, geodesics are not unique: the length ratio in \eqnref{eqn: Th} extends continuously to the compact space of projective measured laminations $\PP \ML(S)$ and the supremum is usually (in a sense of the word) realized on a single point which is a simple closed curve, thus leaving freedom for a geodesic. 


The following is our main theorem:

\begin{theorem}
\label{theorem: main theorem}
There are Thurston stretch paths in a \Teich space with minimal, filling, but not uniquely ergodic horocyclic foliation, that stay in the thick part for the whole time.
\end{theorem}

The theorem contributes to the study of the geometry of the Thurston metric in comparison to the better studied \Teich metric. Namely, our result is in contrast with a criterion for the divergence of \Teich geodesics in the moduli space, given by Masur:

\begin{theorem}[Masur's criterion, \cite{Mas92}]
\label{Theorem: Masur's criterion in Zorich} 
Let $\mathfrak{q}$ be a unit area 
quadratic differential on a Riemann surface $X$ in the moduli space $\calM(S)$. Suppose that the vertical foliation of $\mathfrak{q}$ is minimal but not uniquely ergodic. Then the projection of the corresponding \Teich geodesic $X_t$ to the moduli space $\calM(S)$ eventually leaves every compact set as $t \rightarrow \infty$.
\end{theorem} 

\begin{remark}
The horocyclic foliation is a natural analogue of the vertical foliation in the setting of the Thurston metric, see \cite{Mir08}, \cite{CF21}.
\end{remark}

\begin{remark}
Compare \thmref{theorem: main theorem} to a result of Brock and Modami in the case of the Weil-Petersson metric on \Teich space \cite{BabBr15}: they show that there exist Weil-Petersson geodesics with minimal, filling, non-uniquely ergodic ending lamination, that are recurrent in the moduli space, but not contained in any compact subset. Hence our counterexample disobeys Masur's criterion even more than in their setting of the Weil-Petersson metric.
\end{remark}

Despite being asymmetric, and in general admitting more than one geodesic between two points, the Thurston metric exhibits some similarities to 
the \Teich metric. For example, it
differs from the \Teich metric by at most a constant in the thick part\footnote{here the constant $C(\varepsilon)$ depends on the thick part $\calT_{\varepsilon}(S)$.} and there is an analog of Minsky’s product region theorem \cite{CR07}; every Thurston geodesic between any two points in the thick part with bounded combinatorics is cobounded\footnote{for every $x,y \in \calT_{\varepsilon}(S)$ with $K-$bounded combinatorics (Defnition 2.2 in \cite{LRT12}), every $\calG(x,y)$ is in the $\varepsilon'(\varepsilon,K,S)-$thick part.}
\cite{LRT12}; the shadow of a Thurston geodesic to the curve graph is a reparameterized quasi-geodesic \cite{LRT15}. 

Nevertheless, the Thurston metric is quite different from the \Teich metric. For one, the identity map between them is neither bi-Lipschitz \cite{Li03}, nor a quasi-isometry \cite{CR07}. In the \Teich metric, whenever the vertical and the horizontal foliations of a geodesic have a large projection distance in some subsurface, the boundary of that subsurface gets short along the geodesic\footnote{for every $\varepsilon>0$ there exists $K$ such that $d_W(\mu_+,\mu_-)>K$ implies $\inf_t \ell_{\partial W}(\calG(t)) < \varepsilon$.} \cite{Raf05}. However, it follows from our construction that the endpoints of a cobounded Thurston geodesic do not necessarily have bounded combinatorics. The reason behind it is that a condition equivalent to a curve getting short along a stretch path that is expressed in terms of the subsurface projections of the endpoints is more restrictive than in the case of the \Teich metric \cite{Raf14}, and involves only the annular subsurface of $\alpha$
(see \thmref{Theorem: Short curves along a Thurston geodesic} for a precise and more general statement).
This 
allows us to produce our counterexample by constructing a minimal, filling, non-uniquely ergodic lamination with uniformly bounded annular subsurface projections.

The construction will be done on the seven-times punctured sphere.
First, in \secref{section: construction of the lamination} we construct a minimal, filling, non-uniquely ergodic lamination $\lambda$ using a modification of the machinery developed in \cite{LLR18}. Namely, we choose a partial pseudo-Anosov map $\tau$ supported on a subsurface $Y$ homeomorphic to the three-times punctured sphere with one boundary component. We pick a finite-order homeomorphism $\rho$, such that the subsurface $\rho(Y)$ is disjoint from $Y$, and the orbit of the subsurface $Y$ under $\rho$ fills the surface. Then we set $\varphi_r = \tau^r \circ \rho$ and provided with a sequence of natural numbers $\{r_i \}_{n=1}^{\infty}$ and a curve $\gamma_0$, define 
$$ \Phi_i = \varphi_{r_1}\circ \dotsc \circ \varphi_{r_i}, \,\,\, \gamma_i = \Phi_i(\gamma_0). $$
We show that under a mild growth condition on the coefficients $r_i$, the sequence of curves $\gamma_i$ forms a quasi-geodesic in the curve graph and converges to an ending lamination $\lambda$ in the Gromov boundary.
In \secref{section: invariant bigon track}, we introduce a $\Phi_i$-invariant bigon track and provide matrix representations of the maps $\tau$ and $\tau \circ \rho$. In \secref{section: computing intersection numbers}, we let $\gamma_0$ be a multicurve and produce coarse estimates for the intersection numbers between the pairs of multicurves in the sequence $\gamma_i$. In \secref{section: NUE}, we 
show that $\lambda$ is non-uniquely ergodic and we find all ergodic transverse measures on $\lambda$. In \secref{section: relative twisting bounds}, we prove that $\lambda$ has uniformly bounded annular subsurface projections. Finally, in \secref{section: geodesic} we show that there are Thurston stretch paths whose horocyclic foliation is $\lambda$, that stay in the thick part of the \Teich space for the whole time.


\subsection*{Acknowledgements.} 
I thank my advisor Kasra Rafi for proposing this problem and for his patient supervision. I also thank Mark Bell for developing the Flipper program, which helped to find the bigon track in \secref{section: invariant bigon track}. I thank Howard Masur, Jon Chaika, Saul Schleimer, Jason Behrstock and Leonid Monin for helpful discussions. I thank Babak Modami for reading the preprint and making valuable suggestions. I am grateful to the referee for the helpful comments and the corrections. The work is supported by Ministry of Science and Higher Education of the Russian Federation, agreement № 075–15–2022–287.

\section{Background}
\subsection{Notation.} We adopt the following notation. Given two quantities (or functions) $A$ and $B$, we write $A \asymp_{K,C} B$ if $\frac{1}{K}B-C \leqslant A \leqslant KB+C$. 
Further, unless explicitly stated, by the following notation we will mean that there are universal constants $K \geqslant 1, C \geqslant 0$ such that
\begin{itemize}
    \item $A \ladd B$ means $A \leqslant B+C$.
    \item $A \lmul B$ means $A \leqslant KB$.
    \item $A \eadd B$ means $A-C \leqslant  B \leqslant A+C$.
    \item $A \emul B$ means $\frac{1}{K}B \leqslant A \leqslant KB$.
    \item $A \lmuladd B$ means $A \leqslant KB+C$.
    \vspace{0.15em}
    
    \item $A \emuladd B$ means $\frac{1}{K}B-C \leqslant A \leqslant KB+C$.
\end{itemize}

\subsection{Curves and markings.}  Let $S = S_{g,n}$ be the oriented surface of genus $g \geqslant 0$ with $n \geqslant 0$ punctures and with negative Euler characteristic. A simple closed curve on $S$ is called \emph{essential} if it does not bound a disk or a punctured disk.
We will call a \emph{curve} on $S$ the free homotopy class of an essential 
simple closed curve on $S$.
Given two curves $\alpha$ and $\beta$ on $S$, we will denote the minimal geometric intersection number between their representatives by $i(\alpha,\beta)$. A \emph{multicurve} is 
a collection of pairwise disjoint curves on $S$.
A \emph{pants decomposition} $\calP$ on $S$ is a maximal multicurve on $S$, i.e. whose complement in $S$ is a disjoint union of three-times punctured spheres. A collection of curves $\Gamma$ is called \emph{filling} if for any curve $\beta$ on $S$: $i(\alpha, \beta)>0$ for some $\alpha \in \Gamma$.
A \emph{marking} $\mu$ on $S$ is a filling collection of curves.
The intersection number
between two collections of curves $\Gamma$ and $\Gamma'$ is defined as
$$
i(\Gamma, \Gamma') = \sum_{\gamma \in \Gamma, \, \gamma' \in \Gamma'} i(\gamma,\gamma').
$$

\subsection{Curve graph.} The \emph{curve graph} $\calC(S)$ of a surface $S$ is a graph whose vertex set $\calC_0(S)$ is the set of all curves on $S$. Two vertices $\alpha$ and $\beta$ are connected by an edge if the underlying curves realize the minimal possible geometric intersection number for two curves on $S$. This means that $i(\alpha,\beta)=0$, i.e. the curves are disjoint, unless $S$ is one of the 
exceptional surfaces: if 
$S$ is the punctured torus, then $i(\alpha, \beta)=1$, and
if 
$S$ is the four-times punctured sphere, then $i(\alpha,\beta)=2$.
The curve graph is the 1-skeleton of the curve complex, introduced by Harvey in \cite{Harvey81}. The metric $d_S$ on the curve graph is induced by letting each edge have unit length. Masur and Minsky showed in \cite{MasMin99} that the curve graph is Gromov hyperbolic 
using \Teich theory.

\begin{theorem}
\label{Theorem: Curve graph is hyperbolic}\cite{MasMin99}
The curve graph $\calC(S)$ is Gromov hyperbolic.
\end{theorem} 
Later, Bowditch gave another proof of this result and showed that the hyperbolicity constant of $\calC(S_{g,n})$ is bounded above by a function that is logarithmic in $g+n$ \cite{Bow06}. It was then shown that the hyperbolicity constant is uniformly bounded independently by Bowditch \cite{Bow14}, Aougab \cite{Aoug13}, Hensel, Przytycki, Webb \cite{HPW15}, Clay, Rafi, Schleimer \cite{CRS14}.

Although the compact annulus $\calA$ is not a surface of negative Euler characteristic, it is crucial for us to consider it and we separately define its curve graph. Let the vertices of $\calC(\calA)$ be the arcs connecting two boundary components of $\calA$, up to homotopies that fix the endpoints. Two vertices are connected by an edge of length $1$ if the underlying arcs have representatives with disjoint interiors. It is easy to check that $\calC(\calA)$ is quasi-isometric to $\ZZ$ with the standard metric, hence also Gromov hyperbolic (see Section 2.4 in \cite{MasMin00} for more details).

\subsection{Measured laminations and measured foliations.} We denote the space of \emph{geodesic laminations} on $S$ equipped with the Hausdorff topology by $\calG \calL(S)$. For the background on geodesic laminations we refer to Chapter 4 in \cite{CEG86}. 
We fix some definitions. A geodesic lamination is \emph{minimal} if it does not contain any proper sublaminations. A geodesic lamination is \emph{maximal} if it is not contained in any lamination as a proper subset. A geodesic lamination is \emph{filling} if the connected components of its complement are open disks or once punctured open disks.
A geodesic lamination is \emph{chain-recurrent} if it is in the closure of the set of multicurves in $\calG \calL(S)$.

We denote the space of \emph{measured laminations} on $S$ equipped with the weak$^*$ topology by $\ML(S)$.
For the background on measured laminations we refer to Chapter 8 in \cite{Mar16}. The \emph{stump} of a geodesic lamination 
is its maximal 
sublamination that admits a transverse measure of full support. We note that a minimal, filling geodesic lamination admits a transverse measure of full support. A geodesic
lamination is \emph{uniquely  ergodic} if it supports a unique transverse measure up to scaling. Otherwise it is \emph{non-uniquely ergodic}. 

We denote the space of \emph{projective measured laminations} on $S$ equipped with the quotient topology of $\ML(S) \setminus \{0\}$ by $\PP \ML(S)$. For a non-zero measured lamination $\eta \in \ML(S)$, we denote its projective class by $[\eta] \in \PP \ML(S)$. The intersection number $i(\cdot,\cdot)$ extends continuously to the space of measured laminations (for a further extension to the space of geodesic currents see Chapter 8 in \cite{Mar16}). We say that the intersection number between two projective measured laminations equals zero if it holds for every pair of their representatives in $\ML(S)$. 

Consider the subspace of $\PP \ML(S)$ consisting of projective measured laminations with 
minimal and filling support. Consider the quotient of this subspace by identifying the laminations that have the same support.
The resulting space equipped with the quotient subspace topology is the space of \emph{ending laminations} $\EL(S)$. Alternatively, the topology of $\EL(S)$ can be described as follows: a sequence $\{\nu_i\}$ of minimal, filling geodesic laminations converges to $\nu \in \EL(S)$ if every limit point of $\{\nu_i\}$ in $\calG \calL(S)$ contains $\nu$ as a sublamination. We refer to \cite{Ham06} for more details. Klarreich proved the following:

\begin{theorem}
\label{Theorem: Gromov boundary of the curve graph}\cite{Kl99}
The Gromov boundary of the curve graph $\calC(S)$ is homeomorphic to the space of ending laminations $\EL(S)$. If a sequence of curves $\{\nu_i \}$ is a quasi-geodesic in $\calC(S)$ that converges to $\nu \in \EL(S)$, then any limit point of $\{\nu_i \}$ in $\PP \ML(S)$ projects to $\nu$ under the forgetful map. 
\end{theorem}
We denote the space of \emph{measured foliations} on $S$ equipped with the weak$^*$ topology by $\MF(S)$. For the background on measured foliations we refer to \cite{FLP12}.
The spaces $\MF(S)$ and $\ML(S)$ are canonically identified, and we will sometimes not distinguish between measured laminations and measured  foliations; similarly for their projectivizations $\PP \ML(S)$ and $\PP \MF(S)$.

\subsection{\Teich space and Thurston boundary.} A \emph{marked hyperbolic surface} is a complete finite-area Riemannian surface of constant curvature $-1$ with a fixed homeomorphism from the underlying topological surface $S$. Two marked hyperbolic surfaces $X$ and $Y$ are called equivalent if there is an isometry between $X$ and $Y$ in the correct homotopy class. The collection of equivalence classes of marked hyperbolic surfaces is called the \emph{\Teich space} $\calT(S)$ of the surface $S$.
By $\ell_{\alpha}(X)$ we denote the hyperbolic length of the unique geodesic representative of the curve $\alpha$ on the surface $X$. For $\varepsilon > 0$, the \emph{$\varepsilon$-thick part} $\calT_{\varepsilon}(S)$ of the \Teich space is the set of all marked hyperbolic surfaces with no curves shorter than $\varepsilon$. 
A \emph{Bers constant} of $S$ is a number $B(S)$ such that for every $X \in \calT(S)$, there exist a pants decomposition on $X$ such that the length of each curve in it is at most $B(S)$.
We recall that the \Teich space can be compactified via the \emph{Thurston boundary} homeomorphic to $\PP \ML(S)$ so that the compactification is homeomorphic to the closed ball of dimension $6g-6+2n$. For the details of the construction using the space of geodesic currents in the case of a closed surface we refer to Chapter 8 in \cite{Mar16}.

\subsection{Mapping class group.}
\label{subsection: mapping class group.}
The \emph{mapping class group} of a surface $S$ is the group of the isotopy classes of orientation-preserving self-homeomorphisms of $S$. The mapping class group acts continuously on the space of projective measured laminations $\PP \ML(S)$. A non-periodic element of the mapping class group that has no invariant multicurves is called \emph{pseudo-Anosov}. A pseudo-Anosov mapping class has exactly two fixed points in $\PP \ML(S)$ that represent a pair of transverse measured laminations that are minimal, filling and uniquely ergodic. Moreover, given a pseudo-Anosov mapping class $\Psi$, there is a number $\lambda_{\Psi} > 1$ such that 
\begin{equation}
\label{eq: pA}
\Psi(\nu^u) = \lambda_{\Psi} \nu^u, \,\,\, \Psi(\nu^s) = \lambda_{\Psi}^{-1} \nu^s.
\end{equation}
The (classes of the) laminations $\nu^{u,s}$ in \eqnref{eq: pA} are called the \emph{unstable} and \emph{stable} laminations of $\Psi$, respectively. We refer to \cite{FLP12}, \cite{FM12} for 
more background on pseudo-Anosov homeomorphisms.   

\subsection{Subsurface projections.}
\label{subsection: subsurface projections}
By a \emph{subsurface} $Y \subset S$ we mean the isotopy class of a proper, closed, connected, embedded subsurface, such that its boundary consists of curves 
on $S$ and its punctures agree with those of $S$. Whenever we talk about curves or laminations on $Y$, we think of the boundary components of $Y$ as punctures. 
We allow $Y$ to be an annular subsurface, whose core curve is a curve on $S$.
We assume $Y$ is not a three-times punctured sphere.

The subsurface projection is a map $\pi_{Y}: \calG \calL(S) \rightarrow 2^{\calC_{0}(Y)}$ from the space of geodesic laminations on $S$ to the power set of the vertex set of the curve graph of $Y$. 
Equip $S$ with a hyperbolic metric. 
Let $\widetilde{Y}$ be the Gromov compactification of the cover of $S$ corresponding to the subgroup $\pi_1(Y)$ of $\pi_1(S)$ with the hyperbolic metric pulled back from $S$. There is a natural homeomorphism from to $\widetilde{Y}$ to $Y$, allowing to identify the curve graphs $\calC(\widetilde{Y})$ and $\calC(Y)$. For any geodesic lamination $\nu$ on $S$, let $\widetilde{\nu}$ be the closure of the complete preimage of $\nu$ in $\widetilde{Y}$.
Suppose that $Y \subset S$ is a nonannular subsurface. An arc $\beta \in \widetilde{\nu}$ is \emph{essential} if no component of $\widetilde{Y} \setminus \beta$ has closure which is a disk. For each essential arc $\beta \in \widetilde{\nu}$, let $\calN_{\beta}$ be a regular neighborhood of $\beta \cup \partial \widetilde{Y}$. Define $\pi_Y(\nu)$ to be the union of all curves which are either curve components of $\widetilde{\nu}$ or curve components of $\partial \calN_{\beta}$, where $\beta$ is an essential arc in $\widetilde{\nu}$. If $Y \subset S$ is an annular subsurface, define $\pi_Y(\nu)$ to be the union of all arcs $\beta$ in $\widetilde{\nu}$ that connect two boundary components of $\widetilde{Y}$.

We say that a lamination $\nu$ intersects the subsurface $Y$ \emph{essentially} if $\pi_{Y}(\nu)$ is non-empty. The projection distance between two laminations $\nu, \nu' \in \calG \calL(S)$ that intersect  $Y$ essentially is 
$$ d_{Y} (\nu, \nu') = \text{diam}_{\calC(Y)} (\pi_{Y}(\nu) \cup \pi_{Y}(\nu')).$$
If $Y$ is an annular subsurface with the core curve $\alpha$, we will write $d_{\alpha}(\nu, \nu')$ instead of $d_{Y}(\nu, \nu')$ for convenience (when the quantity makes sense). More generally, if $\Gamma$ is a collection of laminations, we define $\pi_Y(\Gamma) = \cup_{\nu \in \Gamma} \, \pi_Y(\nu)$ and denote by $d_{Y}(\Gamma)$ the quantity $\text{diam}_{\calC(Y)}(\pi_Y(\Gamma))$. We say that a collection of laminations $\Gamma$ intersects the subsurface $Y$ \emph{essentially} if $\pi_{Y}(\Gamma)$ is non-empty. Similarly, if $\Gamma, \Gamma'$ are collections of laminations that intersect $Y$ essentially, we define $d_{Y}(\Gamma, \Gamma') = \text{diam}_{\calC(Y)}(\pi_Y(\Gamma) \cup \pi_Y(\Gamma'))$. A collection of subsurfaces $\Gamma$ is called \emph{filling} if for any  $\nu \in \calG \calL(S)$ there is $Y \in \Gamma$ such that $\pi_{Y}(\nu)$ is non-empty.

The following lemma provides an upper bound for a subsurface projection distance in terms of intersection numbers.
\begin{lemma}[\cite{Hem01}, Lemma 2.1; \cite{MasMin00}, Section 2.4]
\label{lemma: subsurface projection upper bound}
If $Y \subset S$ is a subsurface or $Y=S$, and $\alpha, \beta$ are curves on $S$ that intersect $Y$ essentially, then
$$d_{Y}(\alpha, \beta) \leqslant 2i(\alpha, \beta)+2.$$
If $Y$ is an annular subsurface the above bound holds with multiplicative and additive factors $1$.
\end{lemma}
We state the \emph{Bounded geodesic image theorem} proved by Masur and Minsky in \cite{MasMin00}.
\begin{theorem}
\label{Theorem: BGI}\cite{MasMin00}
Given a surface $S$ there is a constant $M = M(S)$ such that whenever $Y$ is a subsurface and $g = \{\gamma_i\}_{i \in I}$ is a geodesic in $\calC(S)$ such that $\gamma_i$ intersects $Y$ essentially for all $i \in I$, then $d_Y(g) \leqslant M$.
\end{theorem}
Later, Webb proved that the value of $M$ can be chosen to be independent of the surface \cite{Webb15}. We state a corollary of \thmref{Theorem: BGI}, which follows from the stability of quasi-geodesics in Gromov hyperbolic spaces (Theorem 1.7, Chapter III.H in \cite{BH99}):
\begin{corollary}
\label{corollary: BGI}
Given $k \geqslant 1, c \geqslant 0$ and a surface $S$, there exists a constant $A = A(k, c, S)$ such that the following holds. Let $\{\gamma_i \}_{i \in I}$ be a $(k,c)$-quasi-geodesic in $\calC(S)$ which is also $1$-Lipschitz and let $Y$ be a subsurface of $S$. If every $\gamma_i$ intersects $Y$ essentially, then for every $i, j \in I$:
$$ d_Y(\gamma_i,\gamma_j) \leqslant A.$$
\end{corollary}

We say that two subsurfaces $Y, Z$ are \emph{overlapping} if the multicurve $\partial Y$ intersects $Z$ essentially and the multicurve $\partial Z$ intersects $Y$ essentially. The following relationship between subsurface projection distances was found in \cite{Behr06} and an elementary proof with explicit constants was later obtained in \cite{Man13}:

\begin{theorem}[Behrstock inequality]
\label{Theorem: Behrstock inequality}
If $Y,Z \subset S$ are overlapping subsurfaces and $\alpha$ is a lamination that intersects both of them essentially, then
$$
d_{Y} (\alpha, \partial Z) \geqslant 10 \Longrightarrow d_{Z} (\alpha, \partial Y) \leqslant 4.
$$
\end{theorem}

We also state a useful lemma on the convergence of the projection distances (we note that the definition of the projection distance in \cite{BabBr15} is slightly different from ours, but this only results in a bounded change of the additive error compared to their statement).
\begin{lemma}[\cite{BabBr15}, Lemma 2.7]
\label{Lemma: Convergence of projection distances} 
Suppose that a sequence of curves $\{\nu_i \}$ converges to a lamination $\nu$ in the Hausdorff topology on $\calG \calL(S)$. Let $Y$ be a 
subsurface, so that $\nu$ intersects $Y$ essentially. Then for any geodesic lamination $\nu'$ that intersects $Y$ essentially we have 
$$
d_{Y}(\nu_i, \nu') \eadd_{\,8} d_{Y}(\nu, \nu')
$$
for all $i$ sufficiently large.
\end{lemma} 
Finally, we state the following proposition:
\begin{proposition}[\cite{Min00}, p. 121-122]
\label{Prop: pA}
Let $\nu$ be the unstable or stable lamination of a pseudo-Anosov map $\Psi$ on a surface $S$ and let $\Gamma$ be a collection of curves on $S$. Then there is a constant $C_{\Psi, \Gamma} > 0$ such that if $Y \subset S$ is a subsurface such that $\Gamma$ intersects $Y$ essentially, then 
$$
d_Y(\nu, \Gamma) \leqslant C_{\Psi, \Gamma}.
$$
\end{proposition}

\subsection{Relative twisting.}
\label{subsection: relative twisting}
In \secref{subsection: subsurface projections}, the projection distances between laminations for the annular subsurfaces were defined. Here we extend the definition to allow us to compute projection distances between a lamination and a 
point in \Teich space, and between two 
points in \Teich space. We will refer to any of these quantities as the relative twisting about a curve $\alpha$. 


Suppose $\alpha$ is a curve, $X$ is a point in \Teich space and $\nu$ is a geodesic lamination on $S$. Suppose that $\nu$ intersects $\alpha$ essentially. Consider the Gromov compactification of the annular cover $X_{\alpha}$ that corresponds to the cyclic subgroup $\langle\alpha\rangle$ in the fundamental group $ \pi_1(S)$, with the hyperbolic metric pulled back from $X$. Consider the complete preimage $\widetilde{\nu}$ of $\nu$ in $X_{\alpha}$. Let $\alpha^{\bot}$ be a geodesic arc in $X_{\alpha}$ that is perpendicular to the geodesic in the homotopy class of the core curve. Define $d_{\alpha}(X,\nu)$ to be the maximal distance between $\widetilde{\omega}$ and $\alpha^{\bot}$ in $\calC(X_{\alpha})$, where $\widetilde{\omega}$ is any arc of $\widetilde{\nu}$ that connects two boundary components of $X_{\alpha}$ and $\alpha^{\bot}$ is any perpendicular. We refer to Section 3 in \cite{Min96} for another way to measure the amount that a lamination twists around a curve in a hyperbolic surface using the projection of lifts in the universal cover. We note that the quantity in their definition differs from ours by at most $2$. 


Lastly, we  define $d_{\alpha} (X, Y)$, where $X, Y$ are two points in \Teich space. Let $S_{\alpha}$ be the compactification of the annular cover that corresponds to $\alpha$. Let $X_{\alpha},Y_{\alpha}$ be the compactified covers with the hyperbolic metrics defined as before. Using the first metric, construct a geodesic arc $\alpha_{X}^{\bot}$, perpendicular to the geodesic in the homotopy class of the core curve. Similarly, construct a geodesic arc $\alpha_{Y}^{\bot}$. Define $d_{\alpha} (X, Y)$ to be the maximal distance between $\alpha_{X}^{\bot}$ and $\alpha_{Y}^{\bot}$ in $\calC(S_{\alpha})$, over all possible choices of the perpendiculars.


\subsection{Thurston metric on \Teich space.}
\label{subsecion: Thurston metric}
We assume that $S$ has no boundary. For a background on the Thurston metric we refer to \cite{Th86} and \cite{PapThe07}, while here we mention the necessary notions and state the results that we will use.

In \cite{Th86}, Thurston showed that the best Lipschitz constant is realized by a homeomorphism from $X$ to $Y$. Moreover, there is a unique largest chain-recurrent lamination $\Lambda(X,Y)$, called the \emph{maximally stretched lamination}, such that any map from $X$ to $Y$ realizing the infimum in \eqnref{eqn: Th}, multiplies the arc length along the lamination by the factor of $e^{d_{Th}(X,Y)}$. Generically, $\Lambda(X,Y)$ is a curve (\cite{Th86}, Section 10).

For a maximal lamination $\nu$, Thurston constructed a homeomorphism  $\calF_{\nu} : \calT(S) \rightarrow \MF(\nu)$, where $\MF(\nu)$ is the subspace of measured foliations transverse to $\nu$ and standard near the cusps (the latter means that every puncture has a neighborhood in which the leaves are homotopic to that puncture and the transverse measure of a (non-compact) arc going out to a cusp is infinite). The image of a point $X$ in the \Teich space under $\calF_{\nu}$ is the horocyclic foliation of the pair $(X,\nu)$. The space $\MF(\nu)$ has a natural cone structure given by the \emph{shearing coordinates} which produce an embedding $s_{\nu}: \calT(S) \rightarrow \RR^{\dim \calT(S)}$ such that the image is an open convex cone. We refer to \cite{Bon96}, \cite{The14} for the details of the construction. We assume that $\nu$ is not an ideal triangulation of $S$. The stretch paths 
form open rays from the origin in the image of $s_{\nu}$. Namely, given any $X$ in \Teich space $\calT(S)$, a maximal lamination $\nu$, and $t \in \RR$, we let $\text{stretch}(X, \nu, t)$ be a unique point in $\calT(S)$, such that
$$
s_{\nu} (\text{stretch}(X, \nu, t)) = e^t s_{\nu}(X).
$$
Every stretch path 
converges to the projective class of the horocyclic foliation in the Thurston boundary as $t \to \infty$ (\cite{Pap91}, Theorem 5.1). 
Every stretch path such that the stump of $\nu$ is uniquely ergodic converges to the projective class of the stump of $\nu$ as $t \to -\infty$ \cite{Th07}. 
We summarize these results in one theorem.
\begin{theorem}[\cite{Pap91},\cite{Th07}]
\label{theorem: stretch paths}
Suppose that $\nu$ is a maximal lamination on $S$ that is not an ideal triangulation. The stretch path $\text{stretch}(X,\nu,t)$ 
converges to the projective class of the horocyclic foliation $[\calF_{\nu}(X)]$ in the Thurston boundary as $t \rightarrow \infty$. 
Every stretch path $\text{stretch}(X,\nu,t)$ such that  $\text{stump}(\nu)$ is uniquely ergodic converges to the projective class of the stump $[\text{stump}(\nu)]$ in the Thurston boundary as $t \to -\infty$.
\end{theorem}
\subsection{Twisting parameter along a Thurston geodesic.} 
We introduce the notions necessary to state \thmref{Theorem: Short curves along a Thurston geodesic}.
We say that a curve $\alpha$ \emph{interacts} with a lamination $\nu$ if $\alpha$ is a leaf of $\nu$ or if $\nu$ intersects $\alpha$ essentially. 
We call $[a,b]$ the $\varepsilon$-\emph{active interval} for $\alpha$ along a Thurston geodesic $\calG(t)$ if $[a,b]$ is the maximal interval such that $\ell_{\alpha}(a)=\ell_{\alpha}(b)=\varepsilon$.
We use the notation $\text{Log}(x)=\max(1,\log(x))$. Denote $X_t = \calG(t)$.

\begin{theorem}[\cite{DLRT20}, Theorem 3.1] 
\label{Theorem: Short curves along a Thurston geodesic} 
There exists a constant $\varepsilon_0>0$ such that the following statement holds. Let $X,Y \in \calT_{\varepsilon_0}(S)$ and $\alpha$ be a curve that interacts with $\Lambda(X,Y)$. Let $\calG$ be any geodesic from $X$ to $Y$ and $\ell_{\alpha} = \min_t \ell_{\alpha}(t)$. Then
$$ d_{\alpha}(X,Y)
\emuladd
\frac{1}{\ell_{\alpha}} \text{Log} \frac{1}{\ell_{\alpha}}. $$
If $\ell_{\alpha} < \varepsilon_0$, then $d_{\alpha} (X,Y) \eadd d_{\alpha} (X_a, X_b)$, where $[a,b]$ is the $\varepsilon_0$-active interval for $\alpha$. Further, for all sufficiently small $\ell_{\alpha}$, the relative twisting $d_{\alpha}(X_t, \Lambda(X,Y))$ is uniformly bounded for all $t \leqslant a$ and $\ell_{\alpha}(t) \emul e^{t-b} \ell_{\alpha}(b)$ for all $t \geqslant b$. All errors in this statement depend only on $\varepsilon_0$.
\end{theorem}
\begin{remark}
\label{remark: short curves along a Thurston geodesic}
We note that the statement of \thmref{Theorem: Short curves along a Thurston geodesic} remains true if the condition $X,Y \in \calT_{\varepsilon_0}(S)$ is replaced with the weaker condition $\ell_{\alpha}(X),\ell_{\alpha}(Y) \geqslant \varepsilon_0$. The proof is identical. This will be crucial for us to make \corref{corollary: bounded annular combinatorics implies cobounded}.
\end{remark}

\section{Construction of the lamination}
\label{section: construction of the lamination}

In this section we construct 
a quasi-geodesic $\{\alpha_i\}$ in the curve graph of the seven-times punctured sphere $S_{0,7}$ converging to the ending lamination $\lambda$ in the Gromov boundary. We thank the referee for suggesting simpler proofs.
\subsection{Alpha sequence}
\label{Subsection: alpha sequence} 

Denote by $S = S_{0,7}$ the seven-times punctured sphere, obtained by doubling a regular heptagon on the plane along its boundary.  
Consider four curves $\alpha_0, \alpha_1, \alpha_2, \alpha_3$ on $S$ as shown in \figref{figure with alphas}.
\begin{figure}[ht]
  \begin{tikzpicture}
\draw (-1.56,1.26) .. controls (-0.58, 1.23) .. (0,2);
\draw (0,2) .. controls (0.58, 1.22) .. (1.55,1.24);
\draw (1.55,1.24) .. controls (1.32, 0.31) .. (1.95,-0.44);
\draw (1.95,-0.44) .. controls (1.06, -0.83) .. (0.86,-1.78);
\draw (0.86,-1.78) .. controls (0, -1.33) .. (-0.87,-1.77);
\draw (-0.87,-1.77) .. controls (-1.06, -0.83) .. (-1.96,-0.44);
\draw (-1.96,-0.44) .. controls (-1.32, 0.31) .. (-1.56,1.26);
\draw (0.62, 1.33) .. controls (0.65, -0.15)  .. (0, -1.44);
\draw [dashed] (0.62, 1.33) .. controls (0,0) .. (0, -1.44) node[midway,left] {\Large $\alpha_0$};

\end{tikzpicture}
    \quad
  \begin{tikzpicture}
\draw (-1.56,1.26) .. controls (-0.58, 1.23) .. (0,2);
\draw (0,2) .. controls (0.58, 1.22) .. (1.55,1.24);
\draw (1.55,1.24) .. controls (1.32, 0.31) .. (1.95,-0.44);
\draw (1.95,-0.44) .. controls (1.06, -0.83) .. (0.86,-1.78);
\draw (0.86,-1.78) .. controls (0, -1.33) .. (-0.87,-1.77);
\draw (-0.87,-1.77) .. controls (-1.06, -0.83) .. (-1.96,-0.44);
\draw (-1.96,-0.44) .. controls (-1.32, 0.31) .. (-1.56,1.26);
\draw [dashed] (-1.14, -0.91) .. controls (0,0) .. (0.62, 1.33) node[midway,right] {\Large $\alpha_1$};
\draw  (-1.14, -0.91) .. controls (-0.52,0.42) .. (0.62, 1.33);
\end{tikzpicture}
  \quad
  \begin{tikzpicture}
\draw (-1.56,1.26) .. controls (-0.58, 1.23) .. (0,2);
\draw (0,2) .. controls (0.58, 1.22) .. (1.55,1.24);
\draw (1.55,1.24) .. controls (1.32, 0.31) .. (1.95,-0.44);
\draw (1.95,-0.44) .. controls (1.06, -0.83) .. (0.86,-1.78);
\draw (0.86,-1.78) .. controls (0, -1.33) .. (-0.87,-1.77);
\draw (-0.87,-1.77) .. controls (-1.06, -0.83) .. (-1.96,-0.44);
\draw (-1.96,-0.44) .. controls (-1.32, 0.31) .. (-1.56,1.26);
\draw [dashed] (-1.14, -0.91) .. controls (0,0) .. (1.41, 0.40) node[midway,above] {\Large $\alpha_2$};
\draw (-1.14, -0.91) .. controls (0.29,-0.59) .. (1.41, 0.40);
\draw (0.2, -1.48) .. controls (1.05,-0.85) .. (1.6, 0);
\draw [dashed] (0.2, -1.48) .. controls (0.85,-0.65) .. (1.6, 0) node[midway,above] {$\delta_0$};
\draw (-1.03, -1.1) .. controls (0,-1.25) .. (1.03, -1.1);
\draw [dashed] (-1.03, -1.1) .. controls (0,-1.02) .. (1.03, -1.1) node[midway,above] {\small $\delta_1$};
\end{tikzpicture}
  \quad
  \begin{tikzpicture}
\draw (-1.56,1.26) .. controls (-0.58, 1.23) .. (0,2);
\draw (0,2) .. controls (0.58, 1.22) .. (1.55,1.24);
\draw (1.55,1.24) .. controls (1.32, 0.31) .. (1.95,-0.44);
\draw (1.95,-0.44) .. controls (1.06, -0.83) .. (0.86,-1.78);
\draw (0.86,-1.78) .. controls (0, -1.33) .. (-0.87,-1.77);
\draw (-0.87,-1.77) .. controls (-1.06, -0.83) .. (-1.96,-0.44);
\draw (-1.96,-0.44) .. controls (-1.32, 0.31) .. (-1.56,1.26);
\draw [dashed] (-1.42, 0.36) .. controls (0,0) .. (1.41, 0.40) node[midway,below] {\Large $\alpha_3$};
\draw (-1.42, 0.36) .. controls (0,0.67) .. (1.41, 0.40);
\end{tikzpicture}
  \caption{The curves $\alpha_0, \alpha_1, \alpha_2, \alpha_3$ and $\delta_0,\delta_1$ on $S$.}
  \label{figure with alphas}
\end{figure}

Let $\rho$ be the finite order homeomorphism of $S$ which is realized by the counterclockwise rotation along the angle of $\frac{6 \pi}{7}$. In other words, the map $\rho$ rotates $S$ by 3 `clicks' counterclockwise. Let $Y_0, Y_1, Y_2, Y_3$ be the subsurfaces of $S$ with the boundary curves $\alpha_0, \alpha_1, \alpha_2, \alpha_3$, respectively, and with $3$ punctures each. Denote by $\tau$ the partial pseudo-Anosov map on $S$ supported on the subsurface $Y_2$ and obtained as the composition of two half-twists $\tau = H^{-1}_{\delta_1} \circ H_{\delta_0}$ (the core curves are shown in \figref{figure with alphas}). 

For any $n \in \NN$, let $\varphi_{n} = \tau^n \circ \rho$.
Let $\{r_n\}_{n=1}^{\infty}$ be a strictly increasing sequence of natural numbers. We will impose further conditions on $\{r_n\}$ in \secref{section: computing intersection numbers}. 
Set 
\begin{equation}
\label{Eq: def Phi} 
\Phi_{i} = \varphi_{r_1} \varphi_{r_2} \dotsc \varphi_{r_{i-1}} \varphi_{r_i}.
\end{equation}

Define the curves $\alpha_i = \Phi_i(\alpha_0)$ for every $i \in \NN$. Denote by $Y_i$ the subsurface with the boundary curve $\alpha_i$ and $3$ punctures.

Observe that for any $a,b,c \in \NN$:
\begin{equation}
\begin{split}
\label{Eq: no-interac-with-pA} 
\alpha_1 = \varphi_c(\alpha_0), \alpha_2 = \varphi_b \varphi_c(\alpha_0), \alpha_3 = \varphi_a \varphi_b \varphi_c(\alpha_0); \\
Y_1 = \varphi_c(Y_0), Y_2 = \varphi_b \varphi_c(Y_0), Y_3 = \varphi_a \varphi_b \varphi_c(Y_0).
\end{split}
\end{equation}
In particular, for $i = 1,2,3$ we have that $\Phi_i(\alpha_0) = \alpha_i, \Phi_i(Y_0) = Y_i$.

We begin with the observations on the sizes of the subsurface projections between the curves in the sequence $\{ \alpha_i \}$.

\begin{claim} 
\label{Claim: Subsurface projections are big} 
There is a constant $c>0$, so that for every $i \geqslant 2$  $$d_{Y_i}(\alpha_{i-2},\alpha_{i+2}) \geqslant c \, r_{i-1} -1.$$
\end{claim} 

\begin{proof}
First we expand the expression using \eqnref{Eq: def Phi}, then simplify it by applying \eqnref{Eq: no-interac-with-pA} and using the fact that the mapping class group acts on the curve graph by isometries, and then apply the triangle inequality:

\begin{equation} 
\begin{split}
d_{Y_i}(\alpha_{i-2},\alpha_{i+2}) & = d_{\Phi_{i-2}\varphi_{r_{i-1}}  \varphi_{r_{i}}(Y_0)}(\Phi_{i-2}(\alpha_{0}),\Phi_{i-2} \varphi_{r_{i-1}}  \varphi_{r_{i}}  \varphi_{r_{i+1}}  \varphi_{r_{i+2}}  (\alpha_{0})) \\
 & =d_{Y_2}(\alpha_{0},\varphi_{r_{i-1}}(\alpha_{3}))  = d_{Y_2}(\alpha_{0}, \tau^{r_{i-1}} (\rho \alpha_{3})) \\
& \geqslant d_{Y_2}(\alpha_{0}, \tau^{r_{i-1}} (\alpha_{0})) - d_{Y_2}(\tau^{r_{i-1}} (\alpha_{0}), \tau^{r_{i-1}} (\rho \alpha_{3}))
\\
& =
d_{Y_2}(\alpha_{0}, \tau^{r_{i-1}} (\alpha_{0})) - d_{Y_2}(\alpha_{0},\rho \alpha_{3}) 
\\
& =
d_{Y_2}(\alpha_{0}, \tau^{r_{i-1}} (\alpha_{0})) - 1.
\end{split}
\end{equation}
Since the mapping class $\tau$ restricts to a pseudo-Anosov map on the surface $Y_2$, by Proposition 3.6 in \cite{MasMin99} we have $d_{Y_2}(\alpha_{0}, \tau^n (\alpha_{0})) \geqslant c \, n$ for some $c>0$, so the result follows.
\end{proof}

\begin{lemma} 
\label{Lemma: Subsurface projections are big} 
There is a constant $i_0 \in \NN$ such that for every $i_0 \leqslant i < j < k$ with $j-i \geqslant 2, k-j \geqslant 2$, the curves $\alpha_i, \alpha_k$ intersect $Y_j$ essentially and
$$d_{Y_j}(\alpha_{i},\alpha_{k}) \geqslant c \, r_{j-1} - 9.$$
\end{lemma} 

\begin{proof}
Since the sequence $\{r_n\}$ is strictly increasing, we can choose $i_0$ such that $c \, r_{i+1} - 9 \geqslant 10$ for all $i \geqslant i_0$. The proof is by induction on $n = k-i$. 

\textbf{Base: $n=4$}. It follows from \clmref{Claim: Subsurface projections are big}.

\textbf{Step.} Suppose that $k-i = n+1$. We show that the curve $\alpha_i$ intersects the subsurface $Y_j$ essentially, the case of the curve $\alpha_k$ is similar. If $j-i < 4$, it follows from \eqnref{Eq: no-interac-with-pA} together with $i(\alpha_0,\alpha_2) > 0, i(\alpha_0,\alpha_3) >0$. If $j-i \geqslant 4$, then applying the induction hypothesis to the triple $i < i+2 < j$, we obtain 
$d_{Y_{i+2}} (\alpha_i,\alpha_j) \geqslant c \, r_{i+1} - 9$. If $i(\alpha_i, \alpha_j) = 0$, then since the subsurface projection distance for the disjoint curves is at most $2$ (\cite{MasMin00}, lemma 2.2), we have $d_{Y_{i+2}} (\alpha_i,\alpha_j) \leqslant 2$, which contradicts the choice of $i_0$. Therefore, $i(\alpha_i, \alpha_j) \neq 0$ and hence the curve $\alpha_i$ intersects $Y_j$ essentially.

Now we prove that $d_{Y_j}(\alpha_{i},\alpha_{k}) \geqslant c \, r_{j-1} - 9$. By the triangle inequality, we have
$$
d_{Y_j} (\alpha_{j-2}, \alpha_{j+2}) \leqslant d_{Y_j}(\alpha_{j-2}, \alpha_i) + d_{Y_j}(\alpha_i, \alpha_k) + d_{Y_j}(\alpha_k, \alpha_{j+2}).
$$
Hence $ d_{Y_j}(\alpha_i, \alpha_k) \geqslant d_{Y_j} (\alpha_{j-2}, \alpha_{j+2}) - d_{Y_j}(\alpha_i,\alpha_{j-2}) - d_{Y_j}(\alpha_{j+2},\alpha_k)$. If $j-2-i < 2$, then $\alpha_{j-2},\alpha_i$ are disjoint and $d_{Y_j}(\alpha_i,\alpha_{j-2}) \leqslant 2$. If $j-2-i \geqslant 2$, then by the induction hypothesis we have $d_{Y_{j-2}}(\alpha_i,\alpha_j) \geqslant c \, r_{j-3} - 9 \geqslant c \, r_{i+1} - 9 \geqslant 10$. Since $\partial Y_j = \alpha_j$, by \thmref{Theorem: Behrstock inequality} we have $d_{Y_j}(\alpha_i,\alpha_{j-2}) \leqslant 4$. Similarly, $d_{Y_j}(\alpha_{j+2},\alpha_k) \leqslant 4$. Together with \clmref{Claim: Subsurface projections are big}, we obtain
$$ d_{Y_j}(\alpha_i, \alpha_k) \geqslant d_{Y_j} (\alpha_{j-2}, \alpha_{j+2}) - d_{Y_j}(\alpha_{j-2}, \alpha_i) - d_{Y_j}(\alpha_k, \alpha_{j+2}) \geqslant c \, r_{j-1} - 1 - 4 - 4 = c \, r_{j-1} - 9.$$
\end{proof}
Next, we prove the main result of the section. 

\begin{proposition}
\label{Proposition: Quasigeodesic} 
The path $\{\alpha_{i}\}$ is a quasi-geodesic in the curve graph $\calC(S)$.  
\end{proposition}

\begin{proof}
Let $i_1 \in \NN$ be such that $i_1 \geqslant i_0+1$ and $c \, r_{i+1} - 9 \geqslant M+1$ for all $i \geqslant i_1$, where $i_0$ is the constant from \lemref{Lemma: Subsurface projections are big} and $M$ is the constant from \thmref{Theorem: BGI}. We prove that if $i_1 \leqslant i < k$ with $k-i \geqslant 7d-4$ for $d \in \NN$, then $d_{S} (\alpha_i,\alpha_k) \geqslant d$.

Let $\calG$ be a geodesic between $\alpha_i$ and $\alpha_k$ in the curve graph.
By \lemref{Lemma: Subsurface projections are big} and \thmref{Theorem: BGI}, for each $j \in \{ i+2, \dotsc, k-2 \}$ there a curve $v$ in $\calG$ such that $v$ does not intersect the subsurface $Y_j$ essentially. We show that if a curve $v$ does not intersect $Y_{j}$ and $Y_{j'}$ essentially for $j,j' \in \{ i+2, \dotsc, k-2 \}$, then $|j-j'|<7$. Assume on the contrary that $|j-j'| \geqslant 7$. Observe that for every $k \in \NN$, the subsurfaces $\{ Y_k, Y_{k+1}, Y_{k+2}, Y_{k+3} \}$ fill $S$. Indeed, by \eqnref{Eq: no-interac-with-pA} it is sufficient to consider the case $k=0$, which easily follows from \figref{figure with alphas}. This observation allows us to find $m \in \NN$ with $j+2 \leqslant m \leqslant j'-2$, such that the curve $v$ intersects $Y_m$ essentially. From \lemref{Lemma: Subsurface projections are big} we know that $d_{Y_m} (\alpha_{j}, \alpha_{j'}) \geqslant c \, r_{m-1} - 9 \geqslant 10$. On the other hand, since $i(v,\alpha_{j})=i(v,\alpha_{j'})=0$, by the triangle inequality we have
$$
d_{Y_m} (\alpha_{j}, \alpha_{j'}) \leqslant d_{Y_m} (\alpha_{j}, v) + d_{Y_m} (v, \alpha_{j'}) \leqslant 2 + 2 = 4,
$$
contradiction.

For each $j \in \{i+2, \dotsc, k-2\}$ map the curve $\alpha_j$ to some vertex in $\calG$ that does not intersect $Y_j$ essentially. We have shown that this map is at most $7$-to-$1$. Also by \lemref{Lemma: Subsurface projections are big} it omits the endpoints of $\calG$, therefore if $k-i \geqslant 7d-4$, then $|\{i+2, \dotsc, k-2\}| \geqslant 7d-7$ and $d_S(\alpha_i,\alpha_k) \geqslant d$. It follows that path $\{\alpha_{i}\}$ is a quasi-geodesic.
\end{proof}

We obtain an immediate corollary from \thmref{Theorem: Gromov boundary of the curve graph}:
\begin{corollary}
\label{corollary: ending lamination} 
There is an ending lamination $\lambda$ on $S$ representing a point in the Gromov boundary of $\calC(S)$, such that
$$ \lim_{i \rightarrow \infty} \alpha_i = \lambda.$$
Furthermore, every limit point of $\{\alpha_i \}$ in $\PP \ML(S)$ defines a projective class of transverse measure on $\lambda$. 
\end{corollary}

In the remainder of the section we prove more claims about the sequence $\{\alpha_i\}$ that will be used in \secref{section: relative twisting bounds}.

Let $i_1 \in \NN$ be the constant from the proof of \propref{Proposition: Quasigeodesic}. We show:

\begin{lemma}\label{Lemma: Filling pair in 5+ steps}
For every $i_1 \leqslant i < j$ with $j-i \geqslant 5$, the curves $\alpha_i, \alpha_j$ fill $S$.
\end{lemma}

\begin{proof}
The triples $i < i+2 < j, \,\, i < i+3 < j$ satisfy the conditions of \lemref{Lemma: Subsurface projections are big}. Hence $$d_{Y_{i+2}}(\alpha_i,\alpha_j), d_{Y_{i+3}}(\alpha_i,\alpha_j) \geqslant c \, r_{i+1} - 9 \geqslant \max \{10,M+1\}.$$
If $\alpha_i, \alpha_j$ are disjoint, then $d_{Y_{i+2}}(\alpha_i, \alpha_j) \leqslant 2$, contradiction. If $d_{S}(\alpha_i,\alpha_j) =2$, let $\{\alpha_i, \alpha',\alpha_j \}$ be a geodesic in the curve graph between $\alpha_i$ and $\alpha_j$. By \thmref{Theorem: BGI}, the curve $\alpha'$ does not intersect $Y_{i+2}$ and $Y_{i+3}$ essentially. A curve that does not intersect $Y_{i+2}$ and $Y_{i+3}$ essentially is either $\alpha_{i+2}$ or $\alpha_{i+3}$: indeed, by \eqnref{Eq: no-interac-with-pA} it is enough to consider the case $i=0$, which follows from \figref{figure with alphas}. \eqnref{Eq: no-interac-with-pA} also gives $d_S(\alpha_i, \alpha_{i+2}) = d_S(\alpha_i, \alpha_{i+3}) = 2 > 1$, contradiction. Therefore the curves $\alpha_i, \alpha_j$ fill $S$.  
\end{proof}

\begin{remark}
The sequence of subsurfaces $\{ Y_i \}_{i\geqslant j}$ satisfies the conditions of Theorem 4.1 in \cite{BLMR20} for sufficiently large $j \in \NN$ with $m = 2, n = 3$. This gives another proof of \propref{Proposition: Quasigeodesic}.
\end{remark}

Let $i_0 \in \NN$ be the constant from \lemref{Lemma: Subsurface projections are big}. We show:
\begin{claim}
\label{claim: almost filling in 4 steps}
For each $i \geqslant i_0$, there is a unique curve $\beta_i$ on $S$ such that $i(\beta_i,\alpha_{i})=i(\beta_i,\alpha_{i+4})=0$. Further, $\beta_i$ is disjoint from $\alpha_{i+1}, \alpha_{i+2}$ and $\alpha_{i+3}$.
\end{claim}

\begin{proof}
Let $\beta_i$ be a curve such that $i(\beta_i,\alpha_{i})=i(\beta_i,\alpha_{i+4})=0$. By \clmref{Claim: Subsurface projections are big}, we have $d_{Y_{i+2}}(\alpha_i, \alpha_{i+4}) \geqslant c \, r_{i+1} -1  \geqslant 10$. If the curve $\beta_i$ intersects $Y_{i+2}$ essentially, we have $d_{Y_{i+2}}(\alpha_i, \alpha_{i+4}) \leqslant d_{Y_{i+2}}(\alpha_i, \beta_i)+ d_{Y_{i+2}}(\beta_i, \alpha_{i+4}) \leqslant 2+2 = 4$, contradiction.

By applying the homeomorphism $\Phi_i^{-1}$, replace the triple $\{ \beta_i, \alpha_i, \alpha_{i+4} \}$ with the triple $\{ \Phi_i^{-1}(\beta_i), \alpha_0, \varphi_{r_{i+1}}(\alpha_3) \}$ using \eqnref{Eq: no-interac-with-pA}. Denote the curve $\Phi_i^{-1}(\beta_i)$ by $\beta_0$. We proved that $\beta_0$ does not intersect $Y_2$ essentially. Together with $i(\beta_0, \alpha_0) = 0$, from \figref{figure with alphas} we have that either $\beta_0 = \alpha_1$ or $\beta_0 \subset Y_1$. Put the curves $\alpha_1, \alpha_2, \rho \alpha_3$ in minimal position and apply the homeomorphism $\tau^{r_{i+1}}$. This gives representatives of the curves $\alpha_1$ and $\varphi_{r_{i+1}}(\alpha_3)$ that are in minimal position, which shows that $i(\alpha_1, \varphi_{r_{i+1}}(\alpha_3)) >0$, therefore $\beta_0 \neq \alpha_1$. This also shows that there is a unique curve $\beta_0$ in $Y_1$ as in  \figref{figure: beta0} such that $i(\beta_0, \varphi_{r_{i+1}}(\alpha_3)) =0$. Hence the curve $\beta_i = \Phi_i(\beta_0)$ is unique. Further, the curve $\beta_0$ is disjoint from the curves $\alpha_1,\alpha_2,\alpha_3$, hence by \eqnref{Eq: no-interac-with-pA} $\beta_i$ is disjoint from $\alpha_{i+1}, \alpha_{i+2}, \alpha_{i+3}$.
\end{proof}

\begin{figure}[H]
    \centering
    \begin{tikzpicture}
\draw (-1.56,1.26) .. controls (-0.58, 1.23) .. (0,2);
\draw (0,2) .. controls (0.58, 1.22) .. (1.55,1.24);
\draw (1.55,1.24) .. controls (1.32, 0.31) .. (1.95,-0.44);
\draw (1.95,-0.44) .. controls (1.06, -0.83) .. (0.86,-1.78);
\draw (0.86,-1.78) .. controls (0, -1.33) .. (-0.87,-1.77);
\draw (-0.87,-1.77) .. controls (-1.06, -0.83) .. (-1.96,-0.44);
\draw (-1.96,-0.44) .. controls (-1.32, 0.31) .. (-1.56,1.26);

\draw   (-1.44, 0.75) .. controls (-0.45,0.95) .. (0.31, 1.6) node[midway,below] {$\beta_0$};
\draw [dashed] (-1.44, 0.75) .. controls (-0.7,1.2) .. (0.31, 1.6);

\end{tikzpicture}
    \caption{The curve $\beta_0$ on $S$.}
    \label{figure: beta0}
\end{figure}

\begin{claim}
\label{claim: almost filling in 3 steps}
For each $i \in \NN$ there are exactly three curves on $S$ that are disjoint from $\alpha_{i}$ and $\alpha_{i+3}$. For each $i \geqslant i_0+1$, only one of them intersects both $\alpha_{i-1}$ and $\alpha_{i+4}$ essentially. Further, this curve intersects $\alpha_{i+1}$ and $\alpha_{i+2}$ essentially.
\end{claim}

\begin{proof}
For the first statement, by \eqnref{Eq: no-interac-with-pA} it is sufficient to consider the case $i=0$, and it follows from \figref{figure with alphas} that these curves are $\rho^{-1}(\beta_0), \beta_0$ and $\rho^{3}(\beta_0)$. By \clmref{claim: almost filling in 4 steps}, the curves $\Phi_i(\rho^{-1}(\beta_0)) = \beta_{i-1}, \, \Phi_i(\beta_0) = \beta_i$ do not intersect essentially either $\alpha_{i-1}$ or $\alpha_{i+4}$ for $i \geqslant i_0+1$. By \clmref{claim: almost filling in 4 steps}, the curve $\Phi_i(\rho^{3}(\beta_0))$ intersects $\alpha_{i-1}$ and $\alpha_{i+4}$ essentially for $i \geqslant i_0+1$. Futher, the curve $\rho^{3}(\beta_0)$ intersects $\alpha_1$ and $\alpha_2$ essentially, hence $\Phi_i(\rho^{3}(\beta_0))$ intersects $\alpha_{i+1}$ and $\alpha_{i+2}$ essentially.
\end{proof}

\begin{claim}
\label{claim: three consecutive curves}
If a curve on $S$ is disjoint from $\alpha_{i}$ and $\alpha_{i+2}$ for some $i \in \NN$, then it is also disjoint from $\alpha_{i+1}$.
\end{claim}
\begin{proof}
By \eqnref{Eq: no-interac-with-pA} it is sufficient to consider the case $i=0$. Notice that the curve $\alpha_1$ is a boundary component of a unique subsurface which is filled by the curves $\alpha_0$ and $\alpha_2$. Therefore a curve on $S$ that is disjoint from $\alpha_0$ and $\alpha_2$, is also disjoint from $\alpha_1$, which shows the claim.
\end{proof}

\begin{claim}
\label{claim: no two consecutive curves}
For each $i \in \NN$ there is no curve on $S$ that is disjoint from $\alpha_{i+1},\alpha_{i+2}$ and intersects $\alpha_i, \alpha_{i+3}$ essentially.
\end{claim}
\begin{proof}
By \eqnref{Eq: no-interac-with-pA} it is sufficient to consider the case $i=0$. If a curve $\gamma$ on $S$ is disjoint from $\alpha_1$ and $\alpha_2$, then one of the following holds: $\gamma = \alpha_1, \gamma = \alpha_2, \gamma \subset Y_1, \gamma \subset Y_2$. If $\gamma = \alpha_1$ or $\gamma \subset Y_1$, then $\gamma$ is disjoint from $\alpha_0$, if $\gamma = \alpha_2$ or $\gamma \subset Y_2$, then $\gamma$ is disjoint from $\alpha_3$, so the result follows.  
\end{proof}

We have the following corollary:

\begin{corollary}
\label{corollary: footprint on alphas}
If a curve $\gamma$ on $S$ is disjoint from some curves in the sequence $\{\alpha_i\}_{i \geqslant i_1}$, then one of the following holds: $\gamma$ is disjoint from $5$ consecutive curves, $\gamma$ is disjoint from two curves $\alpha_i, \alpha_j$ with $j-i=3$, $\gamma$ is disjoint from $3$ consecutive curves or $\gamma$ is disjoint from $1$ curve.
\end{corollary}
\begin{proof}
Let $\ell \geqslant i_1$ be the smallest index so that $\gamma$ is disjoint from $\alpha_{\ell}$ and $r \geqslant \ell$ be the largest index so that $\gamma$ is disjoint from $\alpha_r$. By \lemref{Lemma: Filling pair in 5+ steps}, we have $r-\ell \leqslant 4$. If $r-\ell = 4$, then by \clmref{claim: almost filling in 4 steps}, $\gamma$ is disjoint from $5$ consecutive curves. If $r-\ell = 3$, then by \clmref{claim: almost filling in 3 steps}, $\gamma$ is disjoint only from $\alpha_{\ell}$ and $\alpha_r$. If $r-\ell = 2$, then by \clmref{claim: three consecutive curves}, $\gamma$ is disjoint from $3$ consecutive curves. The case $r-\ell = 1$ is impossible by \clmref{claim: no two consecutive curves}. If $r-\ell = 0$, then $\gamma$ is disjoint from $1$ curve in $\{\alpha\}_{i \geqslant i_1}$.
\end{proof}

\section{Invariant bigon track}
\label{section: invariant bigon track}

In this section, we introduce a maximal birecurrent bigon track on $S$ that is invariant under the homeomorphisms $\Phi_i$ defined in \eqnref{Eq: def Phi}. We refer the reader to \cite{PH92} for more details on train tracks and specifically to \S 3.4 in \cite{PH92} for more details on bigon tracks. 
The bigon track $T$ is shown in \figref{figure: bigon track T}:

\begin{figure}[ht]
    \centering
    \begin{tikzpicture}
\draw (-1.56,1.26) .. controls (-0.58, 1.23) .. (0,2);
\draw (0,2) .. controls (0.58, 1.22) .. (1.55,1.24);
\draw (1.55,1.24) .. controls (1.32, 0.31) .. (1.95,-0.44);
\draw (1.95,-0.44) .. controls (1.06, -0.83) .. (0.86,-1.78);
\draw (0.86,-1.78) .. controls (0, -1.33) .. (-0.87,-1.77);
\draw (-0.87,-1.77) .. controls (-1.06, -0.83) .. (-1.96,-0.44);
\draw (-1.96,-0.44) .. controls (-1.32, 0.31) .. (-1.56,1.26);

\draw [densely dotted] (-0.15,1.8) .. controls (0,1.9) .. (0.15,1.8);
\draw (-0.15,1.8) .. controls (-0.2,1.6) and (0,1.5) .. (0,1.4);
\draw (0.15,1.8) .. controls (0.2,1.6) and (0,1.5) .. (0,1.4);

\draw [densely dotted, rotate=51] (-0.15,1.8) .. controls (0,1.9) .. (0.15,1.8);
\draw [rotate=51] (-0.15,1.8) .. controls (-0.2,1.6) and (0,1.5) .. (0,1.4);
\draw [rotate=51] (0.15,1.8) .. controls (0.2,1.6) and (0,1.5) .. (0,1.4);

\draw [densely dotted, rotate=102.5] (-0.15,1.8) .. controls (0,1.9) .. (0.15,1.8);
\draw [rotate=102.5] (-0.15,1.8) .. controls (-0.2,1.6) and (0,1.5) .. (0,1.4);
\draw [rotate=102.5] (0.15,1.8) .. controls (0.2,1.6) and (0,1.5) .. (0,1.4);

\draw [densely dotted, rotate=153.8] (-0.15,1.77) .. controls (0,1.9) .. (0.15,1.77);
\draw [rotate=153.8] (-0.15,1.77) .. controls (-0.2,1.6) and (0,1.5) .. (0,1.4);
\draw [rotate=153.8] (0.15,1.77) .. controls (0.2,1.6) and (0,1.5) .. (0,1.4);

\draw [densely dotted, rotate=206] (-0.15,1.77) .. controls (0,1.9) .. (0.15,1.77);
\draw [rotate=206] (-0.15,1.77) .. controls (-0.2,1.6) and (0,1.5) .. (0,1.4);
\draw [rotate=206] (0.15,1.77) .. controls (0.2,1.6) and (0,1.5) .. (0,1.4);

\draw [densely dotted, rotate=257.4] (-0.15,1.8) .. controls (0,1.9) .. (0.15,1.8);
\draw [rotate=257.4] (-0.15,1.8) .. controls (-0.2,1.6) and (0,1.5) .. (0,1.4);
\draw [rotate=257.4] (0.15,1.8) .. controls (0.2,1.6) and (0,1.5) .. (0,1.4);

\draw [densely dotted, rotate=308.7] (-0.15,1.78) .. controls (0,1.9) .. (0.15,1.78);
\draw [rotate=308.7] (-0.15,1.78) .. controls (-0.2,1.6) and (0,1.5) .. (0,1.4);
\draw [rotate=308.7] (0.15,1.78) .. controls (0.2,1.6) and (0,1.5) .. (0,1.4);

\draw (0,1.4) .. controls (0,1.1) and (0.5,0.75) .. (1.1,0.88) 
node[midway,above] {\tiny $4$} ;

\draw [rotate=153.8] (0,1.4) .. controls (0,1.1) and (0.5,0.55) .. (1.1,0.88) node[midway,above] {\tiny $6$} ;

\draw [rotate=206] (0,1.4) .. controls (0,1.1) and (0.5,0.65) .. (1.12,0.87) node[midway,below] {\tiny $2$} ;

\draw (-1.11,0.89) .. controls (0,0.5) .. (1.1,0.88) 
node[midway,above] {\tiny $3$};

\draw [rotate=206] (-1.11,0.88) .. controls (0,0.5) .. (1.12,0.87) node[midway,left] {\tiny $1$}  ;

\draw (-1.11,0.89) .. controls (0,0) .. (1.4,-0.31) node[midway,above] {\tiny $7$};

\draw (1.4,-0.31) .. controls (0,-0.2) .. (-1.4,-0.31) node[midway,above] {\tiny $5$};

\draw (-0.63, 1.33) .. controls (-0.5,1) and (0.5,0.65) .. (1.1,0.88) node[midway,above] {\tiny $9$};
\draw (-1.4,-0.31) .. controls (-1.2,-0.25) and (-1.2,0.2) .. (-1.44, 0.3);
\draw[dashed] (-1.44, 0.3) .. controls (-1.26,1.04) .. (-0.63, 1.33);

\draw [rotate=206] (-0.62, 1.32) .. controls (-0.5,1) and (0.5,0.6) .. (1.12,0.87);
\draw [rotate=206] (-1.4,-0.31) .. controls (-1.2,-0.25) and (-1.2,0.2) .. (-1.44, 0.3) node[midway,below] {\tiny $8$};
\draw[dashed, rotate=206] (-1.43, 0.3) .. controls (-1.26,1.04) .. (-0.63, 1.31);
\end{tikzpicture}
    \caption{The bigon track $T$ with a numbering of some of its branches.}
    \label{figure:    bigon track T}
\end{figure}

The complement to $T$ in $S$ consists of $7$ punctured monogons, $3$ trigons and one bigon. The shaded region in \figref{figure: bigon of the bigon track T} shows a part of the bigon in the complement of $T$. 

\begin{figure}[ht]
    \centering
    \begin{tikzpicture}
\draw (-1.56,1.26) .. controls (-0.58, 1.23) .. (0,2);
\draw (0,2) .. controls (0.58, 1.22) .. (1.55,1.24);
\draw (1.55,1.24) .. controls (1.32, 0.31) .. (1.95,-0.44);
\draw (1.95,-0.44) .. controls (1.06, -0.83) .. (0.86,-1.78);
\draw (0.86,-1.78) .. controls (0, -1.33) .. (-0.87,-1.77);
\draw (-0.87,-1.77) .. controls (-1.06, -0.83) .. (-1.96,-0.44);
\draw (-1.96,-0.44) .. controls (-1.32, 0.31) .. (-1.56,1.26);

\draw [densely dotted] (-0.15,1.8) .. controls (0,1.9) .. (0.15,1.8);
\draw (-0.15,1.8) .. controls (-0.2,1.6) and (0,1.5) .. (0,1.4);
\draw (0.15,1.8) .. controls (0.2,1.6) and (0,1.5) .. (0,1.4);

\draw [densely dotted, rotate=51] (-0.15,1.8) .. controls (0,1.9) .. (0.15,1.8);
\draw [rotate=51] (-0.15,1.8) .. controls (-0.2,1.6) and (0,1.5) .. (0,1.4);
\draw [rotate=51] (0.15,1.8) .. controls (0.2,1.6) and (0,1.5) .. (0,1.4);

\draw [densely dotted, rotate=102.5] (-0.15,1.8) .. controls (0,1.9) .. (0.15,1.8);
\draw [rotate=102.5] (-0.15,1.8) .. controls (-0.2,1.6) and (0,1.5) .. (0,1.4);
\draw [rotate=102.5] (0.15,1.8) .. controls (0.2,1.6) and (0,1.5) .. (0,1.4);

\draw [densely dotted, rotate=153.8] (-0.15,1.77) .. controls (0,1.9) .. (0.15,1.77);
\draw [rotate=153.8] (-0.15,1.77) .. controls (-0.2,1.6) and (0,1.5) .. (0,1.4);
\draw [rotate=153.8] (0.15,1.77) .. controls (0.2,1.6) and (0,1.5) .. (0,1.4);

\draw [densely dotted, rotate=206] (-0.15,1.77) .. controls (0,1.9) .. (0.15,1.77);
\draw [rotate=206] (-0.15,1.77) .. controls (-0.2,1.6) and (0,1.5) .. (0,1.4);
\draw [rotate=206] (0.15,1.77) .. controls (0.2,1.6) and (0,1.5) .. (0,1.4);

\draw [densely dotted, rotate=257.4] (-0.15,1.8) .. controls (0,1.9) .. (0.15,1.8);
\draw [rotate=257.4] (-0.15,1.8) .. controls (-0.2,1.6) and (0,1.5) .. (0,1.4);
\draw [rotate=257.4] (0.15,1.8) .. controls (0.2,1.6) and (0,1.5) .. (0,1.4);

\draw [densely dotted, rotate=308.7] (-0.15,1.78) .. controls (0,1.9) .. (0.15,1.78);
\draw [rotate=308.7] (-0.15,1.78) .. controls (-0.2,1.6) and (0,1.5) .. (0,1.4);
\draw [rotate=308.7] (0.15,1.78) .. controls (0.2,1.6) and (0,1.5) .. (0,1.4);

\draw (0,1.4) .. controls (0,1.1) and (0.5,0.75) .. (1.1,0.88) 
;

\draw [rotate=153.8] (0,1.4) .. controls (0,1.1) and (0.5,0.55) .. (1.1,0.88) ;

\draw [rotate=206] (0,1.4) .. controls (0,1.1) and (0.5,0.65) .. (1.12,0.87) ;

\draw (-1.11,0.89) .. controls (0,0.5) .. (1.1,0.88) 
;

\draw [rotate=206] (-1.11,0.88) .. controls (0,0.5) .. (1.12,0.87)  ;

\draw (-1.11,0.89) .. controls (0,0) .. (1.4,-0.31) ;

\draw (1.4,-0.31) .. controls (0,-0.2) .. (-1.4,-0.31) ;

\draw (-0.63, 1.33) .. controls (-0.5,1) and (0.5,0.65) .. (1.1,0.88) ;
\draw (-1.4,-0.31) .. controls (-1.2,-0.25) and (-1.2,0.2) .. (-1.44, 0.3);
\draw[dashed] (-1.44, 0.3) .. controls (-1.26,1.04) .. (-0.63, 1.33);

\draw [rotate=206] (-0.62, 1.32) .. controls (-0.5,1) and (0.5,0.6) .. (1.12,0.87);
\draw [rotate=206] (-1.4,-0.31) .. controls (-1.2,-0.25) and (-1.2,0.2) .. (-1.44, 0.3);
\draw[dashed, rotate=206] (-1.43, 0.3) .. controls (-1.26,1.04) .. (-0.63, 1.31);

\draw [gray] (0.77,0.85) -- (1.15,1.23);

\draw [gray] (0.37,0.95) -- (0.71,1.29);

\draw [gray] (0.07,1.20) -- (0.38,1.51);

\draw [gray] (0.55,0.81) -- (0.62,0.88);

\draw [gray] (0.07,0.9) -- (0.22,1.05);

\draw [gray] (-0.45,1.13) -- (-0.06,1.52);

\draw [gray] (1.07,0.71) -- (1.26,0.9);

\draw [gray] (1.2,0.45) -- (1.42,0.67);

\draw [gray] (-1.42,-0.69) -- (-1.08,-0.35);

\draw [gray] (-1.15,-0.9) -- (-0.78,-0.53);

\draw [gray] (-1,-1.2) -- (-0.62,-0.82);

\draw [gray] (-1.63,-0.22) -- (-1.29,0.12);

\draw [gray] (-0.5,-1.59) -- (0.1,-0.99);

\draw [gray] (0.1,-1.45) -- (0.45,-1.1);

\draw [gray] (-0.2,-1.02) -- (-0.12,-0.94);

\draw [gray] (0.28,-1.02) -- (0.50,-0.8);

\draw [gray] (0.63,-1.28) -- (1.05,-0.86);







\end{tikzpicture}
    \caption{The bigon in the complement of $T$.}
    \label{figure:    bigon of the bigon track T}
\end{figure}

Let $V(T)$ be the convex cone consisting of all non-negative real assignments of weights to the branches of $T$ that satisfy the switch conditions. Pick the ordered subset of 9 branches of $T$ as in \figref{figure:    bigon track T}. Notice that every non-negative assignment of weights to the chosen branches can be uniquely promoted to a vector in $V(T)$. Denote by $e_i$ the vector in $V(T)$ that assigns the weight $1$ to the $i$-th branch ($i=1,\dotsc,9$) and the weight $0$ to all other branches in the chosen set. It follows that $V(T)$ is the non-negative orthant in the vector space $W(T)$ of all real assignments of weights to the branches of $T$ (that satisfy the switch conditions) with basis $e_1, \dotsc, e_9$.

The dimension of the space of measured laminations on $S$ is equal to $8$, and the natural map from $V(T)$ to $\ML(S)$ is not injective because $T$ has a bigon. Namely, we can show:

\begin{claim}
\label{claim: bigon}
The space of measured laminations carried by $T$ is naturally identified with the linear quotient cone $V'(T) = V(T)\,/\sim\,$, where for $\mu_1,\mu_2 \in V(T)$ we let $\mu_1 \sim \mu_2$ when $\mu_1 - \mu_2 \in \text{span}(2e_2-2e_4+e_6-e_8+e_9) \subset W(T)$.
\end{claim}

\begin{proof}
According to Proposition 3.4.1 in \cite{PH92} and since $\dim V(T) - \dim \ML(S) = 1$, it is sufficient to find two distinct vectors $v_1, v_2 \in V(T)$ that correspond to the same measured lamination. Indeed, it then follows that vectors $\mu_1, \mu_2 \in V(T)$ correspond to the same measured lamination if and only if $\text{span}(\mu_1 - \mu_2) = \text{span}(v_1-v_2) \subset W(T)$. Consider $v_1 = 4e_2+2e_6+2e_9$ and $v_2 = 4e_4+2e_8$. We leave it for the reader to verify that both of them correspond to the curve in \figref{figure: ambivalent curve}.
\begin{figure}[ht]
    \centering
    \begin{tikzpicture}
\draw (-1.56,1.26) .. controls (-0.58, 1.23) .. (0,2);
\draw (0,2) .. controls (0.58, 1.22) .. (1.55,1.24);
\draw (1.55,1.24) .. controls (1.32, 0.31) .. (1.95,-0.44);
\draw (1.95,-0.44) .. controls (1.06, -0.83) .. (0.86,-1.78);
\draw (0.86,-1.78) .. controls (0, -1.33) .. (-0.87,-1.77);
\draw (-0.87,-1.77) .. controls (-1.06, -0.83) .. (-1.96,-0.44);
\draw (-1.96,-0.44) .. controls (-1.32, 0.31) .. (-1.56,1.26);

\draw [densely dotted] (-0.96, -1.35) .. controls (-0.7,-1.4) .. (-0.50, -1.59);

\draw [densely dotted][rotate=154.3] (-0.96, -1.35) .. controls (-0.7,-1.4) .. (-0.50, -1.59);

\draw (1.4, -0.7) .. controls (0,-0.8) .. (-0.96, -1.35)  ;

\draw (-0.50, -1.59) .. controls (0,-1.25) .. (1.03, -1.1)  ;

\draw [dashed] (1.4, -0.7) .. controls (0.3,0.3) .. (-0.48,1.43)  ;

\draw [dashed] (1.03, -1.1) .. controls  (0,0) .. (-0.97,1.25)  ;

\draw [densely dotted][rotate=154.3] (-0.96, -1.35) .. controls (-0.7,-1.4) .. (-0.50, -1.59);

\draw [densely dotted][rotate=154.3] (-0.96, -1.35) .. controls (-0.7,-1.4) .. (-0.50, -1.59);

\draw[rotate=154.3] (1.4, -0.7) .. controls (0,-0.8) .. (-0.96, -1.35)  ;

\draw[rotate=154.3] (-0.50, -1.59) .. controls (0,-1.25) .. (1.03, -1.1)  ;

\end{tikzpicture}
    \caption{A curve on $S$ that can be represented as a vector in $V(T)$ in two different ways.}
    \label{figure: ambivalent curve}
\end{figure}
\end{proof}

\begin{proposition}
\label{Proposition: invariant bigon track} 
The bigon track $T$ is $\Phi_i$-invariant.
\end{proposition} 

\begin{proof}
It is enough to check that $T$ is invariant under the mapping classes $\tau$ and $\tau \circ \rho$. We refer to \figref{figure: The action of tau on T} and \figref{figure: The action of rho followed by tau on T} for the verification.

\begin{figure}[H]
  \begin{tikzpicture}
\draw (-1.56,1.26) .. controls (-0.58, 1.23) .. (0,2);
\draw (0,2) .. controls (0.58, 1.22) .. (1.55,1.24);
\draw (1.55,1.24) .. controls (1.32, 0.31) .. (1.95,-0.44);
\draw (1.95,-0.44) .. controls (1.06, -0.83) .. (0.86,-1.78);
\draw (0.86,-1.78) .. controls (0, -1.33) .. (-0.87,-1.77);
\draw (-0.87,-1.77) .. controls (-1.06, -0.83) .. (-1.96,-0.44);
\draw (-1.96,-0.44) .. controls (-1.32, 0.31) .. (-1.56,1.26);

\draw [densely dotted] (-0.15,1.8) .. controls (0,1.9) .. (0.15,1.8);
\draw (-0.15,1.8) .. controls (-0.2,1.6) and (0,1.5) .. (0,1.4);
\draw (0.15,1.8) .. controls (0.2,1.6) and (0,1.5) .. (0,1.4);

\draw [densely dotted, rotate=51] (-0.15,1.8) .. controls (0,1.9) .. (0.15,1.8);
\draw [rotate=51] (-0.15,1.8) .. controls (-0.2,1.6) and (0,1.5) .. (0,1.4);
\draw [rotate=51] (0.15,1.8) .. controls (0.2,1.6) and (0,1.5) .. (0,1.4);

\draw [densely dotted, rotate=102.5] (-0.15,1.8) .. controls (0,1.9) .. (0.15,1.8);
\draw [rotate=102.5] (-0.15,1.8) .. controls (-0.2,1.6) and (0,1.5) .. (0,1.4);
\draw [rotate=102.5] (0.15,1.8) .. controls (0.2,1.6) and (0,1.5) .. (0,1.4);

\draw [densely dotted, rotate=153.8] (-0.15,1.77) .. controls (0,1.9) .. (0.15,1.77);
\draw [rotate=153.8] (-0.15,1.77) .. controls (-0.2,1.6) and (0,1.5) .. (0,1.4);
\draw [rotate=153.8] (0.15,1.77) .. controls (0.2,1.6) and (0,1.5) .. (0,1.4);

\draw [densely dotted, rotate=206] (-0.15,1.77) .. controls (0,1.9) .. (0.15,1.77);
\draw [rotate=206] (-0.15,1.77) .. controls (-0.2,1.6) and (0,1.5) .. (0,1.4);
\draw [rotate=206] (0.15,1.77) .. controls (0.2,1.6) and (0,1.5) .. (0,1.4);

\draw [densely dotted, rotate=257.4] (-0.15,1.8) .. controls (0,1.9) .. (0.15,1.8);
\draw [rotate=257.4] (-0.15,1.8) .. controls (-0.2,1.6) and (0,1.5) .. (0,1.4);
\draw [rotate=257.4] (0.15,1.8) .. controls (0.2,1.6) and (0,1.5) .. (0,1.4);

\draw [densely dotted, rotate=308.7] (-0.15,1.78) .. controls (0,1.9) .. (0.15,1.78);
\draw [rotate=308.7] (-0.15,1.78) .. controls (-0.2,1.6) and (0,1.5) .. (0,1.4);
\draw [rotate=308.7] (0.15,1.78) .. controls (0.2,1.6) and (0,1.5) .. (0,1.4);

\draw (0,1.4) .. controls (0,1.1) and (0.5,0.75) .. (1.1,0.88) 
;

\draw [rotate=153.8] (0,1.4) .. controls (0,1.1) and (0.5,0.55) .. (1.1,0.88) ;

\draw [rotate=206] (0,1.4) .. controls (0,1.1) and (0.5,0.65) .. (1.12,0.87) ;

\draw (-1.11,0.89) .. controls (0,0.5) .. (1.1,0.88) 
;

\draw [rotate=206] (-1.11,0.88) .. controls (0,0.5) .. (1.12,0.87)  ;

\draw (-1.11,0.89) .. controls (0,0) .. (1.4,-0.31) ;

\draw (1.4,-0.31) .. controls (0,-0.2) .. (-1.4,-0.31) ;

\draw (-0.63, 1.33) .. controls (-0.5,1) and (0.5,0.65) .. (1.1,0.88) ;
\draw (-1.4,-0.31) .. controls (-1.2,-0.25) and (-1.2,0.2) .. (-1.44, 0.3);
\draw[dashed] (-1.44, 0.3) .. controls (-1.26,1.04) .. (-0.63, 1.33);

\draw [rotate=206] (-0.62, 1.32) .. controls (-0.5,1) and (0.5,0.6) .. (1.12,0.87);
\draw [rotate=206] (-1.4,-0.31) .. controls (-1.2,-0.25) and (-1.2,0.2) .. (-1.44, 0.3);
\draw[dashed, rotate=206] (-1.43, 0.3) .. controls (-1.26,1.04) .. (-0.63, 1.31);
\end{tikzpicture}
    \hspace*{.5in}
  \begin{tikzpicture}
\draw (-1.56,1.26) .. controls (-0.58, 1.23) .. (0,2);
\draw (0,2) .. controls (0.58, 1.22) .. (1.55,1.24);
\draw (1.55,1.24) .. controls (1.32, 0.31) .. (1.95,-0.44);
\draw (1.95,-0.44) .. controls (1.06, -0.83) .. (0.86,-1.78);
\draw (0.86,-1.78) .. controls (0, -1.33) .. (-0.87,-1.77);
\draw (-0.87,-1.77) .. controls (-1.06, -0.83) .. (-1.96,-0.44);
\draw (-1.96,-0.44) .. controls (-1.32, 0.31) .. (-1.56,1.26);

\draw [densely dotted] (-0.15,1.8) .. controls (0,1.9) .. (0.15,1.8);
\draw (-0.15,1.8) .. controls (-0.2,1.6) and (0,1.5) .. (0,1.4);
\draw (0.15,1.8) .. controls (0.2,1.6) and (0,1.5) .. (0,1.4);

\draw [densely dotted, rotate=51] (-0.15,1.8) .. controls (0,1.9) .. (0.15,1.8);
\draw [rotate=51] (-0.15,1.8) .. controls (-0.2,1.6) and (0,1.5) .. (0,1.4);
\draw [rotate=51] (0.15,1.8) .. controls (0.2,1.6) and (0,1.5) .. (0,1.4);

\draw [densely dotted, rotate=102.5] (-0.15,1.8) .. controls (0,1.9) .. (0.15,1.8);
\draw [rotate=102.5] (-0.15,1.8) .. controls (-0.2,1.6) and (0,1.5) .. (0,1.4);
\draw [rotate=102.5] (0.15,1.8) .. controls (0.2,1.6) and (0,1.5) .. (0,1.4);

\draw [densely dotted, rotate=153.8] (-0.15,1.77) .. controls (0,1.9) .. (0.15,1.77);
\draw [rotate=153.8] (-0.15,1.77) .. controls (-0.2,1.6) and (0,1.5) .. (0,1.4);
\draw [rotate=153.8] (0.15,1.77) .. controls (0.2,1.6) and (0,1.5) .. (0,1.4);

\draw [densely dotted, rotate=206] (-0.15,1.77) .. controls (0,1.9) .. (0.15,1.77);
\draw [rotate=206] (-0.15,1.77) .. controls (-0.2,1.6) and (0,1.5) .. (0,1.4);
\draw [rotate=206] (0.15,1.77) .. controls (0.2,1.6) and (0,1.5) .. (0,1.4);

\draw [densely dotted, rotate=257.4] (-0.15,1.8) .. controls (0,1.9) .. (0.15,1.8);
\draw [rotate=257.4] (-0.15,1.8) .. controls (-0.2,1.6) and (0,1.5) .. (0,1.4);
\draw [rotate=257.4] (0.15,1.8) .. controls (0.2,1.6) and (0,1.5) .. (0,1.4);

\draw [densely dotted, rotate=308.7] (-0.15,1.78) .. controls (0,1.9) .. (0.15,1.78);
\draw [rotate=308.7] (-0.15,1.78) .. controls (-0.2,1.6) and (0,1.5) .. (0,1.4);
\draw [rotate=308.7] (0.15,1.78) .. controls (0.2,1.6) and (0,1.5) .. (0,1.4);

\draw (0,1.4) .. controls (0,1.1) and (0.5,0.75) .. (1.1,0.88);

\draw [rotate=153.8] (0.37,1.52) .. controls (0,1.1) and (0.5,0.55) .. (1.1,0.88);

\draw [rotate=206] (0,1.4) .. controls (0,1.1) and (0.5,0.65) .. (1.08, 1.21);

\draw [rotate=206] (0,1.4) .. controls (0,1.1) and (0.5,0.65) .. (1.16, 1.21);

\draw [rotate=206] (0,1.4) .. controls (0,1.1) and (0.5,0.65) .. (1.00, 1.21);

\draw [rotate=206] (0,1.4) .. controls (0,1.1) and (0.5,0.65) .. (0.93, 1.22);

\draw [densely dotted, rotate=153.8] (-0.16,1.74) .. controls (0,1.84) .. (0.16,1.74);

\draw [densely dotted, rotate=153.8] (-0.21,1.69) .. controls (0,1.77) .. (0.21,1.69);

\draw [densely dotted, rotate=153.8] (-0.25,1.64) .. controls (0,1.7) .. (0.25,1.64);

\draw [densely dotted, rotate=153.8] (-0.31,1.58) .. controls (0,1.63) .. (0.31,1.58);

\draw [densely dotted, rotate=153.8] (-0.37,1.51) .. controls (0,1.55) .. (0.37,1.51);

\draw (-1.11,0.89) .. controls (0,0.5) .. (1.1,0.88);

\draw [densely dotted, rotate=257.4] (-0.19,1.72) .. controls (0,1.81) .. (0.20,1.72);

\draw [densely dotted, rotate=257.4] (-0.25,1.64) .. controls (0,1.75) .. (0.25,1.64);

\draw [densely dotted, rotate=257.4] (-0.31,1.58) .. controls (0,1.7) .. (0.31,1.58);

\draw [densely dotted, rotate=257.4] (-0.32,1.55) .. controls (0,1.66) .. (0.32,1.55);

\draw [rotate=206] (-1.11,0.88) .. controls (0,0.5) .. (1.46,0.81);

\draw [rotate=206] (-1.45,0.80) .. controls (0,0.40) .. (1.44,0.70);

\draw [rotate=206] (-1.47,0.88) .. controls (0,0.45) .. (1.44,0.75);

\draw [rotate=206] (-1.2,1.24) .. controls (-0.8,0.5) and (0.40,0.54) .. (1.47,0.86);

\draw (-1.11,0.89) .. controls (0,0) .. (1.59,0.01);

\draw (1.62,-0.03) .. controls (0,-0.2) .. (-1.4,-0.31);

\draw (-0.63, 1.33) .. controls (-0.5,1) and (0.5,0.65) .. (1.1,0.88);
\draw (-1.4,-0.31) .. controls (-1.2,-0.25) and (-1.2,0.2) .. (-1.44, 0.3);
\draw[dashed] (-1.44, 0.3) .. controls (-1.26,1.04) .. (-0.63, 1.33);

\draw [rotate=206] (-0.93, 1.24) .. controls (-0.5,0.5) and (0.5,0.6) .. (1.12,0.87);
\draw [rotate=206] (-1.02, 1.24) .. controls (-0.5,0.5) and (0.5,0.6) .. (1.12,0.87);
\draw [rotate=206] (-1.12, 1.24) .. controls (-0.5,0.5) and (0.5,0.6) .. (1.12,0.87);

\draw [rotate=206] (-0.61, 1.33) .. controls (-0.5,1) and (0.5,0.6) .. (1.23,1.21);

\draw [rotate=206] (-1.4,-0.31) .. controls (-1.2,-0.25) and (-1.2,0.2) .. (-1.44, 0.3);
\draw[dashed, rotate=206] (-1.43, 0.3) .. controls (-1.26,1.04) .. (-0.63, 1.31);
\end{tikzpicture}
  \caption{The action of $\tau$ on $T$.}
  \label{figure: The action of tau on T}

\end{figure}

\begin{figure}[H]
  \begin{tikzpicture}
\draw (-1.56,1.26) .. controls (-0.58, 1.23) .. (0,2);
\draw (0,2) .. controls (0.58, 1.22) .. (1.55,1.24);
\draw (1.55,1.24) .. controls (1.32, 0.31) .. (1.95,-0.44);
\draw (1.95,-0.44) .. controls (1.06, -0.83) .. (0.86,-1.78);
\draw (0.86,-1.78) .. controls (0, -1.33) .. (-0.87,-1.77);
\draw (-0.87,-1.77) .. controls (-1.06, -0.83) .. (-1.96,-0.44);
\draw (-1.96,-0.44) .. controls (-1.32, 0.31) .. (-1.56,1.26);

\draw [densely dotted] (-0.15,1.8) .. controls (0,1.9) .. (0.15,1.8);
\draw (-0.15,1.8) .. controls (-0.2,1.6) and (0,1.5) .. (0,1.4);
\draw (0.15,1.8) .. controls (0.2,1.6) and (0,1.5) .. (0,1.4);

\draw [densely dotted, rotate=51] (-0.15,1.8) .. controls (0,1.9) .. (0.15,1.8);
\draw [rotate=51] (-0.15,1.8) .. controls (-0.2,1.6) and (0,1.5) .. (0,1.4);
\draw [rotate=51] (0.15,1.8) .. controls (0.2,1.6) and (0,1.5) .. (0,1.4);

\draw [densely dotted, rotate=102.5] (-0.15,1.8) .. controls (0,1.9) .. (0.15,1.8);
\draw [rotate=102.5] (-0.15,1.8) .. controls (-0.2,1.6) and (0,1.5) .. (0,1.4);
\draw [rotate=102.5] (0.15,1.8) .. controls (0.2,1.6) and (0,1.5) .. (0,1.4);

\draw [densely dotted, rotate=153.8] (-0.15,1.77) .. controls (0,1.9) .. (0.15,1.77);
\draw [rotate=153.8] (-0.15,1.77) .. controls (-0.2,1.6) and (0,1.5) .. (0,1.4);
\draw [rotate=153.8] (0.15,1.77) .. controls (0.2,1.6) and (0,1.5) .. (0,1.4);

\draw [densely dotted, rotate=206] (-0.15,1.77) .. controls (0,1.9) .. (0.15,1.77);
\draw [rotate=206] (-0.15,1.77) .. controls (-0.2,1.6) and (0,1.5) .. (0,1.4);
\draw [rotate=206] (0.15,1.77) .. controls (0.2,1.6) and (0,1.5) .. (0,1.4);

\draw [densely dotted, rotate=257.4] (-0.15,1.8) .. controls (0,1.9) .. (0.15,1.8);
\draw [rotate=257.4] (-0.15,1.8) .. controls (-0.2,1.6) and (0,1.5) .. (0,1.4);
\draw [rotate=257.4] (0.15,1.8) .. controls (0.2,1.6) and (0,1.5) .. (0,1.4);

\draw [densely dotted, rotate=308.7] (-0.15,1.78) .. controls (0,1.9) .. (0.15,1.78);
\draw [rotate=308.7] (-0.15,1.78) .. controls (-0.2,1.6) and (0,1.5) .. (0,1.4);
\draw [rotate=308.7] (0.15,1.78) .. controls (0.2,1.6) and (0,1.5) .. (0,1.4);

\draw (0,1.4) .. controls (0,1.1) and (0.5,0.75) .. (1.1,0.88) 
;

\draw [rotate=153.8] (0,1.4) .. controls (0,1.1) and (0.5,0.55) .. (1.1,0.88) ;

\draw [rotate=206] (0,1.4) .. controls (0,1.1) and (0.5,0.65) .. (1.12,0.87) ;

\draw (-1.11,0.89) .. controls (0,0.5) .. (1.1,0.88) 
;

\draw [rotate=206] (-1.11,0.88) .. controls (0,0.5) .. (1.12,0.87)  ;

\draw (-1.11,0.89) .. controls (0,0) .. (1.4,-0.31) ;

\draw (1.4,-0.31) .. controls (0,-0.2) .. (-1.4,-0.31) ;

\draw (-0.63, 1.33) .. controls (-0.5,1) and (0.5,0.65) .. (1.1,0.88) ;
\draw (-1.4,-0.31) .. controls (-1.2,-0.25) and (-1.2,0.2) .. (-1.44, 0.3);
\draw[dashed] (-1.44, 0.3) .. controls (-1.26,1.04) .. (-0.63, 1.33);

\draw [rotate=206] (-0.62, 1.32) .. controls (-0.5,1) and (0.5,0.6) .. (1.12,0.87);
\draw [rotate=206] (-1.4,-0.31) .. controls (-1.2,-0.25) and (-1.2,0.2) .. (-1.44, 0.3);
\draw[dashed, rotate=206] (-1.43, 0.3) .. controls (-1.26,1.04) .. (-0.63, 1.31);
\end{tikzpicture}
    \hspace*{.5in}
  \begin{tikzpicture}[rotate=153.8]
\draw (-1.56,1.26) .. controls (-0.58, 1.23) .. (0,2);
\draw (0,2) .. controls (0.58, 1.22) .. (1.55,1.24);
\draw (1.55,1.24) .. controls (1.32, 0.31) .. (1.95,-0.44);
\draw (1.95,-0.44) .. controls (1.06, -0.83) .. (0.86,-1.78);
\draw (0.86,-1.78) .. controls (0, -1.33) .. (-0.87,-1.77);
\draw (-0.87,-1.77) .. controls (-1.06, -0.83) .. (-1.96,-0.44);
\draw (-1.96,-0.44) .. controls (-1.32, 0.31) .. (-1.56,1.26);

\draw [densely dotted] (-0.15,1.8) .. controls (0,1.9) .. (0.15,1.8);
\draw (-0.15,1.8) .. controls (-0.2,1.6) and (0,1.5) .. (0,1.4);
\draw (0.15,1.8) .. controls (0.2,1.6) and (0,1.5) .. (0,1.4);

\draw [densely dotted, rotate=51] (-0.15,1.8) .. controls (0,1.9) .. (0.15,1.8);
\draw [rotate=51] (-0.15,1.8) .. controls (-0.2,1.6) and (0,1.5) .. (0,1.4);
\draw [rotate=51] (0.15,1.8) .. controls (0.2,1.6) and (0,1.5) .. (0,1.4);

\draw [densely dotted, rotate=102.5] (-0.15,1.8) .. controls (0,1.9) .. (0.15,1.8);
\draw [rotate=102.5] (-0.15,1.8) .. controls (-0.2,1.6) and (0,1.5) .. (0,1.4);
\draw [rotate=102.5] (0.15,1.8) .. controls (0.2,1.6) and (0,1.5) .. (0,1.4);

\draw [densely dotted, rotate=153.8] (-0.15,1.77) .. controls (0,1.9) .. (0.15,1.77);
\draw [rotate=153.8] (-0.15,1.77) .. controls (-0.2,1.6) and (0,1.5) .. (0,1.4);
\draw [rotate=153.8] (0.15,1.77) .. controls (0.2,1.6) and (0,1.5) .. (0,1.4);

\draw [densely dotted, rotate=206] (-0.15,1.77) .. controls (0,1.9) .. (0.15,1.77);
\draw [rotate=206] (-0.15,1.77) .. controls (-0.2,1.6) and (0,1.5) .. (0,1.4);
\draw [rotate=206] (0.15,1.77) .. controls (0.2,1.6) and (0,1.5) .. (0,1.4);

\draw [densely dotted, rotate=257.4] (-0.15,1.8) .. controls (0,1.9) .. (0.15,1.8);
\draw [rotate=257.4] (-0.15,1.8) .. controls (-0.2,1.6) and (0,1.5) .. (0,1.4);
\draw [rotate=257.4] (0.15,1.8) .. controls (0.2,1.6) and (0,1.5) .. (0,1.4);

\draw [densely dotted, rotate=308.7] (-0.15,1.78) .. controls (0,1.9) .. (0.15,1.78);
\draw [rotate=308.7] (-0.15,1.78) .. controls (-0.2,1.6) and (0,1.5) .. (0,1.4);
\draw [rotate=308.7] (0.15,1.78) .. controls (0.2,1.6) and (0,1.5) .. (0,1.4);

\draw (0,1.4) .. controls (0,1.1) and (0.5,0.75) .. (1.1,0.88);

\draw [rotate=153.8] (0,1.4) .. controls (0,1.1) and (0.5,0.55) .. (1.1,0.88);

\draw [rotate=206] (0,1.4) .. controls (0,1.1) and (0.5,0.65) .. (1.12,0.87);

\draw (-1.11,0.89) .. controls (0,0.5) .. (1.1,0.88);

\draw [rotate=206] (-1.11,0.88) .. controls (0,0.5) .. (1.12,0.87);

\draw (-1.11,0.89) .. controls (0,0) .. (1.4,-0.31);

\draw (1.4,-0.31) .. controls (0,-0.2) .. (-1.4,-0.31);

\draw (-0.63, 1.33) .. controls (-0.5,1) and (0.5,0.65) .. (1.1,0.88);
\draw (-1.4,-0.31) .. controls (-1.2,-0.25) and (-1.2,0.2) .. (-1.44, 0.3);
\draw[dashed] (-1.44, 0.3) .. controls (-1.26,1.04) .. (-0.63, 1.33);

\draw [rotate=206] (-0.62, 1.32) .. controls (-0.5,1) and (0.5,0.6) .. (1.12,0.87);
\draw [rotate=206] (-1.4,-0.31) .. controls (-1.2,-0.25) and (-1.2,0.2) .. (-1.44, 0.3);
\draw[dashed, rotate=206] (-1.43, 0.3) .. controls (-1.26,1.04) .. (-0.63, 1.31);
\end{tikzpicture}
   \hspace*{.5in}
  \begin{tikzpicture}
\draw (-1.56,1.26) .. controls (-0.58, 1.23) .. (0,2);
\draw (0,2) .. controls (0.58, 1.22) .. (1.55,1.24);
\draw (1.55,1.24) .. controls (1.32, 0.31) .. (1.95,-0.44);
\draw (1.95,-0.44) .. controls (1.06, -0.83) .. (0.86,-1.78);
\draw (0.86,-1.78) .. controls (0, -1.33) .. (-0.87,-1.77);
\draw (-0.87,-1.77) .. controls (-1.06, -0.83) .. (-1.96,-0.44);
\draw (-1.96,-0.44) .. controls (-1.32, 0.31) .. (-1.56,1.26);

\draw [densely dotted] (-0.15,1.8) .. controls (0,1.9) .. (0.15,1.8);
\draw (-0.15,1.8) .. controls (-0.2,1.6) and (0,1.5) .. (0,1.4);
\draw (0.15,1.8) .. controls (0.2,1.6) and (0,1.5) .. (0,1.4);

\draw [densely dotted, rotate=51] (-0.15,1.8) .. controls (0,1.9) .. (0.15,1.8);
\draw [rotate=51] (-0.15,1.8) .. controls (-0.2,1.6) and (0,1.5) .. (0,1.4);
\draw [rotate=51] (0.15,1.8) .. controls (0.2,1.6) and (0,1.5) .. (0,1.4);

\draw [densely dotted, rotate=102.5] (-0.15,1.8) .. controls (0,1.9) .. (0.15,1.8);
\draw [rotate=102.5] (-0.15,1.8) .. controls (-0.2,1.6) and (0,1.5) .. (0,1.4);
\draw [rotate=102.5] (0.15,1.8) .. controls (0.2,1.6) and (0,1.5) .. (0,1.4);

\draw [densely dotted, rotate=153.8] (-0.15,1.77) .. controls (0,1.9) .. (0.15,1.77);
\draw [densely dotted, rotate=153.8] (-0.2,1.7) .. controls (0,1.75) .. (0.2,1.7);
\draw [rotate=153.8] (-0.15,1.77) .. controls (-0.2,1.6) and (0,1.5) .. (0,1.4);
\draw [rotate=153.8] (0.15,1.77) .. controls (0.2,1.6) and (0,1.5) .. (0,1.4);

\draw [densely dotted, rotate=206] (-0.15,1.77) .. controls (0,1.9) .. (0.15,1.77);
\draw [rotate=206] (-0.15,1.77) .. controls (-0.2,1.6) and (0,1.5) .. (0,1.4);
\draw [rotate=206] (0.15,1.77) .. controls (0.2,1.6) and (0,1.5) .. (0,1.4);

\draw [densely dotted, rotate=257.4] (-0.15,1.8) .. controls (0,1.9) .. (0.15,1.8);
\draw [densely dotted, rotate=257.4] (-0.2,1.7) .. controls (0,1.75) .. (0.2,1.7);

\draw [rotate=257.4] (-0.15,1.8) .. controls (-0.2,1.6) and (0,1.5) .. (0,1.4);
\draw [rotate=257.4] (0.15,1.8) .. controls (0.2,1.6) and (0,1.5) .. (0,1.4);

\draw [densely dotted, rotate=308.7] (-0.15,1.78) .. controls (0,1.9) .. (0.15,1.78);
\draw [rotate=308.7] (-0.15,1.78) .. controls (-0.2,1.6) and (0,1.5) .. (0,1.4);
\draw [rotate=308.7] (0.15,1.78) .. controls (0.2,1.6) and (0,1.5) .. (0,1.4);

\draw (0,1.4) .. controls (0,1.1) and (0.5,0.75) .. (1.1,0.88);

\draw [rotate=153.8] (0,1.4) .. controls (0,1.1) and (0.5,0.55) .. (1.1,0.88);

\draw [rotate=153.8] (0.23,1.67) .. controls (0,1.1) and (0.5,0.55) .. (1.1,0.88);

\draw [rotate=206] (0,1.4) .. controls (0,1.1) and (0.5,0.65) .. (1.2,1.2);

\draw (-1.11,0.89) .. controls (0,0.5) .. (1.1,0.88);

\draw [rotate=206] (-1.17,1.24) .. controls (-0.3,0.3) and (0.6,0.7) .. (1.12,0.87);

\draw (-1.11,0.89) .. controls (0,0) .. (1.4,-0.31);

\draw (-1.11,0.89) .. controls (0,0) and (1.4,-0.31) .. (1.7,-0.13);

\draw (1.4,-0.31) .. controls (0,-0.2) .. (-1.4,-0.31);

\draw (-0.63, 1.33) .. controls (-0.5,1) and (0.5,0.65) .. (1.1,0.88);
\draw (-1.4,-0.31) .. controls (-1.2,-0.25) and (-1.2,0.2) .. (-1.44, 0.3);
\draw[dashed] (-1.44, 0.3) .. controls (-1.26,1.04) .. (-0.63, 1.33);

\draw [rotate=206] (-0.62, 1.32) .. controls (-0.5,1) and (0.5,0.6) .. (1.12,0.87);
\draw [rotate=206] (-1.4,-0.31) .. controls (-1.2,-0.25) and (-1.2,0.2) .. (-1.44, 0.3);
\draw[dashed, rotate=206] (-1.43, 0.3) .. controls (-1.26,1.04) .. (-0.63, 1.31);
\end{tikzpicture}
  \caption{The action of $\rho$ followed by $\tau$ on $T$.}
  \label{figure: The action of rho followed by tau on T}

\end{figure}
\end{proof}

Denote by $A$ the matrix of the induced action of $\tau$ on the cone $V(T)$ in the basis $\{e_1, \dotsc, e_n\}$. Similarly, denote by $B$ the matrix of the induced action of $\tau \circ \rho$ on the cone $V(T)$ in the same basis. We show:

\begin{proposition}
\label{Proposition: matrices A and B}
The matrices $A$ and $B$ are as follows:

\begin{center}

$$
A=
 \left(\begin{array}{@{}ccccccccc@{}}
    2 & 1 & 0 & 0 & 1 & 0 & 1 & 2 & 0 \\
    1 & 1 & 0 & 0 & 0 & 1 & 0 & 1 & 0 \\
    0 & 0 & 1 & 0 & 0 & 0 & 0 & 0 & 0 \\
    0 & 0 & 0 & 1 & 0 & 0 & 0 & 0 & 0 \\
    0 & 0 & 0 & 0 & 1 & 0 & 0 & 0 & 0 \\
    0 & 0 & 0 & 0 & 0 & 1 & 0 & 0 & 0 \\
    0 & 0 & 0 & 0 & 0 & 0 & 1 & 0 & 0 \\
    0 & 0 & 0 & 0 & 0 & 0 & 0 & 1 & 0 \\
    0 & 0 & 0 & 0 & 0 & 0 & 0 & 0 & 1
  \end{array}\right) 
\,\,\,\,\,B=
 \left(\begin{array}{@{}cc
 ccccccc@{}}
    0 & 0 & 0 & 0 & 1 & 0 & 0 & 0 & 0 \\
    0 & 0 & 0 & 1 & 0 & 0 & 0 & 0 & 0 \\
    1 & 0 & 0 & 0 & 0 & 0 & 0 & 0 & 0 \\
    0 & 1 & 0 & 0 & 0 & 0 & 0 & 0 & 0 \\
    0 & 0 & 1 & 0 & 0 & 0 & 0 & 0 & 0 \\
    0 & 0 & 0 & 1 & 0 & 0 & 0 & 0 & 1 \\
    0 & 0 & 0 & 0 & 1 & 0 & 1 & 0 & 0 \\
    0 & 0 & 0 & 0 & 0 & 1 & 0 & 0 & 0 \\
    0 & 0 & 0 & 0 & 0 & 0 & 0 & 1 & 0
  \end{array}\right)
$$ 
Further, the vector $v = (\phi,1,0,0,0,0,0,0,0)^T$ is an eigenvector of $A$ with the eigenvalue $\phi^2$, where $\phi = \frac{1+\sqrt{5}}{2}$.
\end{center}

\end{proposition}

\begin{proof}
Let $w_i = 2e_i$. The matrices $A$ and $B$ do not change if expressed in the basis $\{w_1, \dotsc, w_9\}$. It is sufficient to find the images of the vectors $w_i$, $i = 1, \dotsc, 9$. We refer to \figref{figure: The action of tau and rho followed by tau on e_1}, \figref{figure: The action of tau and rho followed by tau on e_2}, \figref{figure: The action of tau and rho followed by tau on e_3}, \figref{figure: The action of tau and rho followed by tau on e_4}, \figref{figure: The action of tau and rho followed by tau on e_5}, \figref{figure: The action of tau and rho followed by tau on e_6}, \figref{figure: The action of tau and rho followed by tau on e_7},
\figref{figure: The action of tau and rho followed by tau on e_8},
\figref{figure: The action of tau and rho followed by tau on e_9} and leave the verification for the reader. Finally, the vector $v = (\phi,1,0,0,0,0,0,0,0)^T$ corresponds to the unstable lamination of $\tau$ on $S$.

\begin{figure}[H]
  \begin{tikzpicture}
\draw (-1.56,1.26) .. controls (-0.58, 1.23) .. (0,2);
\draw (0,2) .. controls (0.58, 1.22) .. (1.55,1.24);
\draw (1.55,1.24) .. controls (1.32, 0.31) .. (1.95,-0.44);
\draw (1.95,-0.44) .. controls (1.06, -0.83) .. (0.86,-1.78);
\draw (0.86,-1.78) .. controls (0, -1.33) .. (-0.87,-1.77);
\draw (-0.87,-1.77) .. controls (-1.06, -0.83) .. (-1.96,-0.44);
\draw (-1.96,-0.44) .. controls (-1.32, 0.31) .. (-1.56,1.26);
\draw (1.67, -0.1) .. controls (0.2,-0.3) .. (-0.96, -1.35)  ;
\draw (1.57, -0.61) .. controls (0.5,-0.75) .. (-0.50, -1.59);
\draw [densely dotted] (1.67, -0.1) .. controls (1.55,-0.35) .. (1.57, -0.61);
\draw [densely dotted] (-0.96, -1.35) .. controls (-0.7,-1.4) .. (-0.50, -1.59);

\end{tikzpicture}
    \hspace*{.5in}
  \begin{tikzpicture}
\draw (-1.56,1.26) .. controls (-0.58, 1.23) .. (0,2);
\draw (0,2) .. controls (0.58, 1.22) .. (1.55,1.24);
\draw (1.55,1.24) .. controls (1.32, 0.31) .. (1.95,-0.44);
\draw (1.95,-0.44) .. controls (1.06, -0.83) .. (0.86,-1.78);
\draw (0.86,-1.78) .. controls (0, -1.33) .. (-0.87,-1.77);

\draw (-0.87,-1.77) .. controls (-1.06, -0.83) .. (-1.96,-0.44);
\draw (-1.96,-0.44) .. controls (-1.32, 0.31) .. (-1.56,1.26);


\draw [dashed] (1.67, -0.1) .. controls (1.55,-0.35) .. (1.57, -0.61);

\draw [densely dotted][rotate=52] (-0.96, -1.35) .. controls (-0.7,-1.4) .. (-0.50, -1.59);

\draw (0.96, -1.35) .. controls (0.1,-1) .. (-0.42, -1.55)  ;

\draw (-0.2, -1.47) .. controls (0.1,-1.25) .. (0.50, -1.59)  ;

\draw [densely dotted] (-0.2, -1.47) .. controls (-0.6,-1.2) .. (-1.05,-1.07);

\draw [densely dotted] (-0.98, -1.28) .. controls (-0.65,-1.35) .. (-0.41, -1.55);

\draw (-1.05,-1.07) .. controls (0.2,-0.2) .. (1.53,0.1);

\draw (-0.98, -1.28) .. controls (0.35,-0.35) .. (1.66,-0.09);

\draw [dashed] (1.53, 0.1) .. controls (1.35,-0.25) .. (1.37, -0.72);

\draw (1.37, -0.72) .. controls (-0.05,-0.9) .. (-0.7,-1.68);

\draw (1.57, -0.61) .. controls (0,-0.7) .. (-0.91,-1.57);

\draw [densely dotted] (-0.7,-1.68) .. controls (-0.76,-1.56) .. (-0.91,-1.57);

\end{tikzpicture}
   \hspace*{.5in}
  \begin{tikzpicture}[rotate=154.3]
\draw (-1.56,1.26) .. controls (-0.58, 1.23) .. (0,2);
\draw (0,2) .. controls (0.58, 1.22) .. (1.55,1.24);
\draw (1.55,1.24) .. controls (1.32, 0.31) .. (1.95,-0.44);
\draw (1.95,-0.44) .. controls (1.06, -0.83) .. (0.86,-1.78);
\draw (0.86,-1.78) .. controls (0, -1.33) .. (-0.87,-1.77);
\draw (-0.87,-1.77) .. controls (-1.06, -0.83) .. (-1.96,-0.44);
\draw (-1.96,-0.44) .. controls (-1.32, 0.31) .. (-1.56,1.26);
\draw (1.67, -0.1) .. controls (0.2,-0.3) .. (-0.96, -1.35)  ;
\draw (1.57, -0.61) .. controls (0.5,-0.75) .. (-0.50, -1.59);
\draw [densely dotted] (1.67, -0.1) .. controls (1.55,-0.35) .. (1.57, -0.61);
\draw [densely dotted] (-0.96, -1.35) .. controls (-0.7,-1.4) .. (-0.50, -1.59);

\end{tikzpicture}
  \caption{The curve corresponding to the vector $w_1$ and its images under $\tau$ and $\tau \circ \rho$, respectively.}
  \label{figure: The action of tau and rho followed by tau on e_1}
\end{figure}

\begin{figure}[H]
  \begin{tikzpicture}
\draw (-1.56,1.26) .. controls (-0.58, 1.23) .. (0,2);
\draw (0,2) .. controls (0.58, 1.22) .. (1.55,1.24);
\draw (1.55,1.24) .. controls (1.32, 0.31) .. (1.95,-0.44);
\draw (1.95,-0.44) .. controls (1.06, -0.83) .. (0.86,-1.78);
\draw (0.86,-1.78) .. controls (0, -1.33) .. (-0.87,-1.77);
\draw (-0.87,-1.77) .. controls (-1.06, -0.83) .. (-1.96,-0.44);
\draw (-1.96,-0.44) .. controls (-1.32, 0.31) .. (-1.56,1.26);

\draw [densely dotted] (-0.96, -1.35) .. controls (-0.7,-1.4) .. (-0.50, -1.59);

\draw [densely dotted][rotate=52] (-0.96, -1.35) .. controls (-0.7,-1.4) .. (-0.50, -1.59);

\draw (0.96, -1.35) .. controls (0,-0.8) .. (-0.96, -1.35)  ;

\draw (-0.50, -1.59) .. controls (0,-1.25) .. (0.50, -1.59)  ;

\end{tikzpicture}
    \hspace*{.5in}
  \begin{tikzpicture}
\draw (-1.56,1.26) .. controls (-0.58, 1.23) .. (0,2);
\draw (0,2) .. controls (0.58, 1.22) .. (1.55,1.24);
\draw (1.55,1.24) .. controls (1.32, 0.31) .. (1.95,-0.44);
\draw (1.95,-0.44) .. controls (1.06, -0.83) .. (0.86,-1.78);
\draw (0.86,-1.78) .. controls (0, -1.33) .. (-0.87,-1.77);
\draw (-0.87,-1.77) .. controls (-1.06, -0.83) .. (-1.96,-0.44);
\draw (-1.96,-0.44) .. controls (-1.32, 0.31) .. (-1.56,1.26);

\draw (1.67, -0.1) .. controls (0.2,-0.5) .. (-1.02,-1.13)  ;

\draw (1.57, -0.61) .. controls (0.5,-0.6) .. (-0.95,-1.4);

\draw [densely dotted] (1.67, -0.1) .. controls (1.55,-0.35) .. (1.57, -0.61);

\draw[densely dotted] (-0.5, -1.59) .. controls(-0.65,-1.5) .. (-0.95,-1.4);

\draw[densely dotted] (-1.02,-1.13) .. controls(-0.55,-1.3) .. (-0.2, -1.47);

\draw [densely dotted][rotate=52] (-0.96, -1.35) .. controls (-0.7,-1.4) .. (-0.50, -1.59);

\draw (0.96, -1.35) .. controls (0.1,-1) .. (-0.5, -1.59)  ;

\draw (-0.2, -1.47) .. controls (0.1,-1.25) .. (0.50, -1.59)  ;

\end{tikzpicture}
   \hspace*{.5in}
  \begin{tikzpicture}[rotate=154.3]
\draw (-1.56,1.26) .. controls (-0.58, 1.23) .. (0,2);
\draw (0,2) .. controls (0.58, 1.22) .. (1.55,1.24);
\draw (1.55,1.24) .. controls (1.32, 0.31) .. (1.95,-0.44);
\draw (1.95,-0.44) .. controls (1.06, -0.83) .. (0.86,-1.78);
\draw (0.86,-1.78) .. controls (0, -1.33) .. (-0.87,-1.77);
\draw (-0.87,-1.77) .. controls (-1.06, -0.83) .. (-1.96,-0.44);
\draw (-1.96,-0.44) .. controls (-1.32, 0.31) .. (-1.56,1.26);

\draw [densely dotted] (-0.96, -1.35) .. controls (-0.7,-1.4) .. (-0.50, -1.59);

\draw [densely dotted][rotate=52] (-0.96, -1.35) .. controls (-0.7,-1.4) .. (-0.50, -1.59);

\draw (0.96, -1.35) .. controls (0,-0.8) .. (-0.96, -1.35)  ;

\draw (-0.50, -1.59) .. controls (0,-1.25) .. (0.50, -1.59)  ;

\end{tikzpicture}
  \caption{The curve corresponding to the vector $w_2$ and its images under $\tau$ and $\tau \circ \rho$, respectively.}
  \label{figure: The action of tau and rho followed by tau on e_2}
\end{figure}

\begin{figure}[H]
  \begin{tikzpicture}[rotate=154.3]
\draw (-1.56,1.26) .. controls (-0.58, 1.23) .. (0,2);
\draw (0,2) .. controls (0.58, 1.22) .. (1.55,1.24);
\draw (1.55,1.24) .. controls (1.32, 0.31) .. (1.95,-0.44);
\draw (1.95,-0.44) .. controls (1.06, -0.83) .. (0.86,-1.78);
\draw (0.86,-1.78) .. controls (0, -1.33) .. (-0.87,-1.77);
\draw (-0.87,-1.77) .. controls (-1.06, -0.83) .. (-1.96,-0.44);
\draw (-1.96,-0.44) .. controls (-1.32, 0.31) .. (-1.56,1.26);
\draw (1.67, -0.1) .. controls (0.2,-0.3) .. (-0.96, -1.35)  ;
\draw (1.57, -0.61) .. controls (0.5,-0.75) .. (-0.50, -1.59);
\draw [densely dotted] (1.67, -0.1) .. controls (1.55,-0.35) .. (1.57, -0.61);
\draw [densely dotted] (-0.96, -1.35) .. controls (-0.7,-1.4) .. (-0.50, -1.59);

\end{tikzpicture}
    \hspace*{.5in}
  \begin{tikzpicture}[rotate=154.3]
\draw (-1.56,1.26) .. controls (-0.58, 1.23) .. (0,2);
\draw (0,2) .. controls (0.58, 1.22) .. (1.55,1.24);
\draw (1.55,1.24) .. controls (1.32, 0.31) .. (1.95,-0.44);
\draw (1.95,-0.44) .. controls (1.06, -0.83) .. (0.86,-1.78);
\draw (0.86,-1.78) .. controls (0, -1.33) .. (-0.87,-1.77);
\draw (-0.87,-1.77) .. controls (-1.06, -0.83) .. (-1.96,-0.44);
\draw (-1.96,-0.44) .. controls (-1.32, 0.31) .. (-1.56,1.26);
\draw (1.67, -0.1) .. controls (0.2,-0.3) .. (-0.96, -1.35)  ;
\draw (1.57, -0.61) .. controls (0.5,-0.75) .. (-0.50, -1.59);
\draw [densely dotted] (1.67, -0.1) .. controls (1.55,-0.35) .. (1.57, -0.61);
\draw [densely dotted] (-0.96, -1.35) .. controls (-0.7,-1.4) .. (-0.50, -1.59);

\end{tikzpicture}
   \hspace*{.5in}
  \begin{tikzpicture}
\draw (-1.56,1.26) .. controls (-0.58, 1.23) .. (0,2);
\draw (0,2) .. controls (0.58, 1.22) .. (1.55,1.24);
\draw (1.55,1.24) .. controls (1.32, 0.31) .. (1.95,-0.44);
\draw (1.95,-0.44) .. controls (1.06, -0.83) .. (0.86,-1.78);
\draw (0.86,-1.78) .. controls (0, -1.33) .. (-0.87,-1.77);
\draw (-0.87,-1.77) .. controls (-1.06, -0.83) .. (-1.96,-0.44);
\draw (-1.96,-0.44) .. controls (-1.32, 0.31) .. (-1.56,1.26);

\draw [densely dotted] (1.67, -0.1) .. controls (1.55,-0.35) .. (1.57, -0.61);

\draw [densely dotted][rotate=204] (1.67, -0.1) .. controls (1.55,-0.35) .. (1.57, -0.61);

\draw (1.67, -0.1) .. controls (0,0.2) .. (-1.67, -0.1);

\draw (-1.57, -0.61) .. controls (0,-0.2) .. (1.57, -0.61);

\end{tikzpicture}
  \caption{The curve corresponding to the vector $w_3$ and its images under $\tau$ and $\tau \circ \rho$, respectively.}
  \label{figure: The action of tau and rho followed by tau on e_3}
\end{figure}

\begin{figure}[H]
  \begin{tikzpicture}[rotate=154.3]
\draw (-1.56,1.26) .. controls (-0.58, 1.23) .. (0,2);
\draw (0,2) .. controls (0.58, 1.22) .. (1.55,1.24);
\draw (1.55,1.24) .. controls (1.32, 0.31) .. (1.95,-0.44);
\draw (1.95,-0.44) .. controls (1.06, -0.83) .. (0.86,-1.78);
\draw (0.86,-1.78) .. controls (0, -1.33) .. (-0.87,-1.77);
\draw (-0.87,-1.77) .. controls (-1.06, -0.83) .. (-1.96,-0.44);
\draw (-1.96,-0.44) .. controls (-1.32, 0.31) .. (-1.56,1.26);

\draw [densely dotted] (-0.96, -1.35) .. controls (-0.7,-1.4) .. (-0.50, -1.59);

\draw [densely dotted][rotate=52] (-0.96, -1.35) .. controls (-0.7,-1.4) .. (-0.50, -1.59);

\draw (0.96, -1.35) .. controls (0,-0.8) .. (-0.96, -1.35)  ;

\draw (-0.50, -1.59) .. controls (0,-1.25) .. (0.50, -1.59)  ;

\end{tikzpicture}
    \hspace*{.5in}
  \begin{tikzpicture}[rotate=154.3]
\draw (-1.56,1.26) .. controls (-0.58, 1.23) .. (0,2);
\draw (0,2) .. controls (0.58, 1.22) .. (1.55,1.24);
\draw (1.55,1.24) .. controls (1.32, 0.31) .. (1.95,-0.44);
\draw (1.95,-0.44) .. controls (1.06, -0.83) .. (0.86,-1.78);
\draw (0.86,-1.78) .. controls (0, -1.33) .. (-0.87,-1.77);
\draw (-0.87,-1.77) .. controls (-1.06, -0.83) .. (-1.96,-0.44);
\draw (-1.96,-0.44) .. controls (-1.32, 0.31) .. (-1.56,1.26);

\draw [densely dotted] (-0.96, -1.35) .. controls (-0.7,-1.4) .. (-0.50, -1.59);

\draw [densely dotted][rotate=52] (-0.96, -1.35) .. controls (-0.7,-1.4) .. (-0.50, -1.59);

\draw (0.96, -1.35) .. controls (0,-0.8) .. (-0.96, -1.35)  ;

\draw (-0.50, -1.59) .. controls (0,-1.25) .. (0.50, -1.59)  ;

\end{tikzpicture}
   \hspace*{.5in}
  \begin{tikzpicture}
\draw (-1.56,1.26) .. controls (-0.58, 1.23) .. (0,2);
\draw (0,2) .. controls (0.58, 1.22) .. (1.55,1.24);
\draw (1.55,1.24) .. controls (1.32, 0.31) .. (1.95,-0.44);
\draw (1.95,-0.44) .. controls (1.06, -0.83) .. (0.86,-1.78);
\draw (0.86,-1.78) .. controls (0, -1.33) .. (-0.87,-1.77);
\draw (-0.87,-1.77) .. controls (-1.06, -0.83) .. (-1.96,-0.44);
\draw (-1.96,-0.44) .. controls (-1.32, 0.31) .. (-1.56,1.26);

\draw [densely dotted][rotate=52] (-0.96, -1.35) .. controls (-0.7,-1.4) .. (-0.50, -1.59);

\draw (0.96, -1.35) .. controls (0.1,-1) .. (-0.5, -1.59)  ;

\draw (-0.2, -1.47) .. controls (0.1,-1.25) .. (0.50, -1.59)  ;

\draw [densely dotted][rotate=-51] (-0.96, -1.35) .. controls (-0.7,-1.4) .. (-0.50, -1.59);

\draw (-1.67,-0.1) .. controls (-0.9,-0.4) .. (-0.95,-1.4) ;

\draw (-1.55,-0.62) .. controls (-1.1,-0.65) .. (-1.02,-1.13) ;

\draw[dashed] (-0.5, -1.59) .. controls(-0.65,-1.5) .. (-0.95,-1.4);

\draw[dashed] (-1.02,-1.13) .. controls(-0.55,-1.3) .. (-0.2, -1.47);

\end{tikzpicture}
  \caption{The curve corresponding to the vector $w_4$ and its images under $\tau$ and $\tau \circ \rho$, respectively.}
  \label{figure: The action of tau and rho followed by tau on e_4}
\end{figure}

\begin{figure}[H]
  \begin{tikzpicture}
\draw (-1.56,1.26) .. controls (-0.58, 1.23) .. (0,2);
\draw (0,2) .. controls (0.58, 1.22) .. (1.55,1.24);
\draw (1.55,1.24) .. controls (1.32, 0.31) .. (1.95,-0.44);
\draw (1.95,-0.44) .. controls (1.06, -0.83) .. (0.86,-1.78);
\draw (0.86,-1.78) .. controls (0, -1.33) .. (-0.87,-1.77);
\draw (-0.87,-1.77) .. controls (-1.06, -0.83) .. (-1.96,-0.44);
\draw (-1.96,-0.44) .. controls (-1.32, 0.31) .. (-1.56,1.26);

\draw [densely dotted] (1.67, -0.1) .. controls (1.55,-0.35) .. (1.57, -0.61);

\draw [densely dotted][rotate=204] (1.67, -0.1) .. controls (1.55,-0.35) .. (1.57, -0.61);

\draw (1.67, -0.1) .. controls (0,0.2) .. (-1.67, -0.1);

\draw (-1.57, -0.61) .. controls (0,-0.2) .. (1.57, -0.61);

\end{tikzpicture}
    \hspace*{.5in}
  \begin{tikzpicture}
\draw (-1.56,1.26) .. controls (-0.58, 1.23) .. (0,2);
\draw (0,2) .. controls (0.58, 1.22) .. (1.55,1.24);
\draw (1.55,1.24) .. controls (1.32, 0.31) .. (1.95,-0.44);
\draw (1.95,-0.44) .. controls (1.06, -0.83) .. (0.86,-1.78);
\draw (0.86,-1.78) .. controls (0, -1.33) .. (-0.87,-1.77);
\draw (-0.87,-1.77) .. controls (-1.06, -0.83) .. (-1.96,-0.44);
\draw (-1.96,-0.44) .. controls (-1.32, 0.31) .. (-1.56,1.26);

\draw (1.49, -0.65) .. controls (0.2,-0.6) .. (-0.96, -1.35)  ;

\draw (1.23, -0.81) .. controls (0.1,-0.9) .. (-0.50, -1.59);

\draw[dashed] (1.49, -0.65) .. controls (1.5,-0.3) .. (1.6, 0);

\draw[dashed] (1.23, -0.81) .. controls (1.25,-0.3) .. (1.47, 0.2);

\draw [densely dotted] (-0.96, -1.35) .. controls (-0.7,-1.4) .. (-0.50, -1.59);

\draw [densely dotted][rotate=204] (1.67, -0.1) .. controls (1.55,-0.35) .. (1.57, -0.61);

\draw (1.47, 0.2) .. controls (0,0) .. (-1.67, -0.1);

\draw (-1.57, -0.61) .. controls (0,-0.2) .. (1.6, 0);

\end{tikzpicture}
   \hspace*{.5in}
  \begin{tikzpicture}
\draw (-1.56,1.26) .. controls (-0.58, 1.23) .. (0,2);
\draw (0,2) .. controls (0.58, 1.22) .. (1.55,1.24);
\draw (1.55,1.24) .. controls (1.32, 0.31) .. (1.95,-0.44);
\draw (1.95,-0.44) .. controls (1.06, -0.83) .. (0.86,-1.78);
\draw (0.86,-1.78) .. controls (0, -1.33) .. (-0.87,-1.77);
\draw (-0.87,-1.77) .. controls (-1.06, -0.83) .. (-1.96,-0.44);
\draw (-1.96,-0.44) .. controls (-1.32, 0.31) .. (-1.56,1.26);

\draw [densely dotted][rotate=154.3] (1.67, -0.1) .. controls (1.55,-0.35) .. (1.57, -0.61);

\draw (1.42, 0.6) .. controls (0.2,0.5) .. (-1.17, 1.25);

\draw (1.48, 0.2) .. controls (0,0.1) .. (-1.45, 0.83);

\draw (1.29, -0.77) .. controls (0,-0.7) .. (-0.96, -1.35)  ;

\draw (1.03, -1.1) .. controls (0.2,-1.1) .. (-0.50, -1.59);

\draw [densely dotted] (-0.96, -1.35) .. controls (-0.7,-1.4) .. (-0.50, -1.59);

\draw[dashed] (1.03, -1.1) .. controls (0.9,-0.1) .. (1.42, 0.6);

\draw[dashed] (1.29, -0.77) .. controls (1.25,-0.2) .. (1.48, 0.2);
  
\end{tikzpicture}
  \caption{The curve corresponding to the vector $w_5$ and its images under $\tau$ and $\tau \circ \rho$, respectively.}
  \label{figure: The action of tau and rho followed by tau on e_5}
\end{figure}

\begin{figure}[H]
  \begin{tikzpicture}[rotate=-51.4]
\draw (-1.56,1.26) .. controls (-0.58, 1.23) .. (0,2);
\draw (0,2) .. controls (0.58, 1.22) .. (1.55,1.24);
\draw (1.55,1.24) .. controls (1.32, 0.31) .. (1.95,-0.44);
\draw (1.95,-0.44) .. controls (1.06, -0.83) .. (0.86,-1.78);
\draw (0.86,-1.78) .. controls (0, -1.33) .. (-0.87,-1.77);
\draw (-0.87,-1.77) .. controls (-1.06, -0.83) .. (-1.96,-0.44);
\draw (-1.96,-0.44) .. controls (-1.32, 0.31) .. (-1.56,1.26);

\draw [densely dotted] (-0.96, -1.35) .. controls (-0.7,-1.4) .. (-0.50, -1.59);

\draw [densely dotted][rotate=52] (-0.96, -1.35) .. controls (-0.7,-1.4) .. (-0.50, -1.59);

\draw (0.96, -1.35) .. controls (0,-0.8) .. (-0.96, -1.35)  ;

\draw (-0.50, -1.59) .. controls (0,-1.25) .. (0.50, -1.59)  ;

\end{tikzpicture}
    \hspace*{.5in}
  \begin{tikzpicture}
\draw (-1.56,1.26) .. controls (-0.58, 1.23) .. (0,2);
\draw (0,2) .. controls (0.58, 1.22) .. (1.55,1.24);
\draw (1.55,1.24) .. controls (1.32, 0.31) .. (1.95,-0.44);
\draw (1.95,-0.44) .. controls (1.06, -0.83) .. (0.86,-1.78);
\draw (0.86,-1.78) .. controls (0, -1.33) .. (-0.87,-1.77);
\draw (-0.87,-1.77) .. controls (-1.06, -0.83) .. (-1.96,-0.44);
\draw (-1.96,-0.44) .. controls (-1.32, 0.31) .. (-1.56,1.26);

\draw [densely dotted][rotate=52] (-0.96, -1.35) .. controls (-0.7,-1.4) .. (-0.50, -1.59);

\draw (0.96, -1.35) .. controls (0.1,-1) .. (-0.5, -1.59)  ;

\draw (-0.2, -1.47) .. controls (0.1,-1.25) .. (0.50, -1.59)  ;

\draw [densely dotted][rotate=-51] (-0.96, -1.35) .. controls (-0.7,-1.4) .. (-0.50, -1.59);

\draw (-1.67,-0.1) .. controls (-0.9,-0.4) .. (-0.95,-1.4) ;

\draw (-1.55,-0.62) .. controls (-1.1,-0.65) .. (-1.02,-1.13) ;

\draw[dashed] (-0.5, -1.59) .. controls(-0.65,-1.5) .. (-0.95,-1.4);

\draw[dashed] (-1.02,-1.13) .. controls(-0.55,-1.3) .. (-0.2, -1.47);

\end{tikzpicture}
   \hspace*{.5in}
  \begin{tikzpicture}
\draw (-1.56,1.26) .. controls (-0.58, 1.23) .. (0,2);
\draw (0,2) .. controls (0.58, 1.22) .. (1.55,1.24);
\draw (1.55,1.24) .. controls (1.32, 0.31) .. (1.95,-0.44);
\draw (1.95,-0.44) .. controls (1.06, -0.83) .. (0.86,-1.78);
\draw (0.86,-1.78) .. controls (0, -1.33) .. (-0.87,-1.77);
\draw (-0.87,-1.77) .. controls (-1.06, -0.83) .. (-1.96,-0.44);
\draw (-1.96,-0.44) .. controls (-1.32, 0.31) .. (-1.56,1.26);

\draw [densely dotted] (-0.96, -1.35) .. controls (-0.7,-1.4) .. (-0.50, -1.59);

\draw [densely dotted][rotate=154.3] (-0.96, -1.35) .. controls (-0.7,-1.4) .. (-0.50, -1.59);

\draw (1.4, -0.7) .. controls (0,-0.8) .. (-0.96, -1.35)  ;

\draw (-0.50, -1.59) .. controls (0,-1.25) .. (1.03, -1.1)  ;

\draw[dashed] (1.4, -0.7) .. controls (1.4,-0.4) .. (1.55,0)  ;

\draw[dashed] (1.03, -1.1) .. controls  (1,-0.3) .. (1.41,0.35)  ;

\draw (1.55,0) .. controls (0.7,0.5) .. (1.15,1.24);

\draw (1.42,0.36) .. controls (1.25,0.5) .. (1.45,0.83);

\end{tikzpicture}
  \caption{The curve corresponding to the vector $w_6$ and its images under $\tau$ and $\tau \circ \rho$, respectively.}
  \label{figure: The action of tau and rho followed by tau on e_6}
\end{figure}

\begin{figure}[H]
  \begin{tikzpicture}
\draw (-1.56,1.26) .. controls (-0.58, 1.23) .. (0,2);
\draw (0,2) .. controls (0.58, 1.22) .. (1.55,1.24);
\draw (1.55,1.24) .. controls (1.32, 0.31) .. (1.95,-0.44);
\draw (1.95,-0.44) .. controls (1.06, -0.83) .. (0.86,-1.78);
\draw (0.86,-1.78) .. controls (0, -1.33) .. (-0.87,-1.77);
\draw (-0.87,-1.77) .. controls (-1.06, -0.83) .. (-1.96,-0.44);
\draw (-1.96,-0.44) .. controls (-1.32, 0.31) .. (-1.56,1.26);

\draw [densely dotted] (1.67, -0.1) .. controls (1.55,-0.35) .. (1.57, -0.61);

\draw [densely dotted][rotate=154.3] (1.67, -0.1) .. controls (1.55,-0.35) .. (1.57, -0.61);

\draw (1.67, -0.1) .. controls (0.2,0.2) .. (-1.17, 1.25);

\draw (1.57, -0.61) .. controls (0,-0.2) .. (-1.45, 0.83);

\end{tikzpicture}
    \hspace*{.5in}
  \begin{tikzpicture}
\draw (-1.56,1.26) .. controls (-0.58, 1.23) .. (0,2);
\draw (0,2) .. controls (0.58, 1.22) .. (1.55,1.24);
\draw (1.55,1.24) .. controls (1.32, 0.31) .. (1.95,-0.44);
\draw (1.95,-0.44) .. controls (1.06, -0.83) .. (0.86,-1.78);
\draw (0.86,-1.78) .. controls (0, -1.33) .. (-0.87,-1.77);
\draw (-0.87,-1.77) .. controls (-1.06, -0.83) .. (-1.96,-0.44);
\draw (-1.96,-0.44) .. controls (-1.32, 0.31) .. (-1.56,1.26);

\draw [densely dotted][rotate=154.3] (1.67, -0.1) .. controls (1.55,-0.35) .. (1.57, -0.61);

\draw (1.42, 0.6) .. controls (0.2,0.5) .. (-1.17, 1.25);

\draw (1.48, 0.2) .. controls (0,0.1) .. (-1.45, 0.83);

\draw (1.29, -0.77) .. controls (0,-0.7) .. (-0.96, -1.35)  ;

\draw (1.03, -1.1) .. controls (0.2,-1.1) .. (-0.50, -1.59);

\draw [densely dotted] (-0.96, -1.35) .. controls (-0.7,-1.4) .. (-0.50, -1.59);

\draw[dashed] (1.03, -1.1) .. controls (0.9,-0.1) .. (1.42, 0.6);

\draw[dashed] (1.29, -0.77) .. controls (1.25,-0.2) .. (1.48, 0.2);
  
\end{tikzpicture}
   \hspace*{.5in}
  \begin{tikzpicture}
\draw (-1.56,1.26) .. controls (-0.58, 1.23) .. (0,2);
\draw (0,2) .. controls (0.58, 1.22) .. (1.55,1.24);
\draw (1.55,1.24) .. controls (1.32, 0.31) .. (1.95,-0.44);
\draw (1.95,-0.44) .. controls (1.06, -0.83) .. (0.86,-1.78);
\draw (0.86,-1.78) .. controls (0, -1.33) .. (-0.87,-1.77);
\draw (-0.87,-1.77) .. controls (-1.06, -0.83) .. (-1.96,-0.44);
\draw (-1.96,-0.44) .. controls (-1.32, 0.31) .. (-1.56,1.26);

\draw [densely dotted] (1.67, -0.1) .. controls (1.55,-0.35) .. (1.57, -0.61);

\draw [densely dotted][rotate=154.3] (1.67, -0.1) .. controls (1.55,-0.35) .. (1.57, -0.61);

\draw (1.67, -0.1) .. controls (0.2,0.2) .. (-1.17, 1.25);

\draw (1.57, -0.61) .. controls (0,-0.2) .. (-1.45, 0.83);

\end{tikzpicture}
  \caption{The curve corresponding to the vector $w_7$ and its images under $\tau$ and $\tau \circ \rho$, respectively.}
  \label{figure: The action of tau and rho followed by tau on e_7}
\end{figure}

\begin{figure}[H]
  \begin{tikzpicture}
\draw (-1.56,1.26) .. controls (-0.58, 1.23) .. (0,2);
\draw (0,2) .. controls (0.58, 1.22) .. (1.55,1.24);
\draw (1.55,1.24) .. controls (1.32, 0.31) .. (1.95,-0.44);
\draw (1.95,-0.44) .. controls (1.06, -0.83) .. (0.86,-1.78);
\draw (0.86,-1.78) .. controls (0, -1.33) .. (-0.87,-1.77);
\draw (-0.87,-1.77) .. controls (-1.06, -0.83) .. (-1.96,-0.44);
\draw (-1.96,-0.44) .. controls (-1.32, 0.31) .. (-1.56,1.26);

\draw [densely dotted] (-0.96, -1.35) .. controls (-0.7,-1.4) .. (-0.50, -1.59);

\draw [densely dotted][rotate=154.3] (-0.96, -1.35) .. controls (-0.7,-1.4) .. (-0.50, -1.59);

\draw (1.4, -0.7) .. controls (0,-0.8) .. (-0.96, -1.35)  ;

\draw (-0.50, -1.59) .. controls (0,-1.25) .. (1.03, -1.1)  ;

\draw[dashed] (1.4, -0.7) .. controls (1.4,-0.4) .. (1.55,0)  ;

\draw[dashed] (1.03, -1.1) .. controls  (1,-0.3) .. (1.41,0.35)  ;

\draw (1.55,0) .. controls (0.7,0.5) .. (1.15,1.24);

\draw (1.42,0.36) .. controls (1.25,0.5) .. (1.45,0.83);

\end{tikzpicture}
    \hspace*{.5in}
  \begin{tikzpicture}
\draw (-1.56,1.26) .. controls (-0.58, 1.23) .. (0,2);
\draw (0,2) .. controls (0.58, 1.22) .. (1.55,1.24);
\draw (1.55,1.24) .. controls (1.32, 0.31) .. (1.95,-0.44);
\draw (1.95,-0.44) .. controls (1.06, -0.83) .. (0.86,-1.78);
\draw (0.86,-1.78) .. controls (0, -1.33) .. (-0.87,-1.77);
\draw (-0.87,-1.77) .. controls (-1.06, -0.83) .. (-1.96,-0.44);
\draw (-1.96,-0.44) .. controls (-1.32, 0.31) .. (-1.56,1.26);

\draw [densely dotted][rotate=52] (-0.96, -1.35) .. controls (-0.7,-1.4) .. (-0.50, -1.59);

\draw [dashed] (1.34, -0.73) .. controls (1.26,-0.25) .. (1.53, 0.1)  ;

\draw[dashed] (1.18, -0.86) .. controls  (1,-0.3) .. (1.41,0.35)  ;

\draw  (1.34, -0.73) .. controls  (0,-0.7) .. (-0.6,-1.64)  ;
 
\draw [densely dotted] (-0.6,-1.64) .. controls  (-0.73,-1.52) .. (-0.92,-1.52)  ;

\draw [densely dotted] (-0.46,-1.57) .. controls  (-0.7,-1.42) .. (-0.95,-1.42)  ;

\draw [densely dotted] (-0.3, -1.5) .. controls  (-0.64,-1.3) .. (-0.97,-1.3)  ;

\draw [densely dotted] (-0.1, -1.45) .. controls  (-0.54,-1.18) .. (-1.02,-1.15)  ;

\draw (-1.02,-1.15) .. controls (0,-0.4) .. (1.67,-0.1) ; 

\draw (-0.97,-1.3) .. controls (0,-0.45) .. (1.77,-0.22) ; 

\draw (-0.95,-1.42) .. controls (0,-0.5) .. (1.7,-0.55) ; 

\draw (-0.92,-1.52) .. controls (0,-0.6) .. (1.56,-0.61) ; 

\draw [densely dotted] (1.77,-0.22) .. controls (1.7,-0.35) .. (1.7,-0.55) ;

\draw [densely dotted] (1.56,-0.61) .. controls (1.55,-0.35) .. (1.67,-0.1) ;

\draw (1.18, -0.86) .. controls  (0,-0.9) .. (-0.46,-1.57)   ;

\draw (1.53,0.1) .. controls (0.7,0.5) .. (1.15,1.24);

\draw (1.42,0.36) .. controls (1.25,0.5) .. (1.45,0.83);

\draw [densely dotted][rotate=154.3] (-0.96, -1.35) .. controls (-0.7,-1.4) .. (-0.50, -1.59);

\draw (0.96, -1.35) .. controls (0.1,-1) .. (-0.3, -1.5)  ;

\draw (-0.1, -1.45) .. controls (0.1,-1.25) .. (0.50, -1.59)  ;

\end{tikzpicture}
   \hspace*{.5in}
  \begin{tikzpicture}[rotate=154.3]
\draw (-1.56,1.26) .. controls (-0.58, 1.23) .. (0,2);
\draw (0,2) .. controls (0.58, 1.22) .. (1.55,1.24);
\draw (1.55,1.24) .. controls (1.32, 0.31) .. (1.95,-0.44);
\draw (1.95,-0.44) .. controls (1.06, -0.83) .. (0.86,-1.78);
\draw (0.86,-1.78) .. controls (0, -1.33) .. (-0.87,-1.77);
\draw (-0.87,-1.77) .. controls (-1.06, -0.83) .. (-1.96,-0.44);
\draw (-1.96,-0.44) .. controls (-1.32, 0.31) .. (-1.56,1.26);

\draw [densely dotted] (-0.96, -1.35) .. controls (-0.7,-1.4) .. (-0.50, -1.59);

\draw [densely dotted][rotate=154.3] (-0.96, -1.35) .. controls (-0.7,-1.4) .. (-0.50, -1.59);

\draw (1.4, -0.7) .. controls (0,-0.8) .. (-0.96, -1.35)  ;

\draw (-0.50, -1.59) .. controls (0,-1.25) .. (1.03, -1.1)  ;

\draw[dashed] (1.4, -0.7) .. controls (1.4,-0.4) .. (1.55,0)  ;

\draw[dashed] (1.03, -1.1) .. controls  (1,-0.3) .. (1.41,0.35)  ;

\draw (1.55,0) .. controls (0.7,0.5) .. (1.15,1.24);

\draw (1.42,0.36) .. controls (1.25,0.5) .. (1.45,0.83);

\end{tikzpicture}
  \caption{The curve corresponding to the vector $w_8$ and its images under $\tau$ and $\tau \circ \rho$, respectively.}
  \label{figure: The action of tau and rho followed by tau on e_8}
\end{figure}

\begin{figure}[H]
  \begin{tikzpicture}[rotate=154.3]
\draw (-1.56,1.26) .. controls (-0.58, 1.23) .. (0,2);
\draw (0,2) .. controls (0.58, 1.22) .. (1.55,1.24);
\draw (1.55,1.24) .. controls (1.32, 0.31) .. (1.95,-0.44);
\draw (1.95,-0.44) .. controls (1.06, -0.83) .. (0.86,-1.78);
\draw (0.86,-1.78) .. controls (0, -1.33) .. (-0.87,-1.77);
\draw (-0.87,-1.77) .. controls (-1.06, -0.83) .. (-1.96,-0.44);
\draw (-1.96,-0.44) .. controls (-1.32, 0.31) .. (-1.56,1.26);

\draw [densely dotted] (-0.96, -1.35) .. controls (-0.7,-1.4) .. (-0.50, -1.59);

\draw [densely dotted][rotate=154.3] (-0.96, -1.35) .. controls (-0.7,-1.4) .. (-0.50, -1.59);

\draw (1.4, -0.7) .. controls (0,-0.8) .. (-0.96, -1.35)  ;

\draw (-0.50, -1.59) .. controls (0,-1.25) .. (1.03, -1.1)  ;

\draw[dashed] (1.4, -0.7) .. controls (1.4,-0.4) .. (1.55,0)  ;

\draw[dashed] (1.03, -1.1) .. controls  (1,-0.3) .. (1.41,0.35)  ;

\draw (1.55,0) .. controls (0.7,0.5) .. (1.15,1.24);

\draw (1.42,0.36) .. controls (1.25,0.5) .. (1.45,0.83);

\end{tikzpicture}
    \hspace*{.5in}
  \begin{tikzpicture}[rotate=154.3]
\draw (-1.56,1.26) .. controls (-0.58, 1.23) .. (0,2);
\draw (0,2) .. controls (0.58, 1.22) .. (1.55,1.24);
\draw (1.55,1.24) .. controls (1.32, 0.31) .. (1.95,-0.44);
\draw (1.95,-0.44) .. controls (1.06, -0.83) .. (0.86,-1.78);
\draw (0.86,-1.78) .. controls (0, -1.33) .. (-0.87,-1.77);
\draw (-0.87,-1.77) .. controls (-1.06, -0.83) .. (-1.96,-0.44);
\draw (-1.96,-0.44) .. controls (-1.32, 0.31) .. (-1.56,1.26);

\draw [densely dotted] (-0.96, -1.35) .. controls (-0.7,-1.4) .. (-0.50, -1.59);

\draw [densely dotted][rotate=154.3] (-0.96, -1.35) .. controls (-0.7,-1.4) .. (-0.50, -1.59);

\draw (1.4, -0.7) .. controls (0,-0.8) .. (-0.96, -1.35)  ;

\draw (-0.50, -1.59) .. controls (0,-1.25) .. (1.03, -1.1)  ;

\draw[dashed] (1.4, -0.7) .. controls (1.4,-0.4) .. (1.55,0)  ;

\draw[dashed] (1.03, -1.1) .. controls  (1,-0.3) .. (1.41,0.35)  ;

\draw (1.55,0) .. controls (0.7,0.5) .. (1.15,1.24);

\draw (1.42,0.36) .. controls (1.25,0.5) .. (1.45,0.83);

\end{tikzpicture}
   \hspace*{.5in}
  \begin{tikzpicture}[rotate=-51.4]
\draw (-1.56,1.26) .. controls (-0.58, 1.23) .. (0,2);
\draw (0,2) .. controls (0.58, 1.22) .. (1.55,1.24);
\draw (1.55,1.24) .. controls (1.32, 0.31) .. (1.95,-0.44);
\draw (1.95,-0.44) .. controls (1.06, -0.83) .. (0.86,-1.78);
\draw (0.86,-1.78) .. controls (0, -1.33) .. (-0.87,-1.77);
\draw (-0.87,-1.77) .. controls (-1.06, -0.83) .. (-1.96,-0.44);
\draw (-1.96,-0.44) .. controls (-1.32, 0.31) .. (-1.56,1.26);

\draw [densely dotted] (-0.96, -1.35) .. controls (-0.7,-1.4) .. (-0.50, -1.59);

\draw [densely dotted][rotate=52] (-0.96, -1.35) .. controls (-0.7,-1.4) .. (-0.50, -1.59);

\draw (0.96, -1.35) .. controls (0,-0.8) .. (-0.96, -1.35)  ;

\draw (-0.50, -1.59) .. controls (0,-1.25) .. (0.50, -1.59)  ;

\end{tikzpicture}
  \caption{The curve corresponding to the vector $w_9$ and its images under $\tau$ and $\tau \circ \rho$, respectively.}
  \label{figure: The action of tau and rho followed by tau on e_9}
\end{figure}
\end{proof}

\section{Estimating the intersection numbers}
\label{section: computing intersection numbers}

Let $\gamma_0$ be the multicurve on $S$ that corresponds to the vector $w_1+w_3 \in V(T)$ as in \figref{figure: gamma0_new}. Define the multicurves $\gamma_i = \Phi_i(\gamma_0)$. In this subsection we will coarsely estimate the intersection numbers between pairs of multicurves in the sequence $\{\gamma_i\}$. To state the result we introduce some notation.


\begin{figure}[H]
    \centering
    \begin{tikzpicture}
\draw (-1.56,1.26) .. controls (-0.58, 1.23) .. (0,2);
\draw (0,2) .. controls (0.58, 1.22) .. (1.55,1.24);
\draw (1.55,1.24) .. controls (1.32, 0.31) .. (1.95,-0.44);
\draw (1.95,-0.44) .. controls (1.06, -0.83) .. (0.86,-1.78);
\draw (0.86,-1.78) .. controls (0, -1.33) .. (-0.87,-1.77);
\draw (-0.87,-1.77) .. controls (-1.06, -0.83) .. (-1.96,-0.44);
\draw (-1.96,-0.44) .. controls (-1.32, 0.31) .. (-1.56,1.26);
\draw (1.67, -0.1) .. controls (0.2,-0.3) .. (-0.96, -1.35) node[midway,above] {\Large $\gamma_0$} ;
\draw (1.57, -0.61) .. controls (0.5,-0.75) .. (-0.50, -1.59);
\draw [densely dotted] (1.67, -0.1) .. controls (1.55,-0.35) .. (1.57, -0.61);
\draw [densely dotted] (-0.96, -1.35) .. controls (-0.7,-1.4) .. (-0.50, -1.59);

\draw [rotate=154.3] (1.67, -0.1) .. controls (0.2,-0.3) .. (-0.96, -1.35);
\draw [rotate=154.3] (1.57, -0.61) .. controls (0.5,-0.75) .. (-0.50, -1.59);
\draw [rotate=154.3] [densely dotted] (1.67, -0.1) .. controls (1.55,-0.35) .. (1.57, -0.61);
\draw [rotate=154.3] [densely dotted] (-0.96, -1.35) .. controls (-0.7,-1.4) .. (-0.50, -1.59);

\end{tikzpicture}
    \caption{The multicurve $\gamma_0$ on $S$.}
    \label{figure: gamma0_new}
\end{figure}

Let $f_0=0, f_1=1, f_{n} = f_{n-1} + f_{n-2}$ for $n \geqslant 2$ be the Fibonacci sequence. Define the numbers $c_i = 2f_{2r_i-2}$ for $i \geqslant 1$. We assume that the sequence $\{r_n\}$ satisfies
\begin{equation}
    \label{eq: condition on r_n}
    c_1 \neq 0, \,\,\,\,\, \sum_{i=1}^{\infty} \frac{c_{i}}{c_{i+1}} < \infty.
\end{equation}
We prove:

\begin{proposition}
\label{Prop: Intersection numbers}

There is a constant $i_2 \in \NN$ such that for $i_2 \leqslant i < j$ with odd $j-i$, the following holds: 
$$
i(\gamma_{i-1},\gamma_{j}) \, \emul \, i(\gamma_{i},\gamma_{j}) \, \emul \, c_{i+1}c_{i+3} \dotsc c_{j}.
$$
The multiplicative constants are independent of $i$ and $j$.
\end{proposition}

To prove this proposition we will study the asymptotic behavior of the matrix products involving matrices $A$ and $B$ from \propref{Proposition: matrices A and B}. We start with elementary observations about Fibonacci sequence.


\begin{claim}
\label{claim: basic fibonacci 1}
For $m \geqslant 0$, the following holds:
$2f_{m+1}+f_m = f_{m+3}, \,\, 2f_{m+1}+f_{m}+f_{m+2} = f_{m+4}$.
\end{claim}
\begin{proof}
We have:
$$
2f_{m+1}+f_m = f_{m+1}+f_{m+1}+f_m=f_{m+1}+f_{m+2}=f_{m+3}.
$$
$$
2f_{m+1}+f_{m}+f_{m+2} = f_{m+3}+f_{m+2} = f_{m+4}.
$$
\end{proof}
Let $\phi = \frac{1+\sqrt{5}}{2}$ be the golden ratio. 
\begin{claim}
\label{claim: basic fibonacci 2}
For $m \geqslant 1$, the following holds:
$$
\phi^{-2}f_{2m} - \phi^{-2m} = f_{2m-2}, \,\,\,\, \phi^{-1} f_{2m} - \phi^{-2m} = f_{2m-1}, \,\,\,\, \phi f_{2m} + \phi^{-2m} = f_{2m+1}, \,\,\,\, \phi^2 f_{2m} + \phi^{-2m} = f_{2m+2}.
$$
\end{claim}
\begin{proof}
By Binet's formula, we have $f_{2m} = \frac{\phi^{2m} - \psi^{2m}}{\sqrt{5}}$, where $\psi = \frac{1-\sqrt{5}}{2}$. Since $\psi^2 = \phi^{-2}$, we have $f_{2m} = \frac{\phi^{2m} - \phi^{-2m}}{\sqrt{5}}$. Since $\phi^2-\phi^{-2}=\phi-\phi^{-1}=\sqrt{5}$, we have

$$
\phi^{-2}f_{2m} - \phi^{-2m} = \frac{\phi^{2m-2} - \phi^{-2m-2}-\phi^{-2m}\sqrt{5}}{\sqrt{5}} = \frac{\phi^{2m-2} -\phi^{-2m}(\phi^{-2}+\sqrt{5})}{\sqrt{5}}  = \frac{\phi^{2m-2} - \phi^{-2m+2}}{\sqrt{5}} = f_{2m-2}.
$$

$$
\phi^{-1} f_{2m} - \phi^{-2m} =  \frac{\phi^{2m-1} - \phi^{-2m-1}-\phi^{-2m}\sqrt{5}}{\sqrt{5}} = \frac{\phi^{2m-1} -\phi^{-2m}(\phi^{-1}+\sqrt{5})}{\sqrt{5}}  = \frac{\phi^{2m-1} - \phi^{-2m+1}}{\sqrt{5}} = f_{2m-1}.
$$

$$
\phi f_{2m} + \phi^{-2m} = \frac{\phi^{2m+1} - \phi^{-2m+1}+\phi^{-2m}\sqrt{5}}{\sqrt{5}} = \frac{\phi^{2m+1} -\phi^{-2m}(\phi-\sqrt{5})}{\sqrt{5}}  = \frac{\phi^{2m+1} - \phi^{-2m-1}}{\sqrt{5}}  = f_{2m+1}.
$$

$$
\phi^2 f_{2m} + \phi^{-2m} = \frac{\phi^{2m+2} - \phi^{-2m+2}+\phi^{-2m}\sqrt{5}}{\sqrt{5}} = \frac{\phi^{2m+2} -\phi^{-2m}(\phi^2-\sqrt{5})}{\sqrt{5}}  = \frac{\phi^{2m+2} - \phi^{-2m-2}}{\sqrt{5}}  = f_{2m+2}.
$$
\end{proof}


Next, we show: 

\begin{claim}
\label{claim: matrix A^n}
For $n \geqslant 1$, the matrix $A^n$ is as follows: 

$$
A^n=
 \left(\begin{array}{@{}ccccccccc@{}}
    f_{2n+1} & f_{2n} & 0 & 0 & f_{2n} & f_{2n-1}-1 & f_{2n} & f_{2n+2}-1 & 0 \\
   f_{2n} & f_{2n-1} & 0 & 0 & f_{2n-1}-1 & f_{2n-2}+1 & f_{2n-1}-1 & f_{2n+1}-1 & 0 \\
    0 & 0 & 1 & 0 & 0 & 0 & 0 & 0 & 0 \\
    0 & 0 & 0 & 1 & 0 & 0 & 0 & 0 & 0 \\
    0 & 0 & 0 & 0 & 1 & 0 & 0 & 0 & 0 \\
    0 & 0 & 0 & 0 & 0 & 1 & 0 & 0 & 0 \\
    0 & 0 & 0 & 0 & 0 & 0 & 1 & 0 & 0 \\
    0 & 0 & 0 & 0 & 0 & 0 & 0 & 1 & 0 \\
    0 & 0 & 0 & 0 & 0 & 0 & 0 & 0 & 1
  \end{array}\right) 
$$
\end{claim}

\begin{proof}
The proof is by induction.

\textbf{Base: $n=1$}. It holds since $f_2 = 1, f_3=2, f_4=3$.

\textbf{Step.} Using \clmref{claim: basic fibonacci 1}, we calculate:

$$
A^{n+1}=A^n \cdot A = $$
$$
\scalemath{0.9}{\left(\begin{array}{@{}ccccccccc@{}}
    2f_{2n+1}+f_{2n} & f_{2n+1}+f_{2n} & 0 & 0 & f_{2n+1}+f_{2n} & f_{2n}+f_{2n-1}-1 & f_{2n+1}+f_{2n} & 2f_{2n+1}+f_{2n}+f_{2n+2}-1 & 0 \\
   2f_{2n}+f_{2n-1} & f_{2n}+f_{2n-1} & 0 & 0 & f_{2n}+f_{2n-1}-1 & f_{2n-1}+f_{2n-2}+1 & f_{2n}+f_{2n-1}-1 & 2f_{2n}+f_{2n-1}+f_{2n+1}-1 & 0 \\
    0 & 0 & 1 & 0 & 0 & 0 & 0 & 0 & 0 \\
    0 & 0 & 0 & 1 & 0 & 0 & 0 & 0 & 0 \\
    0 & 0 & 0 & 0 & 1 & 0 & 0 & 0 & 0 \\
    0 & 0 & 0 & 0 & 0 & 1 & 0 & 0 & 0 \\
    0 & 0 & 0 & 0 & 0 & 0 & 1 & 0 & 0 \\
    0 & 0 & 0 & 0 & 0 & 0 & 0 & 1 & 0 \\
    0 & 0 & 0 & 0 & 0 & 0 & 0 & 0 & 1
  \end{array}\right) }
$$

$$
=\left(\begin{array}{@{}ccccccccc@{}}
    f_{2n+3} & f_{2n+2} & 0 & 0 & f_{2n+2} & f_{2n+1}-1 & f_{2n}+2 & f_{2n+4}-1 & 0 \\
   f_{2n+2} & f_{2n+1} & 0 & 0 & f_{2n+1}-1 & f_{2n2}+1 & f_{2n+1}-1 & f_{2n+3}-1 & 0 \\
    0 & 0 & 1 & 0 & 0 & 0 & 0 & 0 & 0 \\
    0 & 0 & 0 & 1 & 0 & 0 & 0 & 0 & 0 \\
    0 & 0 & 0 & 0 & 1 & 0 & 0 & 0 & 0 \\
    0 & 0 & 0 & 0 & 0 & 1 & 0 & 0 & 0 \\
    0 & 0 & 0 & 0 & 0 & 0 & 1 & 0 & 0 \\
    0 & 0 & 0 & 0 & 0 & 0 & 0 & 1 & 0 \\
    0 & 0 & 0 & 0 & 0 & 0 & 0 & 0 & 1
  \end{array}\right) 
$$
\end{proof}

\begin{corollary}
\label{corollary: matrix A^nB}
For $n \geqslant 1$, the matrix $A^nB$ is as follows: 
$$
A^nB=\left(\begin{array}{@{}ccccccccc@{}}
    0 & 0 & f_{2n} & f_{2n+1}-1 & f_{2n+2} & f_{2n+2}-1 & f_{2n} & 0 & f_{2n-1}-1 \\
   0 & 0 & f_{2n-1}-1 & f_{2n}+1 & f_{2n+1}-1 & f_{2n+1}-1 & f_{2n-1}-1 & 0 & f_{2n-2}+1 \\
    1 & 0 & 0 & 0 & 0 & 0 & 0 & 0 & 0 \\
    0 & 1 & 0 & 0 & 0 & 0 & 0 & 0 & 0 \\
    0 & 0 & 1 & 0 & 0 & 0 & 0 & 0 & 0 \\
    0 & 0 & 0 & 1 & 0 & 0 & 0 & 0 & 1 \\
    0 & 0 & 0 & 0 & 1 & 0 & 1 & 0 & 0 \\
    0 & 0 & 0 & 0 & 0 & 1 & 0 & 0 & 0 \\
    0 & 0 & 0 & 0 & 0 & 0 & 0 & 1 & 0
  \end{array}\right) 
$$
\end{corollary}
\begin{proof}
Direct check.
\end{proof}

\begin{claim}
\label{claim: expressing A^nB}
For $n \geqslant 1$, the matrix $A^nB$ can be expressed as follows:

\begin{equation}
\label{Eq: expressing A^nB} 
A^nB = 2f_{2n} N + M + \phi^{-2n}L,
\end{equation}
where 
$$
N =\left(\begin{array}{@{}ccccccccc@{}}
    0 & 0 & 1/2 & \phi/2 & \phi^2/2 & \phi^2/2 & 1/2 & 0 & 1/2\phi \\
   0 & 0 & 1/2\phi & 1/2 & \phi/2 & \phi/2 & 1/2\phi & 0 & 1/2\phi^2 \\
    0 & 0 & 0 & 0 & 0 & 0 & 0 & 0 & 0 \\
    0 & 0 & 0 & 0 & 0 & 0 & 0 & 0 & 0 \\
    0 & 0 & 0 & 0 & 0 & 0 & 0 & 0 & 0 \\
    0 & 0 & 0 & 0 & 0 & 0 & 0 & 0 & 0 \\
    0 & 0 & 0 & 0 & 0 & 0 & 0 & 0 & 0 \\
    0 & 0 & 0 & 0 & 0 & 0 & 0 & 0 & 0 \\
    0 & 0 & 0 & 0 & 0 & 0 & 0 & 0 & 0
  \end{array}\right) \,\,\,\,\,
 M = \left(\begin{array}{@{}ccccccccc@{}}
    0 & 0 & 0 & -1 & 0 & -1 & 0 & 0 & -1 \\
   0 & 0 & -1 & 1 & -1 & -1 & -1 & 0 & 1 \\
    1 & 0 & 0 & 0 & 0 & 0 & 0 & 0 & 0 \\
    0 & 1 & 0 & 0 & 0 & 0 & 0 & 0 & 0 \\
    0 & 0 & 1 & 0 & 0 & 0 & 0 & 0 & 0 \\
    0 & 0 & 0 & 1 & 0 & 0 & 0 & 0 & 1 \\
    0 & 0 & 0 & 0 & 1 & 0 & 1 & 0 & 0 \\
    0 & 0 & 0 & 0 & 0 & 1 & 0 & 0 & 0 \\
    0 & 0 & 0 & 0 & 0 & 0 & 0 & 1 & 0
  \end{array}\right) 
  $$
  $$
  L = \left(\begin{array}{@{}ccccccccc@{}}
     0 & 0 & 0 & 1 & 1 & 1 & 0 & 0 & -1 \\
     0 & 0 & -1 & 0 & 1 & 1 & -1 & 0 & -1 \\
    0 & 0 & 0 & 0 & 0 & 0 & 0 & 0 & 0 \\
    0 & 0 & 0 & 0 & 0 & 0 & 0 & 0 & 0 \\
    0 & 0 & 0 & 0 & 0 & 0 & 0 & 0 & 0 \\
   0 & 0 & 0 & 0 & 0 & 0 & 0 & 0 & 0 \\
   0 & 0 & 0 & 0 & 0 & 0 & 0 & 0 & 0 \\
    0 & 0 & 0 & 0 & 0 & 0 & 0 & 0 & 0 \\
    0 & 0 & 0 & 0 & 0 & 0 & 0 & 0 & 0
 \end{array}\right)
$$
Further, the following holds:
\begin{enumerate}
    \item $N^2=0$ and $\text{rk}\,(N) =1$.
    \item $(MN)^2 = MN$ and $MN(v)=v$ for $v = (0,0,\phi,1,0,0,0,0,0)^T$.
    \item $(NM)^2 = NM$.
    \item $NL=LN=0$.
    \item $L^2=0$.
\end{enumerate}

\end{claim}

\begin{proof}
\eqnref{Eq: expressing A^nB} holds by \corref{corollary: matrix A^nB} together with \clmref{claim: basic fibonacci 2}. The rest is a direct check.
\end{proof}

Let $\|\cdot\|$ denote the operator norm induced by the standard norm on $V(T)$ with basis $\{e_1, \dotsc, e_9\}$.

\begin{claim}
\label{claim: norm bounds for A^nB and A^mBA^nB}
There is a constant $C>0$ such that for $m,n \in \NN$, the following holds:
$$
\left\lVert \frac{A^nB}{2f_{2n}}-N \right\rVert \leqslant C \cdot \frac{1}{f_{2n}}, \,\,\,\,\,
\left\lVert \frac{A^mBA^nB}{2f_{2n}} - MN \right\rVert \leqslant C \cdot \frac{f_{2m}}{f_{2n}}.
$$
\end{claim}

\begin{proof}
By \clmref{claim: expressing A^nB}, we have
$$
\frac{A^nB}{2f_{2n}} - N = \frac{1}{2f_{2n}} (M + \phi^{-2n}L).
$$
Hence 
$$
\left\lVert \frac{A^nB}{2f_{2n}} - N \right\rVert \leqslant \frac{1}{2f_{2n}} \|M+\phi^{-2n}L\| \leqslant \frac{1}{2f_{2n}} (\|M\|+\|L\|).
$$
By \clmref{claim: expressing A^nB}, we have
\begin{equation}
\label{eq: expressing A^mBA^nB}
\begin{split}
A^mBA^nB &= (2f_{2m} N + M + \phi^{-2m}L) (2f_{2n} N + M + \phi^{-2n}L) \\&= 2f_{2n} MN + 2f_{2m} NM + M^2 + \phi^{-2m} LM + \phi^{-2n} ML.
\end{split}
\end{equation}
Hence 
$$
\frac{A^mBA^nB}{2f_{2n}} - MN = \frac{f_{2m}}{f_{2n}} \left( NM + \frac{1}{2f_{2m}}M^2+ \frac{\phi^{-2m}}{2f_{2m}} LM + \frac{\phi^{-2n}}{2f_{2m}} ML \right).
$$
Therefore
\begin{equation}
\begin{split}
\left\lVert \frac{A^mBA^nB}{2f_{2n}} - MN \right\rVert &\leqslant  \frac{f_{2m}}{f_{2n}} \left( \|NM\| + \frac{1}{2f_{2m}}\|M\|^2+ \frac{\phi^{-2m}}{2f_{2m}} \|LM\| + \frac{\phi^{-2n}}{2f_{2m}} \|ML\|  \right) \\ &\leqslant \frac{f_{2m}}{f_{2n}} \left( \|NM\| + \|M\|^2+ \|LM\| + \|ML\|  \right).
\end{split}
\end{equation}
Letting $C = \max\{(\|M\|+\|L\|)/2, \|NM\| + \|M\|^2+ \|LM\| + \|ML\| \}$ concludes the proof.
\end{proof}

Observe that the matrix $A^nB$ is the induced matrix of the homeomorphism $\varphi_{n+1}$ since $\varphi_{n+1} = \tau^{n+1} \circ \rho = \tau^{n} \circ (\tau \circ \rho)$.

Then the matrix $P_i$ defined as $P_i = A^{r_i-1}BA^{r_{i+1}-1}B$ for $i \geqslant 1$ corresponds to $\varphi_{r_i}\varphi_{r_{i+1}}$. We show:

\begin{claim}
\label{claim: matrix product norm bounds}
There are constants $C'>0$ and $j_0 \in \NN$ such that for $j_0 \leqslant j < k$ with odd $k-j$, the following holds:
$$ \left\lVert \frac{P_j}{c_{j+1}} \cdot \frac{P_{j+2}}{c_{j+3}} \cdot  \dotsc \cdot \frac{P_{k-1}}{c_{k}} - MN  \right\rVert \leqslant 
C' \cdot \sum_{i=j}^{\infty} \frac{c_{i}}{c_{i+1}},
$$
$$
\left\lVert \frac{A^{r_{j+1}-1}B}{c_{j+1}} \cdot \frac{P_{j+2}}{c_{j+3}} \cdot \frac{P_{j+4}}{c_{j+5}} \cdot  \dotsc \cdot \frac{P_{k-1}}{c_{k}} - NMN  \right\rVert \leqslant 
C' \cdot \sum_{i=j}^{\infty} \frac{c_{i}}{c_{i+1}}.
$$
\end{claim}

\begin{proof}
By the definition of $P_i$ and ${c_i}$ we have 
$$
\frac{P_i}{c_{i+1}} = \frac{A^{r_i-1}BA^{r_{i+1}-1}B}{2f_{2r_{i+1}-2}},
$$
therefore by \clmref{claim: norm bounds for A^nB and A^mBA^nB}, we get
$$
\left\lVert \frac{P_i}{c_{i+1}} - MN \right\rVert \leqslant C \cdot \frac{c_i}{c_{i+1}}.
$$
It follows from \eqnref{eq: condition on r_n} that $\sum_{i=1}^{\infty} \left\lVert \frac{P_i}{c_{i+1}} - MN \right\rVert < \infty$. Since the matrix $MN$ is idempotent by \clmref{claim: expressing A^nB}, we can invoke \lemref{Lemma: convergence of matrices} (see \eqnref{eq: general matrix norm bound}) 
to conclude that there is a constant $j_0 \in \NN$ such that for $j_0 \leqslant j < k$:

$$ \left\lVert \frac{P_j}{c_{j+1}} \cdot \frac{P_{j+2}}{c_{j+3}} \cdot  \dotsc \cdot \frac{P_{k-1}}{c_{k}} - MN  \right\rVert \leqslant 
2 \cdot \left(C \cdot \sum_{i=j}^{\infty} \frac{c_{i}}{c_{i+1}}\right) \cdot \|MN\|^2.
$$

Together with the triangle inequality and the first inequality in \clmref{claim: norm bounds for A^nB and A^mBA^nB}, we obtain:

\begin{equation}
\begin{split}
\left\lVert \frac{A^{r_{j+1}-1}B}{c_{j+1}} \cdot \frac{P_{j+2}}{c_{j+3}} \cdot  \dotsc \cdot \frac{P_{k-1}}{c_{k}} - NMN  \right\rVert \leqslant \left\lVert \frac{A^{r_{j+1}-1}B}{c_{j+1}} \cdot \frac{P_{j+2}}{c_{j+3}}  \cdot  \dotsc \cdot \frac{P_{k-1}}{c_{k}} - \frac{A^{r_{j+1}-1}B}{c_{j+1}}MN  \right\rVert +
\\
+ \left\lVert \frac{A^{r_{j+1}-1}B}{c_{j+1}}MN - NMN  \right\rVert
\leqslant
\left(\|N\|+ \frac{2C}{c_{j+1}} \right) \cdot 2 \cdot \left(C\cdot\sum_{i=j}^{\infty} \frac{c_{i}}{c_{i+1}}\right) \cdot \|MN\|^2 + \frac{2C}{c_{j+1}} \cdot \|MN\|.
\end{split}
\end{equation}
Letting $C' = (\|N\|+2C)\cdot2C\|MN\|^2 +  2C\|MN\| $ concludes the proof.
\end{proof}

We prove the main result of the section:

\begin{proof}[Proof of \propref{Prop: Intersection numbers}]

Using \eqnref{Eq: def Phi}, we can write 

$$
i(\gamma_{i-1}, \gamma_{j}) = i ( \Phi_{i-1} (\gamma_0), \Phi_{i-1} \varphi_{r_i} \dotsc \varphi_{r_j}(\gamma_0) ) =  i ( \gamma_0, \varphi_{r_i} \dotsc \varphi_{r_j}(\gamma_0) ). 
$$

We can express the multicurve $\varphi_{r_i} \dotsc \varphi_{r_j}(\gamma_0)$ as a vector in $V (T)$ as follows: $P_i P_{i+2} \dotsc P_{j-1} (w_1+w_3)$. 
Notice that the measured lamination that corresponds to the vector $MN(w_1+w_3)=\frac{1}{2}w_3+ \frac{1}{2\phi}w_4$ is the unstable lamination of the homeomorphism $\rho \tau \rho^{-1}$, which has a positive intersection number with the curve that corresponds to $w_3$, hence also with the multicurve $\gamma_0$. Since the natural map $V(T) \to \ML(S)$ and the intersection number $i(\cdot,\cdot)$ are continuous, by \clmref{claim: matrix product norm bounds} we can choose $i_2 \in \NN$ so that for $i_2\leqslant i < j$, the intersection number of the measured lamination $\frac{P_i}{c_{i+1}} \cdot \frac{P_{i+2}}{c_{i+3}} \cdot  \dotsc \cdot \frac{P_{j-1}}{c_{j}}(w_1+w_3)$ and $\gamma_0$ is bounded above and below from zero, where the bound is independent of $i$ and $j$. Hence for $i_2 \leqslant i<j$, the intersection number $i(\gamma_{i-1},\gamma_j)$ is equal to $c_{i+1}c_{i+3} \dotsc c_{j}$ up to a fixed multiplicative constant.

Similarly, we can write
$$
i(\gamma_{i}, \gamma_{j}) = i ( \Phi_{i} (\gamma_0), \Phi_{i} \varphi_{r_{i+1}} \dotsc \varphi_{r_j}(\gamma_0) ) =  i ( \gamma_0, \varphi_{r_{i+1}} \dotsc \varphi_{r_j}(\gamma_0) ). 
$$
We can express the multicurve $\varphi_{r_{i+1}} \dotsc \varphi_{r_j}(\gamma_0)$ as a vector in $V (T)$ as follows: $A^{r_{i+1}-1}BP_{i+2} P_{i+4} \dotsc P_{j-1} (w_1+w_3)$. Notice that the measured lamination that corresponds to the vector $NMN(w_1+w_3) =  \frac{1}{2}w_1+ \frac{1}{2\phi}w_2$ is the unstable lamination of the homeomorphism $\tau$, which has a positive intersection number with the curve that corresponds to $w_1$, hence also with the multicurve $\gamma_0$. By \clmref{claim: matrix product norm bounds}, for $i_2\leqslant i < j$ the intersection number of the measured lamination $\frac{A^{r_{i+1}-1}B}{c_{i+1}}\cdot\frac{P_{i+2}}{c_{i+3}} \cdot \frac{P_{i+4}}{c_{i+5}} \cdot  \dotsc \cdot \frac{P_{j-1}}{c_{j}}(w_1+w_3)$ and $\gamma_0$ is bounded above and below from zero, where the bound is independent of $i$ and $j$. Hence for $i_2 \leqslant i<j$, the intersection number $i(\gamma_{i},\gamma_j)$ is equal to $c_{i+1}c_{i+3} \dotsc c_{j}$ up to a fixed multiplicative constant.  
\end{proof}


\section{Non-unique ergodicity}
\label{section: NUE}

In this section we show that the ending lamination $\lambda$ constructed in \secref{section: construction of the lamination} is not uniquely ergodic. Namely, we prove that the appropriately scaled subsequences of multicurves $\{\gamma_i\}$ with even and odd indices converge to non-zero measured laminations that are not multiples of each other. Further, we show that the limiting measured laminations are ergodic and are the only ergodic transverse measures on $\lambda$.

\begin{claim}
\label{claim: def even and odd}
There are $\lambda_{e}, \lambda_{o} \in \ML(S)$ such that the following holds as $n \to \infty$:
$$\frac{\gamma_{2n}}{c_2c_4\dotsc c_{2n}} \to \lambda_{e}, \,\,\,\,\, \frac{\gamma_{2n+1}}{c_1c_3\dotsc c_{2n+1}} \to \lambda_{o}.$$
\end{claim}
\begin{proof}
Notice that the vector $(\prod_{i=1}^n \frac{P_{2i-1}}{c_{2i}}) (w_1 + w_3)$ corresponds to $\frac{\gamma_{2n}}{c_2c_4\dotsc c_{2n}}$ and the vector  $\frac{A^{r_1-1}B}{c_1}(\prod_{i=1}^n  \frac{P_{2i}}{c_{2i+1}})(w_1+w_3)$ corresponds to $\frac{\gamma_{2n+1}}{c_1c_3\dotsc c_{2n+1}}$. By \eqnref{eq: condition on r_n}, \clmref{claim: norm bounds for A^nB and A^mBA^nB} and \lemref{Lemma: convergence of matrices}, the infinite products $\prod_{i=1}^\infty \frac{P_{2i-1}}{c_{2i}}$ and $\prod_{i=1}^\infty  \frac{P_{2i}}{c_{2i+1}}$ converge. Hence the vectors $(\prod_{i=1}^n \frac{P_{2i-1}}{c_{2i}}) (w_1 + w_3)$ and $\frac{A^{r_1-1}B}{c_1}(\prod_{i=1}^n  \frac{P_{2i}}{c_{2i+1}})(w_1+w_3)$ converge as $n \to \infty$, and the result follows.
\end{proof}

\begin{claim}
\label{claim: even and odd}
As $n \to \infty$,
$$
\frac{i(\gamma_{2n}, \lambda_{e})}{i(\gamma_{2n}, \lambda_{o})} \to 0, \,\,\,\,\, \frac{i(\gamma_{2n+1}, \lambda_{e})}{i(\gamma_{2n+1}, \lambda_{0})} \to \infty.
$$
\end{claim}
\begin{proof}
Suppose that $2n \geqslant i_2$ where $i_2 \in \NN$ is the constant from \propref{Prop: Intersection numbers}. Then by \propref{Prop: Intersection numbers} for $m > n $ we have:
$$
i \left(\gamma_{2n}, \frac{\gamma_{2m}}{c_2c_4\dotsc c_{2m}} \right) = \frac{i(\gamma_{2n},\gamma_{2m})}{c_2c_4\dotsc c_{2m}} \, \emul \, \frac{c_{2n+2}c_{2n+4}\dotsc c_{2m}}{c_2c_4\dotsc c_{2m}} = \frac{1}{c_2c_4\dotsc c_{2n}}.
$$
Since it holds for every $m>n$, by passing to the limit as $m \to \infty$, we have $i(\gamma_{2n}, \lambda_{e}) \, \emul \, \frac{1}{c_2c_4\dotsc c_{2n}}$. In particular, $i(\gamma_{2n}, \lambda_{e}) \neq 0$. 

Similarly, for $m > n$ we have
$$
i \left(\gamma_{2n}, \frac{\gamma_{2m+1}}{c_1c_3\dotsc c_{2m+1}} \right) = \frac{i(\gamma_{2n},\gamma_{2m+1})}{c_1c_3\dotsc c_{2m+1}} \, \emul \, \frac{c_{2n+1}c_{2n+3}\dotsc c_{2m+1}}{c_1c_3\dotsc c_{2m+1}} = \frac{1}{c_1c_3\dotsc c_{2n-1}}.
$$
Since it holds for every $m>n$, by passing to the limit as $m \to \infty$, we have $i(\gamma_{2n}, \lambda_{o}) \, \emul \, \frac{1}{c_1c_3\dotsc c_{2n-1}}$. In particular, $i(\gamma_{2n}, \lambda_{o}) \neq 0$.

Putting this together, we obtain
$$
\frac{i(\gamma_{2n}, \lambda_{e})}{i(\gamma_{2n}, \lambda_{o})} \, \emul \, \frac{c_1c_3\dotsc c_{2n-1}}{c_2c_4\dotsc c_{2n}}.
$$
It follows from \eqnref{eq: condition on r_n} that $\frac{c_{2n-1}}{c_{2n}} \to 0$ as $n \to \infty$. Hence $ \frac{c_1c_3\dotsc c_{2n-1}}{c_2c_4\dotsc c_{2n}} \to 0$, and therefore $\frac{i(\gamma_{2n}, \lambda_{e})}{i(\gamma_{2n}, \lambda_{o})} \to 0$ as $n \to \infty$.

Similarly, for $m > n$ we have:
$$
\frac{i(\gamma_{2n+1}, \frac{\gamma_{2m}}{c_2c_4\dotsc c_{2m}})}{i(\gamma_{2n+1}, \frac{\gamma_{2m+1}}{c_1c_3\dotsc c_{2m+1}})} = \frac{c_1c_3\dotsc c_{2m+1}}{c_2c_4\dotsc c_{2m}} \cdot  \frac{i(\gamma_{2n+1},\gamma_{2m})}{i(\gamma_{2n+1},\gamma_{2m+1})} \emul \frac{c_1c_3\dotsc c_{2m+1}}{c_2c_4\dotsc c_{2m}} \cdot \frac{c_{2n+2}c_{2n+4}\dotsc c_{2m}}{c_{2n+3}c_{2n+5}\dotsc c_{2m+1}} = \frac{c_1c_3\dotsc c_{2n+1}}{c_2c_4 \dotsc c_{2n}}.
$$
Since it holds for every $m>n$, by passing to the limit as $m \to \infty$, we have $\frac{i(\gamma_{2n+1},\lambda_{e})}{i(\gamma_{2n+1},\lambda_{o})} \, \emul \, \frac{c_1c_3\dotsc c_{2n+1}}{c_2c_4 \dotsc c_{2n}}$. It follows from \eqnref{eq: condition on r_n} that $\frac{c_{2n+1}}{c_{2n}} \to \infty$, hence $\frac{c_1c_3\dotsc c_{2n+1}}{c_2c_4 \dotsc c_{2n}} \to \infty$ and therefore $\frac{i(\gamma_{2n+1}, \lambda_{e})}{i(\gamma_{2n+1}, \lambda_{0})} \to \infty$ as $n \to \infty$.
\end{proof}

\begin{corollary}
\label{corollary: even and odd}
The measured laminations $\lambda_{e}, \lambda_{o}$ are non-zero and are not the multiples of each other.
\end{corollary}

\begin{proof}
It was shown in \clmref{claim: even and odd} that $i(\gamma_{2n},\lambda_{e}) \neq 0$ and $i(\gamma_{2n},\lambda_{o}) \neq 0$ for $2n\geqslant i_2$, hence $\lambda_{e}\neq 0$ and $\lambda_{o}\neq 0$.
If $\lambda_{e}, \lambda_{o}$ are multiples of each other, then the sequence $\frac{i(\gamma_{2n}, \lambda_{e})}{i(\gamma_{2n}, \lambda_{o})}$ is constant, which contradicts \clmref{claim: even and odd}.
\end{proof}

\begin{proposition}
\label{proposition: not uniquely ergodic}
The ending lamination $\lambda$ is not uniquely ergodic.
\end{proposition}
\begin{proof}
The measured lamination $\lambda_{e}$ can be expressed as $\lambda_{e} = \lambda_{e}' + \lambda_{e}''$, where $\lambda_{e}'$ is the measured lamination that corresponds to the vector $(\prod_{i=1}^{\infty} \frac{P_{2i-1}}{c_{2i}}) (w_1)$ and $\lambda_{e}''$ is the measured lamination that corresponds to the vector $(\prod_{i=1}^{\infty} \frac{P_{2i-1}}{c_{2i}}) (w_3)$. The simple closed curve that corresponds to the vector $\prod_{i=1}^n P_{2i-1} (w_1)$ is at distance $2$ from the curve $\alpha_{2n}$ in the curve graph for each $n\geqslant 1$, hence the sequence of curves $\prod_{i=1}^n P_{2i-1} (w_1), n\geqslant 1$ converges to $\lambda$ in the Gromov boundary as $n \to \infty$. Then by \thmref{Theorem: Gromov boundary of the curve graph}, the measured lamination $\lambda_{e}'$ is either supported on $\lambda$ or zero. Repeating the same argument for $\lambda_{e}''$ and since $\lambda_{e} \neq 0$ by \corref{corollary: even and odd}, we obtain that $\lambda_{e}$ is supported on $\lambda$. By a similar argument, the measured lamination $\lambda_{o}$ is supported on $\lambda$. By \corref{corollary: even and odd}, $\lambda$ is not uniquely ergodic.
\end{proof}

Let $C(\lambda)$ denote the convex cone of transverse measures supported on $\lambda$. Since the measured lamination $\lambda_{e}$ is carried by $T$, the ending lamination $\lambda$, being the support of $\lambda_{e}$, is carried by $T$. Hence every measured lamination in $C(\lambda)$ is carried by $T$. In fact, we can show more: 

\begin{claim}
\label{claim: smaller cones carry}
For every $n \in \NN$, the image of the convex cone $P_1 P_3 \dotsc P_{2n-1} (V(T))$ under the natural map to $\ML(S)$ contains $C(\lambda)$.
\end{claim}
\begin{proof}
Notice that $P_1 P_3 \dotsc P_{2n-1} (V(T))$ is isomorphic to the convex cone of the non-negative real assignments of weights satisfying the switch conditions to the branches of the train track $\varphi_{r_1} \dotsc \varphi_{r_{2n}}(T)$. It is then sufficient to show that the measured lamination $\lambda_{e}$ is carried by the train track $\varphi_{r_1} \dotsc \varphi_{r_{2n}}(T)$. Indeed, in this case every measured lamination in $C(\lambda)$ is carried by $\varphi_{r_1} \dotsc \varphi_{r_{2n}}(T)$. Since the measured lamination corresponding to the vector $(\prod_{i=n+1}^{\infty} \frac{P_{2i-1}}{c_{2i}}) (w_1 + w_3)$ is carried by $T$ by \propref{Proposition: invariant bigon track}, the measured lamination corresponding to the vector $\frac{P_1}{c_2} \frac{P_3}{c_4} \dotsc \frac{P_{2n-1}}{c_{2n}} \cdot (\prod_{i=n+1}^{\infty} \frac{P_{2i-1}}{c_{2i}}) (w_1 + w_3)$ is carried by $\varphi_{r_1} \dotsc \varphi_{r_{2n}}(T)$. Since the latter measured lamination is $\lambda_{e}$, the result follows. 
\end{proof}

To find all ergodic transverse measures on $\lambda$, we study the shapes of the convex cones $P_1P_3\dotsc P_{2n-1} (V(T))$ as $n \to \infty$. Roughly speaking, we will show that for each $n \in \NN$, the set of the generators of the cone $P_1P_3\dotsc P_{2n-1} (V(T))$ can be divided into two subsets such that the angles between pairs of generators within each of the subsets converge to zero as $n \to \infty$ (\lemref{lemma: small angles}). From this the upper bound on the number of ergodic transverse measures will follow.

Endow $V(T)$ with the standard inner product with respect to the basis $\{e_1, \dotsc, e_9\}$. We start with the following helpful observation:

\begin{claim}
\label{claim: inner product observation}
For $i \in \NN$,
$$
\langle P_i(e_1), e_1 \rangle = \frac{c_i}{2}, \,\,\, \langle P_i(e_3), e_3 \rangle = \frac{c_{i+1}}{2}.
$$ 
\end{claim}
\begin{proof}
By \eqnref{eq: expressing A^mBA^nB}, we have 
\begin{equation}
\label{eq: expressing P_i}
P_i = c_{i+1} MN + c_{i} NM + M^2 + \phi^{-2(r_{i}-1)} LM + \phi^{-2(r_{i+1}-1)} ML.
\end{equation}
Notice that $MN(e_1)=0$ since $e_1 \in \text{ker} \, N$. We have $NM(e_1) = \frac{1}{2}e_1+\frac{1}{2\phi}e_2$. It is also a direct check that $\langle M^2(e_1), e_1 \rangle = \langle LM(e_1), e_1 \rangle =\langle ML(e_1), e_1 \rangle =0$, hence $\langle P_i(e_1), e_1 \rangle = \frac{c_i}{2}$.

Similarly, we have $MN(e_3) = \frac{1}{2}e_3+\frac{1}{2\phi}e_4$ and $\langle NM(e_3), e_3 \rangle = \langle M^2(e_3), e_3 \rangle = \langle LM(e_3), e_3 \rangle = \langle ML(e_3), e_3 \rangle =0$. Hence $\langle P_i(e_3), e_3 \rangle = \frac{c_{i+1}}{2}$.
\end{proof}
Next, we prove:
\begin{claim}
\label{claim: P_1P_3...P_2n-1 norm lower bounds}
For every $n \in \NN$ the following holds. If $e_i \notin \text{ker} \, N$, then 
$$
\| P_1P_3 \dotsc P_{2n-1}MN(e_i) \| \geqslant \frac{c_2c_4 \dotsc c_{2n}}{2^{n+1}\phi}.
$$
If $e_i \in \text{ker} \, N$, then 
$$
\| P_1P_3 \dotsc P_{2n-1}NM(e_i)\| \geqslant \frac{c_1c_3 \dotsc c_{2n-1}}{2^{n+1}\phi}.
$$
\end{claim}
\begin{proof}
Notice that if $e_i \notin \text{ker} \, N$, then $\langle MN(e_i), e_3 \rangle \geqslant \frac{1}{2\phi}$. Since the matrices $P_i$ are non-negative for $i \in \NN$, we have $$\langle P_i(v) , e_3 \rangle \geqslant \langle P_i(\langle v, e_3 \rangle \cdot e_3) , e_3 \rangle $$ for all $v \in V(T)$. Applying \clmref{claim: inner product observation} $n$ times, it follows that $\langle P_1P_3 \dotsc P_{2n-1}MN(e_i), e_3 \rangle \geqslant \frac{c_2c_4 \dotsc c_{2n}}{2^{n+1}\phi}$, therefore $
\| P_1P_3 \dotsc P_{2n-1}MN(e_i) \| \geqslant \frac{c_2c_4 \dotsc c_{2n}}{2^{n+1}\phi}$. 

Similarly, if $e_i \in \text{ker} \, N$, then $\langle NM(e_i),e_1 \rangle \geqslant \frac{1}{2\phi}$. Since the matrices $P_i$ are non-negative for $i \in \NN$, together with \clmref{claim: inner product observation} it follows that $\langle P_1P_3 \dotsc P_{2n-1}NM(e_i), e_1 \rangle \geqslant \frac{c_1c_3 \dotsc c_{2n-1}}{2^{n+1}\phi}$, hence $
\| P_1P_3 \dotsc P_{2n-1}NM(e_i) \| \geqslant \frac{c_1c_1 \dotsc c_{2n-1}}{2^{n+1}\phi}$. 
\end{proof}

Let $K_i$ be the matrix defined as $K_i = M^2 + \phi^{-2(r_{i}-1)} LM + \phi^{-2(r_{i+1}-1)} ML$ for $i \in \NN$. Then 
\begin{equation}
\label{eq: expressing P_i more concisely}
P_i = c_{i+1}MN + c_{i}NM+K_i.
\end{equation}
Notice that $\|K_i\| \leqslant \|M\|^2+\|LM\|+\|ML\|$ for all $i \in \NN$.  

\begin{claim}
\label{claim: P_1P_3...P_2n-1 norm upper bounds}
There is a constant $D>0$ such that for every $n \in \NN$ and $1 \leqslant i \leqslant 9$ the following holds:
$$
\| P_1P_3 \dotsc P_{2n-1}NM(e_i) \| \leqslant D^{n+1} \cdot c_1c_3 \dotsc c_{2n-1},
$$
$$
\| P_1P_3 \dotsc P_{2n-1}MN(e_i)\| \leqslant D^{n+1} \cdot c_2c_4 \dotsc c_{2n},
$$
$$
\| P_1P_3 \dotsc P_{2n-1}K_{2n+1}(e_i)\| \leqslant D^{n+1} \cdot c_2c_4 \dotsc c_{2n}.
$$
\end{claim}
\begin{proof}
Consider the first inequality. Expressing each matrix $P_i$ in the product $P_1P_3\dotsc P_{2n-1}NM$ as in \eqnref{eq: expressing P_i more concisely} and opening up the brackets, we obtain a sum of $3^n$ matrices with coefficients. It follows from the identity $N^2=0$ that $c_1c_3 \dotsc c_{2n-1}$ is the largest coefficient in the sum which is multiplied by a non-zero matrix. Since the norm of each matrix in the sum is at most $(\sup_{i}\{\|MN\|,\|NM\|,\|K_i\|\})^{n+1}$, by the triangle inequality we have
$$
\| P_1P_3 \dotsc P_{2n-1}MN(e_i)\| \leqslant 3^n \cdot c_2c_4 \dotsc c_{2n} \cdot \left(\sup_{i} \, \{\|MN\|,\|NM\|,\|K_i\|\}\right)^{n+1}.
$$
Letting $D = 3 \cdot (\sup_{i}\{\|MN\|,\|NM\|,\|K_i\|\})$ concludes the first inequality. The second and the third inequalities follow by a similar argument and noticing that $c_2c_4 \dotsc c_{2n}$ is the largest coefficient in the corresponding sum.
\end{proof}
\begin{remark}
It is possible to obtain better upper bounds for $\|P_1P_3 \dotsc P_{2n-1}MN(e_i)\|$ and $\| P_1P_3 \dotsc P_{2n-1}K_{2n+1}(e_i)\|$ using \clmref{claim: matrix product norm bounds}, but weaker bounds will suffice for our purposes.
\end{remark}

\begin{lemma}
\label{lemma: small angles}
There is a constant $D'>0$ such that for every $n\in \NN$ the following holds. If $e_i, e_j \notin \text{ker} \, N$, then 
$$
1- \cos \angle (P_1 P_3 \dotsc P_{2n+1} (e_i), P_1 P_3 \dotsc P_{2n+1} (e_j)) \leqslant (D')^{n+1} \cdot \frac{c_1c_3\dotsc c_{2n+1}}{c_2c_4 \dotsc c_{2n+2}}.
$$
If $e_i, e_j \in \text{ker} \, N$, then 
$$
1 - \cos \angle (P_1 P_3 \dotsc P_{2n+1} (e_i), P_1 P_3 \dotsc P_{2n+1} (e_j)) \leqslant (D')^{n+1} \cdot \frac{c_2c_4\dotsc c_{2n}}{c_1c_3\dotsc c_{2n+1}}.
$$ 
Further, in either case $\angle (P_1 P_3 \dotsc P_{2n+1} (e_i), P_1 P_3 \dotsc P_{2n+1} (e_j)) \to 0$ as $n \to \infty$.
\end{lemma}

\begin{proof}
By \eqnref{eq: expressing P_i more concisely}, we can write
$$
P_1 P_3 \dotsc P_{2n+1} = P_1 P_3 \dotsc P_{2n-1} (c_{2n+2}MN+c_{2n+1}NM+K_{2n+1}).
$$
Consider the first inequality. Let 
$$v_1=c_{2n+2}P_1 P_3 \dotsc P_{2n-1}MN(e_i), \,\,\, v_2=c_{2n+2}P_1 P_3 \dotsc P_{2n-1}MN(e_j),$$ $$w_1= P_1 P_3 \dotsc P_{2n-1} (c_{2n+1}NM+K_{2n+1})(e_i), \,\,\, w_2= P_1 P_3 \dotsc P_{2n-1} (c_{2n+1}NM+K_{2n+1})(e_j).$$
Notice that $v_1+w_1 = P_1 P_3 \dotsc P_{2n+1} (e_i)$ and $v_2+w_2=P_1 P_3 \dotsc P_{2n+1} (e_j)$. We also have $\angle (v_1,v_2)=0$ since the vectors $MN(e_i)$ and $MN(e_j)$ are collinear.
By \clmref{claim: P_1P_3...P_2n-1 norm lower bounds} and \clmref{claim: P_1P_3...P_2n-1 norm upper bounds}, we have 
$$
\frac{c_2c_4 \dotsc c_{2n+2}}{2^{n+1}\phi} \leqslant \|v_1\|,\|v_2\| \leqslant D^{n+1} \cdot c_2c_4 \dotsc c_{2n+2},
$$
$$
\|w_1\|,\|w_2\|  \leqslant D^{n+1} \cdot (c_1c_3\dotsc c_{2n+1} + c_2c_4 \dotsc c_{2n})  \leqslant 2\cdot D^{n+1} \cdot c_1c_3\dotsc c_{2n+1}. 
$$

Then by \lemref{lemma: Angle Estimate Lemma}, we have
\begin{equation}
\begin{split}
& 1- \cos \angle (P_1 P_3 \dotsc P_{2n+1} (e_i), P_1 P_3 \dotsc P_{2n+1} (e_j)) \\ & \leqslant \frac{2\left( 4 \cdot D^{2n+2} \cdot c_2c_4\dotsc c_{2n+2}\cdot c_1c_3\dotsc c_{2n+1} + 4 \cdot D^{2n+2} \cdot (c_1c_3\dotsc c_{2n+1})^2 \right)}{\frac{c_2c_4 \dotsc c_{2n+2}}{2^{n+1}\phi}\cdot \frac{c_2c_4 \dotsc c_{2n+2}}{2^{n+1}\phi}}  \\ & \leqslant 2^{2n+2}\phi^2 \cdot 8 \cdot D^{2n+2} \cdot  \frac{2c_2c_4\dotsc c_{2n+2}\cdot c_1c_3\dotsc c_{2n+1}}{(c_2c_4\dotsc c_{2n+2})^2}
\\ & \leqslant 2^{2n+6}\phi^2 \cdot D^{2n+2} \cdot \frac{c_1c_3\dotsc c_{2n+1}}{c_2c_4 \dotsc c_{2n+2}}.
\end{split}
\end{equation}
For the second inequality, notice that $c_{2n+2}P_1 P_3 \dotsc P_{2n-1}MN(e_i)=c_{2n+2}P_1 P_3 \dotsc P_{2n-1}MN(e_j)=0$ since $N(e_i)=N(e_j)=0$. We let
$$v_1'=c_{2n+1}P_1 P_3 \dotsc P_{2n-1}NM(e_i), \,\,\, v_2'=c_{2n+1}P_1 P_3 \dotsc P_{2n-1}NM(e_j),$$ $$w_1'= P_1 P_3 \dotsc P_{2n-1} K_{2n+1}(e_i), \,\,\, w_2'= P_1 P_3 \dotsc P_{2n-1} K_{2n+1}(e_j).$$
Notice that $v_1'+w_1' = P_1 P_3 \dotsc P_{2n+1} (e_i)$ and $v_2'+w_2'=P_1 P_3 \dotsc P_{2n+1} (e_j)$. We also have $\angle (v_1',v_2')=0$ since the vectors $NM(e_i)$ and $NM(e_j)$ are collinear.
Then by \lemref{lemma: Angle Estimate Lemma}, \clmref{claim: P_1P_3...P_2n-1 norm lower bounds} and \clmref{claim: P_1P_3...P_2n-1 norm upper bounds}, we have
\begin{equation}
\begin{split}
& 1- \cos \angle (P_1 P_3 \dotsc P_{2n+1} (e_i), P_1 P_3 \dotsc P_{2n+1} (e_j)) \\ & \leqslant 
\frac{2\left( 2\cdot  D^{2n+2} \cdot c_1c_3 \dotsc c_{2n+1} \cdot c_2c_4 \dotsc c_{2n} +   D^{2n+2} \cdot (c_2c_4 \dotsc c_{2n})^2 \right)}{\frac{c_1c_3 \dotsc c_{2n+1}}{2^{n+1}\phi}\cdot \frac{c_1c_3 \dotsc c_{2n+1}}{2^{n+1}\phi}} \\ & \leqslant 
2^{2n+2}\phi^2 \cdot 2 \cdot D^{2n+2} \cdot  \frac{3c_1c_3\dotsc c_{2n+1}\cdot c_2c_4 \dotsc c_{2n}}{(c_1c_3\dotsc c_{2n+1})^2} \\ & \leqslant 
2^{2n+3}\phi^2 \cdot 3 \cdot D^{2n+2} \cdot  \frac{c_2c_4 \dotsc c_{2n}}{c_1c_3\dotsc c_{2n+1}}.
\end{split}
\end{equation}
Letting $D' =  2^4 \phi \cdot 3 \cdot D^{2}$ concludes the desired inequalities.
By \eqnref{eq: condition on r_n}, we have $\frac{c_{n}}{c_{n+1}} \to 0$ as $n\to \infty$. Hence for sufficiently large $n$ the upper bounds for $1- \cos \angle (P_1 P_3 \dotsc P_{2n+1} (e_i), P_1 P_3 \dotsc P_{2n+1} (e_j))$ decrease at least exponentially with $n$, therefore $\angle (P_1 P_3 \dotsc P_{2n+1} (e_i), P_1 P_3 \dotsc P_{2n+1} (e_j)) \to 0$ as $n \to \infty$.
\end{proof}

\begin{proposition}
The measured laminations $\lambda_{e}, \lambda_{o}$ are ergodic. Further, any transverse measure on $\lambda$ is a linear combination of $\lambda_{e}$ and $\lambda_{o}$.
\end{proposition}
\begin{proof}



Let $\Delta = \{ x_1e_1 + \dotsc + x_9 e_9  \, | \, \sum_{i=1}^9 x_i = 1 \} \subset V(T)$ be the standard unit simplex. Notice that the set $P_1P_3\dotsc P_{2n-1} (V(T)) \, \cap \, \Delta$ is the convex hull of the points $ \text{span}(P_1P_3\dotsc P_{2n-1}(e_i)) \cap \Delta$ for $ i=1,\dotsc,9$. Since the infinite product $\prod_{i=1}^{\infty} \frac{P_{2i-1}}{c_{2i}}$ converges (see the proof of \clmref{claim: def even and odd}), the sequence of compact sets $P_1P_3\dotsc P_{2n-1} (V(T)) \, \cap \, \Delta$ converges in the Hausdorff metric on $V(T)$ as $n \to \infty$. It follows from \lemref{lemma: small angles} that the limiting set is either an interval or a point. If $\lambda$ admits at least three ergodic transverse measures up to scalar, then the set of points in $\Delta$ that correspond to measured laminations in $C(\lambda)$ contains a convex triangle. Let $R>0$ be the radius of the circumscribed circle of such triangle. For sufficiently large $n$ so that the Hausdorff distance from $P_1P_3\dotsc P_{2n-1} (V(T)) \, \cap \, \Delta$ to the limiting set is less than $R/8$, the set $P_1P_3\dotsc P_{2n-1} (V(T)) \, \cap \, \Delta$ cannot contain a 2-dimensional disk of radius $r$: it is immediate if the limiting set is a point and if the limiting set is an interval it follows from \lemref{lemma: Interval Neighborhood Lemma} since $R > 2\sqrt{2(8-1)}R/8$. Since the projective class of every non-zero measure in $C(\lambda)$ is represented in $P_1P_3\dotsc P_{2n-1} (V(T)) \, \cap \, \Delta$ by \clmref{claim: smaller cones carry}, we arrive at a contradiction. Hence $\lambda$ admits at most two ergodic transverse measures up to scalar. Together with \propref{proposition: not uniquely ergodic} we obtain that $\lambda$ admits exactly two ergodic transverse measured up to scalar.

Let $\lambda_1, \lambda_2$ be ergodic transverse measures on $\lambda$ that are not multiples of each other. Then we can write $$\lambda_{e} = \alpha_{1} \lambda_1  + \alpha_{2} \lambda_2, \,\,\, \lambda_{o} = \beta_1 \lambda_1  + \beta_2 \lambda_2$$ for $\alpha_{1},\alpha_2,\beta_1,\beta_2 \geqslant 0$. Since $\lambda_{o} \neq 0$, at least one of the numbers $\beta_1, \beta_2$ is non-zero. Without loss of generality, assume that $\beta_1 \neq 0$. Since $\lambda_1, \lambda_2$ are filling, we can write
$$
\frac{i(\gamma_{2n+1},\lambda_{e})}{i(\gamma_{2n+1},\lambda_{o})} = \frac{\alpha_1 \cdot i(\gamma_{2n+1},\lambda_{1})+\alpha_2 \cdot i(\gamma_{2n+1},\lambda_{2})}{\beta_1 \cdot i(\gamma_{2n+1},\lambda_{1})+\beta_2 \cdot i(\gamma_{2n+1},\lambda_{2})} \leqslant  \frac{\alpha_1 \cdot i(\gamma_{2n+1},\lambda_{1})+\alpha_2 \cdot i(\gamma_{2n+1},\lambda_{2})}{\beta_1 \cdot i(\gamma_{2n+1},\lambda_{1})} = \frac{\alpha_1}{\beta_1} + \frac{\alpha_2}{\beta_1}\cdot\frac{i(\gamma_{2n+1},\lambda_{2})}{i(\gamma_{2n+1},\lambda_{1})}.
$$
Since by \clmref{claim: even and odd}, $\frac{i(\gamma_{2n+1},\lambda_{e})}{i(\gamma_{2n+1},\lambda_{o})} \to \infty$ as $n \to \infty$, it follows that $\frac{i(\gamma_{2n+1},\lambda_{2})}{i(\gamma_{2n+1},\lambda_{1})} \to \infty$ as $n \to \infty$. If $\beta_2 \neq 0$ we have
$$
\frac{i(\gamma_{2n+1},\lambda_{e})}{i(\gamma_{2n+1},\lambda_{o})} = \frac{\alpha_1 \cdot i(\gamma_{2n+1},\lambda_{1})+\alpha_2 \cdot i(\gamma_{2n+1},\lambda_{2})}{\beta_1 \cdot i(\gamma_{2n+1},\lambda_{1})+\beta_2 \cdot i(\gamma_{2n+1},\lambda_{2})} \leqslant  \frac{\alpha_1 \cdot i(\gamma_{2n+1},\lambda_{1})+\alpha_2 \cdot i(\gamma_{2n+1},\lambda_{2})}{\beta_2 \cdot i(\gamma_{2n+1},\lambda_{2})} =  \frac{\alpha_1}{\beta_2}\cdot\frac{i(\gamma_{2n+1},\lambda_{1})}{i(\gamma_{2n+1},\lambda_{2})} + \frac{\alpha_2}{\beta_2} .
$$
Letting $n \to \infty$, we get $\limsup_{n\to \infty}\frac{i(\gamma_{2n+1},\lambda_{e})}{i(\gamma_{2n+1},\lambda_{o})} \leqslant \frac{\alpha_2}{\beta_2}$, contradiction. Hence $\beta_2 = 0$. A similar argument using that $\frac{i(\gamma_{2n},\lambda_{o})}{i(\gamma_{2n},\lambda_{e})} \to \infty$ as $n \to \infty$ by \clmref{claim: even and odd} shows that one of the numbers $\alpha_1,\alpha_2$ is zero.
Therefore $\lambda_{e}, \lambda_{o}$ are ergodic transverse measures themselves, and the result follows.
\end{proof}

\section{Relative twisting bounds} 
\label{section: relative twisting bounds}

In this section we prove that the lamination $\lambda$ constructed in \secref{section: construction of the lamination} has uniformly bounded annular projection distances. To show this, we return to the sequence of curves $\{ \alpha_i \}$, defined in \secref{Subsection: alpha sequence}. First we show:


\begin{lemma}
\label{lemma: annular projections are small}
For every $i \geqslant 2$
$$
d_{\alpha_i} (\alpha_{i-2},\alpha_{i+2}) \leqslant 7.
$$
\end{lemma}
\begin{proof}
By \eqnref{Eq: no-interac-with-pA}, we have
\begin{equation}
\label{equation: twist 3}
d_{\alpha_i}(\alpha_{i-2},\alpha_{i+2}) =  d_{\alpha_2}(\alpha_0,\varphi_{r_{i-1}}\alpha_{3}) = d_{\alpha_2}(\alpha_0,\tau^{r_{i-1}}(\rho\alpha_3)).
\end{equation}
Choose a marked complete hyperbolic metric $X$ on $S$ of finite volume. Let $\widetilde{\alpha_2}$ and $\widetilde{\rho\alpha_3}$ be geodesic lifts of $\alpha_2$ and $\rho\alpha_3$ in the universal cover $\widetilde{X} \cong \HH^2$ intersecting at the point $O \in \HH^2$. Let $\delta_0, \delta_1$ be the curves on $S$ shown in \figref{figure with alphas}. Let $\widetilde{\delta_0}$ and $\widetilde{\delta'_0}$ be the geodesic lifts of $\delta_0$ in $\HH^2$ that intersect $\widetilde{\rho\alpha_3}$ at the points $A, B \in \HH^2$, respectively, such that the geodesic segment $[AB] \subset \HH^2$ contains the point $O$ and does not contain any other intersection points of lifts of $\delta_0$ with $\widetilde{\rho\alpha_3}$.
Let $p,q \in \partial \HH^2$ be the endpoints of $\widetilde{\delta_0}$ and let $r,s \in \partial \HH^2$ be the endpoints of $\widetilde{\delta'_0}$. Let $(pq) \subset \partial \HH^2$ be the open interval such that $r,s \notin (pq)$ and let $(rs) \subset \partial \HH^2$ be the open interval such that $p,q \notin (rs)$. See \figref{figure: universal cover} (left).

Let $\widetilde{H_{\delta_0}}: \HH^2 \to \HH^2$ be a lift of the half-twist $H_{\delta_0}$ such that $\widetilde{H_{\delta_0}}(O)=O$. Similarly, let $\widetilde{H^{-1}_{\delta_1}}: \HH^2 \to \HH^2$ be a lift of the half-twist $H^{-1}_{\delta_1}$ such that $\widetilde{H^{-1}_{\delta_1}}(O)=O$. Then the map $\widetilde{\tau}: \HH^2 \to \HH^2$ defined as $\widetilde{\tau} = \widetilde{H^{-1}_{\delta_1}} \circ \widetilde{H_{\delta_0}}$ is a lift of $\tau$. We prove that for every $n \in \NN$, the endpoints of the curve $\widetilde{\tau}^n(\widetilde{\rho\alpha_3})$ in $\partial \HH^2$ belong to $(pq)\cup(rs)$, from which lemma will follow as we now show. Denote by $g_n$ the geodesic joining the endpoints of $\widetilde{\tau}^n(\widetilde{\rho\alpha_3})$. Then since $i(\alpha_0, \delta_0)=0$, the set of geodesic lifts of $\alpha_0$ in $\HH^2$ that intersect both $\widetilde{\alpha_2}$ and $g_n$ coincides with the set of geodesic lifts of $\alpha_0$ in $\HH^2$ that intersect both $\widetilde{\alpha_2}$ and $\widetilde{\rho\alpha_3}$. Then the projections of $g_n$ and $\widetilde{\rho\alpha_3}$ to the annular cover that corresponds to the hyperbolic isometry with the axis $\widetilde{\alpha_2}$ and translation length $\ell_{\alpha_2}(X)$ intersect at most once. It follows that $|d_{\alpha_2}(\alpha_0,\tau^{n}(\rho\alpha_3)) - d_{\alpha_2}(\alpha_0,\rho\alpha_3)| \leqslant  4$ for every $n \in \NN$. Then 
by \lemref{lemma: subsurface projection upper bound} we have
$$
d_{\alpha_2}(\alpha_0,\tau^{r_{i-1}}(\rho\alpha_3)) \leqslant d_{\alpha_2}(\alpha_0,\rho\alpha_3) + 4 \leqslant (i(\alpha_0,\rho\alpha_3) + 1) + 4 = 7.
$$
Now we prove that for every $n \in \NN$, the endpoints of the curve $\widetilde{\tau}^n(\widetilde{\rho\alpha_3})$ in $\partial \HH^2$ belong to $(pq)\cup(rs)$. Let $\widetilde{\delta_1}$ be the geodesic lift of $\delta_1$ in $\HH^2$ that intersects $\widetilde{\delta'_0}$ at the point $C \in \HH^2$ such that the geodesic segment $[BC] \subset \HH^2$ does not contain any other intersection points of lifts of $\delta_1$ with $\widetilde{\delta'_0}$. Let $t,v \in \partial \HH^2$ be the endpoints of $\widetilde{\delta_1}$ and let $(tv) \subset \partial \HH^2$ be the open interval such that $r \in (tv)$. Let $(rv) \subset \partial \HH^2$ be the open interval such that $t \notin (rv)$. Let $w \in \partial \HH^2$ be the endpoint of $\widetilde{\rho\alpha_3}$ such that $w \in (rs)$. Let $(rw] \subset \partial \HH^2$ be the half-open interval such that $v \in (rw]$. See \figref{figure: universal cover} (right).

It follows from Proposition 2.1 in \cite{Sm01} that the boundary extension of $\widetilde{H_{\delta_0}}$ fixes $r,s \in \partial \HH^2$ and moves all points in $(rs)$ counterclockwise (Proposition 2.1 is about Dehn twists, but the argument applies to half-twists as well). Similarly, the boundary extension of $\widetilde{H^{-1}_{\delta_1}}$ fixes $t,v \in \partial \HH^2$ and moves all points in $(tv)$ clockwise. Further,
every point in $(rw]$ is either fixed under the boundary extension of $\widetilde{H^{-1}_{\delta_1}}$ (such as the point $v$) or moves clockwise (such as any point in $(rv)$). Since $i(\rho\alpha_3,\delta_1)=0$, no geodesic lift of $\delta_1$ intersects $\widetilde{\rho\alpha_3}$, therefore no point in $(rw]$ moves past $w$ under the boundary extension of $\widetilde{H^{-1}_{\delta_1}}$. It follows that the boundary extension of 
$\widetilde{\tau} = \widetilde{H^{-1}_{\delta_1}} \circ \widetilde{H_{\delta_0}}$ maps $(rw]$ to itself. Hence for every $n \in \NN$, the boundary extension of
$\widetilde{\tau}^n$ maps the point $w \in \partial \HH^2$ to a point in $(rs)$. By a similar argument, the other endpoint of $\widetilde{\rho\alpha_3}$ is mapped to a point in $(pq)$ under the boundary extension of 
$\widetilde{\tau}^n$ for every $n \in \NN$, which concludes the proof.
\begin{figure}[H]
  \begin{tikzpicture}

\draw (0,0) circle (2);

\draw  (-2,0) -- (2,0) node at (-1,-0.18) {\tiny $A$} node[midway, below left] {\tiny $O$}  node at (0.6,-0.18) {\tiny $B$} node[at end, below left] {\tiny $\tilde{\rho\alpha}_{3}$};

\draw (0,-2) -- (0,2) node[near end, left] {\tiny $\tilde{\alpha}_{2}$};

\filldraw[black] (-0.79,0) circle (1pt);

\filldraw[black] (0,0) circle (1pt);

\filldraw[black] (0.79,0) circle (1pt);

\draw (1.42,1.42) arc (132:228:1.9) node[near end, left] {\tiny $\tilde{\delta}'_{0}$} node[at end, left] {\tiny $s$} node[at start, left] {\tiny $r$};

\draw  (-1.42,-1.42) arc (-48:48:1.9) node[near start, right] {\tiny $\tilde{\delta}_{0}$} node[at end, right] {\tiny $p$} node[at start, right] {\tiny $q$};

\end{tikzpicture}
    \hspace*{.5in}
  \begin{tikzpicture}

\draw (0,0) circle (2);

\draw 
(-2,0) -- (2,0) node[midway, below left] {\tiny $O$} node at (0.6,-0.18) {\tiny $B$}  node at (1.85,-0.12) {\tiny $w$} node at (-1, -0.2) {\tiny $\tilde{\rho\alpha}_{3}$};

\draw (0,-2) -- (0,2) node[near end, left] {\tiny $\tilde{\alpha}_{2}$};

\filldraw[black] (0,0) circle (1pt);

\filldraw[black] (0.79,0) circle (1pt);

\filldraw[black] (0.96,0.81) circle (1pt);

\draw (1.42,1.42) arc (132:228:1.9) node[at end, left] {\tiny $s$} node[at start, left] {\tiny $r$} node[near end, left]  {\tiny $\tilde{\delta}'_{0}$};

\draw (0.6,1.9) arc (160:295:0.92) node at (0.48,1.82) {\tiny $t$} node at (1.8,0.6) {\tiny $v$}  node at (0.75,0.75) {\tiny $C$} node[near start, right] {\tiny $\tilde{\delta}_{1}$};

\end{tikzpicture}
    \caption{Left: lifts of the curves $\alpha_2, \rho\alpha_3, \delta_0$ in the universal cover. Right: lifts of the curves $\alpha_2, \rho\alpha_3, \delta_0, \delta_1$ in the universal cover.}
    \label{figure: universal cover}
\end{figure}
\end{proof}
Let $\mu_{1} = \{\alpha_{i_1},\alpha_{i_1+5}\}$ be a collection of curves on $S$, where $i_1$ is the constant from \lemref{Lemma: Filling pair in 5+ steps}. By \lemref{Lemma: Filling pair in 5+ steps} the collection of curves $\mu_{1}$ is a marking on $S$. We prove: 

\begin{proposition}
\label{prop: uniformly bounded annular projections}
There is a constant $E \in \NN$ such that the following holds. For every curve $\gamma$ on $S$ there is $j_{\gamma} \in \NN$ such that for all $j \geqslant j_{\gamma}$ the curve $\alpha_j$ intersects $\gamma$ essentially and $$d_{\gamma}(\mu_{1},\alpha_j) \leqslant E.$$
\end{proposition}

\begin{proof}
If the curve $\gamma$ intersects every curve $\alpha_{j}$ essentially for $j \geqslant i_1$, then by \corref{corollary: BGI} we have
$$
d_{\gamma}(\mu_{1}, \alpha_j) = \max \{ d_{\gamma}(\alpha_{i_1}, \alpha_j), d_{\gamma}(\alpha_{i_1+5}, \alpha_j) \}  \leqslant A
$$
for every $j \geqslant i_1$.
Otherwise, the curve $\gamma$ is disjoint from some curves in the sequence $\{\alpha_i \}_{i\geqslant i_1}$. If $\gamma$ is disjoint from $\alpha_{i_1}$, then by \lemref{Lemma: Filling pair in 5+ steps},\ $\gamma$ intersects every $\alpha_{j}$ essentially for $j \geqslant i_1+5$. Then by \corref{corollary: BGI}
$$
d_{\gamma}(\mu_{1}, \alpha_j) = d_{\gamma}(\alpha_{i_1+5}, \alpha_j) \leqslant A
$$
for every $j \geqslant i_1+5$. If $\gamma$ intersects $\alpha_{i_1}$ essentially, let $\ell > i_1$ be the smallest index so that $\gamma$ is disjoint from $\alpha_{\ell}$ and $r \geqslant \ell$ be the largest index so that $\gamma$ is disjoint from $\alpha_r$. By \lemref{Lemma: Filling pair in 5+ steps}, we have $r-\ell \leqslant 4$.  Let $j_{\gamma} = r+1$. Then for $j \geqslant j_{\gamma} =  r+1$ by the triangle inequality we have
\begin{equation}
\label{eq: relative twisting bounds, triangle inequality}
d_{\gamma}(\mu_{1},\alpha_j) \leqslant d_{\gamma}(\mu_{1},\alpha_{\ell-1}) + d_{\gamma}(\alpha_{\ell-1},\alpha_{r+1}) + d_{\gamma}(\alpha_{r+1}, \alpha_j).
\end{equation}
By \corref{corollary: BGI}, $d_{\gamma}(\alpha_{r+1}, \alpha_j) \leqslant A$. Next, we show that $d_{\gamma}(\mu_{1},\alpha_{\ell-1}) \leqslant \max\{A,4\} + d_{\gamma}(\alpha_{\ell-1},\alpha_{r+1})$, thus it will remain to find an upper bound for $d_{\gamma}(\alpha_{\ell-1},\alpha_{r+1})$. If $\gamma$ intersects $\alpha_{i_1}$ essentially and is disjoint from $\alpha_{i_1+5}$, then by \corref{corollary: BGI} $$d_{\gamma}(\mu_{1},\alpha_{\ell-1}) = d_{\gamma}(\alpha_{i_1},\alpha_{\ell-1}) \leqslant A.$$
Suppose that $\gamma$ intersects both $\alpha_{i_1}$ and $\alpha_{i_1+5}$ essentially. If $\ell - 1 \geqslant i_1+5$, then by \corref{corollary: BGI}
$$
d_{\gamma}(\mu_{1},\alpha_{\ell-1}) = \max \{ d_{\gamma}(\alpha_{i_1},\alpha_{\ell-1}), d_{\gamma}(\alpha_{i_1+5},\alpha_{\ell-1}) \} \leqslant A,
$$
If $\ell - 1 < i_1+5$ and $r+1 \leqslant i_1+5$, then by the triangle inequality and \corref{corollary: BGI}
$$
d_{\gamma}(\alpha_{i_1+5},\alpha_{\ell-1}) \leqslant d_{\gamma}(\alpha_{i_1+5},\alpha_{r+1}) + d_{\gamma}(\alpha_{r+1},\alpha_{\ell-1}) \leqslant A + d_{\gamma}(\alpha_{\ell-1},\alpha_{r+1}).
$$
Therefore we have
$$
d_{\gamma}(\mu_{1},\alpha_{\ell-1}) = \max \{ d_{\gamma}(\alpha_{i_1},\alpha_{\ell-1}), d_{\gamma}(\alpha_{i_1+5},\alpha_{\ell-1}) \} \leqslant \max \{ A, A + d_{\gamma}(\alpha_{\ell-1},\alpha_{r+1}) \} = A + d_{\gamma}(\alpha_{\ell-1},\alpha_{r+1}).
$$
If $\ell - 1 < i_1+5$ and $r+1 > i_1+5$, then since $\gamma$ intersects $\alpha_{i_1+5}$ essentially, we have $\ell < i_1+5$ and $r > i_1+5$. Then the curves in $\{\alpha_i\}_{i\geqslant i_1}$ that are disjoint from $\gamma$ are not consecutive. It follows from \corref{corollary: footprint on alphas} and \clmref{claim: almost filling in 3 steps} that either $i_1+5 = r-1$ or $i_1+5 = r-2$ and $\gamma$ intersects $\alpha_{r-1}$ essentially. In the first case, 
by \lemref{lemma: subsurface projection upper bound} and \eqnref{Eq: no-interac-with-pA} we have 
$$
d_{\gamma}(\alpha_{i_1+5},\alpha_{r+1}) = d_{\gamma}(\alpha_{r-1},\alpha_{r+1}) \leqslant i(\alpha_{r-1},\alpha_{r+1}) + 1 = i(\alpha_0, \alpha_2) + 1 = 3.
$$
In the second case, by the triangle inequality
$$
d_{\gamma}(\alpha_{i_1+5},\alpha_{r+1}) = d_{\gamma}(\alpha_{r-2},\alpha_{r+1}) \leqslant d_{\gamma}(\alpha_{r-2},\alpha_{r-1}) + d_{\gamma}(\alpha_{r-1},\alpha_{r+1}).
$$
Notice that $d_{\gamma}(\alpha_{r-2},\alpha_{r-1}) = 1$ since $\alpha_{r-2}$ and $\alpha_{r-1}$ are disjoint. Since $d_{\gamma}(\alpha_{r-1},\alpha_{r+1}) \leqslant 3$, we have $d_{\gamma}(\alpha_{i_1+5},\alpha_{r+1}) \leqslant 4$. We obtain that if $\ell - 1 < i_1+5$ and $r+1 > i_1+5$, then
\begin{equation}
\label{eq: relative twisting bounds, bound of a term}
\begin{split}
d_{\gamma}(\mu_{1},\alpha_{\ell-1}) & = \max \{ d_{\gamma}(\alpha_{i_1},\alpha_{\ell-1}), d_{\gamma}(\alpha_{i_1+5},\alpha_{\ell-1}) \} \\ & \leqslant \max \{ A , d_{\gamma}(\alpha_{i_1+5},\alpha_{r+1}) + d_{\gamma}(\alpha_{r+1},\alpha_{\ell-1}) \} \\ & \leqslant \max \{ A , 4 + d_{\gamma}(\alpha_{r+1},\alpha_{\ell-1}) \} \\ & \leqslant \max\{A,4\} + d_{\gamma}(\alpha_{\ell-1},\alpha_{r+1}).
\end{split}
\end{equation}
Now we find an upper bound for $d_{\gamma}(\alpha_{\ell-1},\alpha_{r+1})$. Depending on the value of $r-\ell$, we consider the following cases:

\textbf{Case: $r-\ell = 4.$} By \clmref{claim: almost filling in 4 steps}, we have $\gamma = \beta_{\ell}$. By \eqnref{Eq: no-interac-with-pA}, we have
$$
d_{\gamma}(\alpha_{\ell-1},\alpha_{r+1}) = d_{\beta_{\ell}} ( \alpha_{\ell-1}, \alpha_{\ell+5}) = d_{\beta_{0}} ( \varphi_{r_{\ell}}^{-1} (\alpha_0), \varphi_{r_{\ell+1}} \varphi_{r_{\ell+2}} (\alpha_3)).
$$
Let $\nu', \nu''$ denote the limits of the laminations $\varphi_{r_{\ell}}^{-1} (\alpha_0), \varphi_{r_{\ell+1}} \varphi_{r_{\ell+2}} (\alpha_3)$ in the Hausdorff topology as $\ell \to \infty$, respectively. Notice that $\nu'$ and $\nu''$ intersect $\beta_0$ essentially. Then by \lemref{Lemma: Convergence of projection distances}, there is $\ell_0 \in \NN$ such that for every $\ell \geqslant \ell_0$, we have
$$
d_{\beta_{0}} ( \varphi_{r_{\ell}}^{-1} (\alpha_0), \varphi_{r_{\ell+1}} \varphi_{r_{\ell+2}} (\alpha_3)) \eadd_{\,16} d_{\beta_{0}} (\nu', \nu'').
$$
Then we have 
\begin{equation}
\label{eq: relative twisting bounds, r-l = 4}
d_{\gamma}(\alpha_{\ell-1},\alpha_{r+1}) \leqslant \max \Big\{\max_{\ell < \ell_0}\big\{d_{\beta_{0}} ( \varphi_{r_{\ell}}^{-1} (\alpha_0), \varphi_{r_{\ell+1}} \varphi_{r_{\ell+2}} (\alpha_3))\big\}, d_{\beta_{0}} (\nu', \nu'') + 16\Big\}.
\end{equation}

\textbf{Case: $r-\ell = 3.$} By \clmref{claim: almost filling in 3 steps} and the triangle inequality we can write 
$$
d_{\gamma}(\alpha_{\ell-1},\alpha_{r+1}) = d_{\gamma}(\alpha_{\ell-1},\alpha_{\ell+4}) \leqslant d_{\gamma}(\alpha_{\ell-1},\alpha_{\ell+1}) + d_{\gamma}(\alpha_{\ell+1},\alpha_{\ell+2}) + d_{\gamma}(\alpha_{\ell+2},\alpha_{\ell+4}).
$$
Notice that $d_{\gamma}(\alpha_{\ell+1},\alpha_{\ell+2}) = 1$ since $\alpha_{\ell+1}$ and $\alpha_{\ell+2}$ are disjoint. By \lemref{lemma: subsurface projection upper bound} we also have
$$d_{\gamma}(\alpha_{\ell-1},\alpha_{\ell+1}) \leqslant i(\alpha_{\ell-1},\alpha_{\ell+1}) + 1 = i(\alpha_0, \alpha_2) + 1 = 3.$$
Similarly, $d_{\gamma}(\alpha_{\ell+2},\alpha_{\ell+4}) \leqslant 3$. Hence $d_{\gamma}(\alpha_{\ell-1},\alpha_{r+1}) \leqslant 7$.

\textbf{Case: $r-\ell = 2.$} If $\gamma$ is disjoint from $\alpha_{\ell}$ and $\alpha_{\ell+2}$, then one of the following holds: $\gamma = \beta_{\ell-1}, \gamma = \alpha_{\ell+1}, \gamma \subset Y_{\ell+1}$. Indeed, by applying the homeomorphism $\Phi_{\ell-1}^{-1}$ and by \eqnref{Eq: no-interac-with-pA} it is enough to consider the case $\ell=1$, which follows from \figref{figure with alphas}. If $\gamma = \beta_{\ell-1}$, then by \clmref{claim: almost filling in 4 steps} $\gamma$ is disjoint from $\alpha_{\ell+3}$, which is impossible. If $\gamma = \alpha_{\ell+1}$, then by \lemref{lemma: annular projections are small} we have
$$
d_{\gamma} (\alpha_{\ell-1}, \alpha_{r+1}) = d_{\alpha_{\ell+1}}(\alpha_{\ell-1}, \alpha_{\ell+3}) \leqslant 7. 
$$
If $\gamma \subset Y_{\ell+1}$, we have
$$
d_{\gamma} (\alpha_{\ell-1}, \alpha_{r+1}) = d_{\gamma}(\alpha_{\ell-1}, \alpha_{\ell+3}) = d_{\Phi^{-1}_{\ell-1}\gamma}(\alpha_0, \tau^{r_{\ell}}(\rho\alpha_3)),
$$
where $\Phi^{-1}_{\ell-1}\gamma$ is a curve in $Y_2$. 
Denote the curve $\Phi^{-1}_{\ell-1}\gamma$ by $\gamma'$.
Let $\nu_{\tau}$ denote the unstable lamination of $\tau$. Notice that $\nu_{\tau}$ intersects every curve in $Y_2$ essentially and that the curves $\tau^{n}(\rho\alpha_3)$ converge in the Hausdorff topology as $n \to \infty$ to a lamination that contains $\nu_{\tau}$. Then for sufficiently large $n \in \NN$ so that \lemref{Lemma: Convergence of projection distances} applies and the curve $\tau^{n}(\rho\alpha_3)$ intersects $\gamma'$ essentially, we have
$$
d_{\gamma'}(\alpha_0, \tau^{n}(\rho\alpha_3)) \eadd_{\,9} d_{\gamma'}(\alpha_0, \nu_{\tau}).
$$
By \propref{Prop: pA}, we have $d_{\gamma'}(\alpha_0, \nu_{\tau}) \leqslant C_{\tau, \alpha_0}$. By the triangle inequality, we have
$$
d_{\gamma'}(\alpha_0, \tau^{r_{\ell}}(\rho\alpha_3)) \leqslant d_{\gamma'}(\alpha_0, \tau^{n}(\rho\alpha_3)) + d_{\gamma'}(\tau^{n}(\rho\alpha_3), \tau^{r_{\ell}}(\rho\alpha_3)) \leqslant (C_{\tau, \alpha_0} + 9) + d_{\tau^{-r_{\ell}}(\gamma')}(\tau^{n-r_{\ell}}(\rho\alpha_3), \rho\alpha_3),
$$
where $\tau^{-r_{\ell}}(\gamma')$ is a curve in $Y_2$. Denote the curve $\tau^{-r_{\ell}}(\gamma')$ by $\gamma''$. Since $r_{\ell} \in \NN$ is fixed, for sufficiently large $n \in \NN$ so that \lemref{Lemma: Convergence of projection distances} applies, together with \propref{Prop: pA} we have
$$
d_{\gamma''}(\tau^{n-r_{\ell}}(\rho\alpha_3), \rho\alpha_3) \eadd_{\,9} d_{\gamma''}(\nu_{\tau},\rho\alpha_3) \leqslant C_{\tau, \rho\alpha_3}.  
$$
Therefore we have
\begin{equation}
\label{eq: relative twisting bounds, r-l = 2}
d_{\gamma} (\alpha_{\ell-1}, \alpha_{r+1}) \leqslant C_{\tau, \alpha_0} + C_{\tau, \rho\alpha_3} + 18.
\end{equation}

\textbf{Case: $r-\ell = 1.$} This case is impossible by \clmref{claim: no two consecutive curves}.

\textbf{Case: $r-\ell = 0.$} By \lemref{lemma: subsurface projection upper bound} we have
$$
d_{\gamma}(\alpha_{\ell-1},\alpha_{r+1}) = d_{\gamma}(\alpha_{\ell-1},\alpha_{\ell+1}) \leqslant i(\alpha_{\ell-1},\alpha_{\ell+1}) + 1  = i(\alpha_0,\alpha_2) + 1 = 3.
$$

Finally, according to \eqnref{eq: relative twisting bounds, triangle inequality}, \eqnref{eq: relative twisting bounds, bound of a term}, \eqnref{eq: relative twisting bounds, r-l = 4}, \eqnref{eq: relative twisting bounds, r-l = 2}, if we let 
$$
E = \max\{A,4\} + 2 \cdot \max\bigg\{ \max \Big\{\max_{\ell < \ell_0}\big\{d_{\beta_{0}} ( \varphi_{r_{\ell}}^{-1} (\alpha_0), \varphi_{r_{\ell+1}} \varphi_{r_{\ell+2}} (\alpha_3))\big\}, d_{\beta_{0}} (\nu', \nu'') + 16\Big\}, C_{\tau, \alpha_0} + C_{\tau, \rho\alpha_3} + 18 \bigg\} + A,
$$
then $d_{\gamma}(\mu_1,\alpha_j) \leqslant E$ for every $j \geqslant j_{\gamma} = r+1$, which concludes the proof.
\end{proof}

In the following corollary, $\lambda$ is the non-uniquely ergodic ending lamination on $S$ constructed in \secref{section: construction of the lamination}.
\begin{corollary}
\label{corollary: uniformly bounded annular projection distances}
There is a constant $E' \in \NN$ such that $d_{\gamma}(\mu_1,\lambda) \leqslant E'$ for all curves $\gamma$ on $S$.
\end{corollary}
\begin{proof}
By \corref{corollary: ending lamination}, there is a subsequence of $\{\alpha_i\}$ that converges in the Hausdorff topology on $\calG \calL(S)$ to a geodesic lamination $\lambda'$ that contains $\lambda$. Taking an index $i \in \NN$ in the subsequence sufficiently large so that \lemref{Lemma: Convergence of projection distances} applies for the annular subsurface of a curve $\gamma$ on $S$ and so that $\alpha_i$ intersects $\gamma$ essentially, we obtain
$$
d_{\gamma}(\mu_1,\alpha_i) \eadd_{\,8} d_{\gamma}(\mu_1,\lambda').
$$
Since $\lambda \subset \lambda'$, we have $d_{\gamma}(\mu_1,\lambda) \leqslant d_{\gamma}(\mu_1,\lambda')$. Taking $i \in \NN$ sufficiently large so that \propref{prop: uniformly bounded annular projections} applies as well, we have 
$$
d_{\gamma}(\mu_1,\lambda) \leqslant d_{\gamma}(\mu_1,\lambda') \leqslant d_{\gamma}(\mu_1,\alpha_i) + 8 \leqslant E+8.
$$
Letting $E' = E+8$ concludes the proof.
\end{proof}
We remark that not all projection distances for $\lambda$ are uniformly bounded. We prove the following:
\begin{claim}
\label{claim: some projections are large}
Let $\nu$ be a minimal, filling geodesic lamination on $S$ such that $d_{Y}(\mu_1 ,\nu) \leqslant G$ for some constant $G > 0$ and all subsurfaces $Y \subset S$. Then
$$
d_{Y_i}(\nu, \lambda) \geqslant c \, r_{i-1} - G - 18
$$
for all $i \geqslant i_1+2$.
\end{claim}
\begin{proof}
By \lemref{Lemma: Subsurface projections are big}, we have $$d_{Y_i}(\mu_{1}, \alpha_j) \geqslant c \, r_{i-1}-9$$ 
for all $i \geqslant i_1+2$ and $j \geqslant i+2$. By an argument, similar to the one in \corref{corollary: uniformly bounded annular projection distances}, we have $$d_{Y_i}(\mu_{1}, \lambda) \geqslant (c \, r_{i-1}-9)-9$$ 
for all $i \geqslant i_1+2$. Then by the triangle inequality we have 
$$d_{Y_i}(\nu, \lambda) \geqslant d_{Y_i}(\mu_{1},\lambda) - d_{Y_i}(\mu_{1}, \nu) \geqslant c \, r_{i-1}-18 - G$$ for all $i \geqslant i_1+2$.
\end{proof}
We obtain the following corollary which is in contrast with \thmref{theorem: main theorem}.
\begin{corollary}
Suppose $X_t$ is a \Teich geodesic such that the support of the lamination that corresponds to its vertical foliation contains the support of $\lambda$ constructed in \secref{section: construction of the lamination}, and such that the support of the lamination that corresponds to its horizontal foliation contains the support of $\nu$ as in \clmref{claim: some projections are large}. Then for all sufficiently large $i \in \NN$, the minimal length $\ell_{\alpha_i}$ of the curve $\alpha_i$ along $X_t$ satisfies:
$$
\ell_{\alpha_i} \lmul \frac{1}{r_{i-1}}.
$$
\end{corollary}
\begin{proof}
Since the sequence $\{r_n\}$ is strictly increasing, we can choose $i_G \geqslant i_1+2$ such that $c \, r_{i-1} - G - 18 \geqslant \frac{c}{2}r_{i-1}$ for all $i \geqslant i_G$. Then the statement follows from \clmref{claim: some projections are large} and Theorem 6.1 in \cite{Raf05}. In particular, $X_t$ does not stay in the thick part of the \Teich space. Moreover, it follows from \thmref{Theorem: Masur's criterion in Zorich} that $X_t$ diverges in the moduli space as $t \to \infty$.
\end{proof}

\section{Geodesics in the thick part}
\label{section: geodesic}
In this section we prove \thmref{theorem: main theorem}. First, we prove some technical lemmas.


\begin{lemma} 
\label{lemma: short curves and convergent sequence in Teich space} 
Let $X_n \in \calT(S)$ be a sequence in \Teich space converging to $[\xi]$ in the Thurston boundary, and let $\eta_n$ be a curve on $S$ such that $\ell_{\eta_n}(X_n) \leqslant C$ for some $C > 0$. If $[\eta]$ is a limit point of the sequence $[\eta_n]$ in the Thurston boundary, then $i(\xi, \eta)=0$.
\end{lemma}
\begin{proof}
By definition, there is a sequence $\{a_n\}$ of positive numbers, such that $a_n X_n \to \xi$ as geodesic currents. We have (Prop. 15 in \cite{Bon88}):
$$
i(a_n X_n, a_n X_n) = a_n^2 \, i(X_n,X_n) = a_n^2 \pi^2 |\chi(S)|.
$$
By the continuity of the intersection number, $i(a_n X_n, a_n X_n) \to i(\xi,\xi) = 0$, since $\xi \in \ML(S)$. Hence $a_n^2 \to 0$, and in particular, $a_n \to 0$.
By definition, there is a sequence $\{b_n\}$ of non-negative numbers, such that $b_n \eta_n \to \eta$ as geodesic currents. Let $\gamma$ be a filling collection of curves on $S$, then $i(\gamma, b_n \eta_n) = b_n i(\gamma,\eta_n) \geqslant b_n$. We also have $i(\gamma, b_n \eta_n) \to i(\gamma,\eta) < \infty$. Hence the sequence $\{b_n\}$ is bounded from above, so suppose $b_n \leqslant B$ for some $B > 0$. Then
$$i(a_n X_n, b_n \eta_{n}) = a_{n} b_n \ell_{\eta_n}(X_{n}) \leqslant a_{n} BC.$$
Since $i(a_n X_n, b_n \eta_{n}) \to i(\xi, \eta)$, we obtain $i(\xi, \eta)=0$.
\end{proof}

Let $B(S)$ be a Bers constant of $S$. We prove:
\begin{lemma}
\label{lemma: convergence of the relative twisting} 
Let $X_n, Y_n \in \calT(S)$ be sequences in \Teich space converging to $[\xi]$ and $[\zeta]$ in the Thurston boundary, respectively. Suppose that the supports of $\xi$ and $\zeta$ are minimal and filling. If $\alpha$ is a curve on $S$ such that $\ell_{\alpha}(X_n), \ell_{\alpha}(Y_n) > B(S)$ for all $n \in \NN$, then
$$
d_{\alpha}(X_n,Y_n) \, \eadd \, d_{\alpha}(\xi, \zeta)
$$
for infinitely many $n \in \NN$.
\end{lemma} 

\begin{proof}
It follows from the definition of a Bers constant that for every $n \in \NN$ there are curves $\eta_n$ and $\nu_n$ on $S$ that intersect $\alpha$ essentially such that $\ell_{\eta_n}(X_n), \ell_{\eta_n}(X_n) \leqslant B(S)$. By the triangle inequality we have
$$
d_{\alpha}(X_n,Y_n) \leqslant d_{\alpha}(X_n,\eta_n) + d_{\alpha}(\eta_n,\nu_n) + d_{\alpha}(\nu_n,Y_n).
$$
Similarly,
$$
d_{\alpha}(\eta_n,\nu_n) \leqslant d_{\alpha}(\eta_n,X_n) + d_{\alpha}(X_n,Y_n) + d_{\alpha}(Y_n,\nu_n).
$$
Hence $|d_{\alpha}(X_n,Y_n) - d_{\alpha}(\eta_n,\nu_n)| \leqslant d_{\alpha}(X_n,\eta_n)+d_{\alpha}(Y_n,\nu_n)$. It is sufficient to show that $d_{\alpha}(X_n,\eta_n),d_{\alpha}(Y_n,\nu_n) \eadd 0$ and that $d_{\alpha}(\eta_n,\nu_n) \eadd  d_{\alpha}(\xi, \zeta)$ for infinitely many $n \in \NN$.

We show that the relative twisting coefficients $d_{\alpha}(X_n, \eta_n)$ are uniformly bounded, the case of $d_{\alpha}(Y_n,\nu_n)$ is identical. Let $\ell_n = \ell_{\alpha}(X_n)$. By the Collar Lemma (\cite{FM12}, Section 13.5), the $\omega_n$-neighborhood (collar) of the geodesic representative of $\alpha$ in $X_n$ for $\omega_n = \arcsinh\left(\frac{1}{\sinh(\ell_n/2)}\right)$ is embedded in $X_n$. Consider an arc $\widehat{\eta}_n$ of the geodesic representative of $\eta_n$ inside the collar of $\alpha$ in $X_n$ with one endpoint on $\alpha$ and the other endpoint on the boundary of the neighborhood. Since the collar is embedded, the length of $\widehat{\eta}_n$ is at most $B(S)$. From the trigonometry of right triangles,
we find a lower bound on the angle $\delta_n$ that $\widehat{\eta}_n$ makes with $\alpha$ in $X_n$: 
$$
\sin \delta_n \geqslant \frac{\sinh{\omega_n}}{\sinh B(S)}.
$$
Denote by $L_n$ the length of the orthogonal projection of a lift of $\eta_n$ on a lift of $\alpha$ in the universal cover of $X_n$ that intersect at the angle $\delta_n$. Then from the angle of parallelism formula, we have $\cosh\frac{L_n}{2}\sin \delta_n = 1.$ Since $\sinh x \leqslant e^x/2$ and $\arccosh x \leqslant \ln 2x$ for $x>0$, we find:
$$
    L_n \leqslant 2\arccosh \frac{\sinh B(S)}{\sinh \omega_n} = 2\arccosh (\sinh B(S) \sinh (\ell_n/2)) \leqslant 2 \ln (\sinh B(S)e^{\ell_n/2}) \leqslant \ell_n + 2B(S)-2\ln2 < 3\ell_n.
$$
We estimate the relative twisting coefficients (see \cite{Min96}, Section 3):
$$
d_{\alpha}(X_n, \eta_n) \eadd_{\,2} \frac{L_n}{\ell_n} \eadd_{\,3} 0.
$$
We show that $d_{\alpha}(\eta_n,\nu_n) \eadd  d_{\alpha}(\xi, \zeta)$ for infinitely many $n \in \NN$. Let $[\eta] \in \PP \ML(S)$ be the limit of a subsequence of $[\eta_n]$. By \lemref{lemma: short curves and convergent sequence in Teich space}, $i(\xi, \eta)=0$. Since $\xi$ is minimal and filling, we have $\text{supp}(\xi) = \text{supp}(\eta)$, in particular $\eta$ intersects $\alpha$ essentially and $d_{\alpha}(\xi, \zeta) = d_{\alpha}(\eta, \zeta)$. Let $\eta'$ be the limit of a further subsequence of $\{ \eta_n \}$ in $\calG\calL(S)$. Then $\text{supp}(\eta) \subset \text{supp}(\eta')$, hence $d_{\alpha}(\eta, \zeta) \eadd_{\,1} d_{\alpha}(\eta', \zeta)$. By \lemref{Lemma: Convergence of projection distances}, $d_{\alpha}(\eta', \zeta) \eadd d_{\alpha}(\eta_n, \zeta)$ for infinitely many $n \in \NN$. By a similar argument for $\{\nu_n\}$, we have $d_{\alpha}(\eta_n, \zeta) \eadd d_{\alpha}(\eta_n, \nu_n)$, hence $d_{\alpha}(\xi, \zeta) \eadd d_{\alpha}(\eta_n, \nu_n) $ for infinitely many $n \in \NN$, which proves the lemma.
\end{proof}

Together with \thmref{Theorem: Short curves along a Thurston geodesic}, we obtain the following corollary:
\begin{corollary}[Bounded annular combinatorics implies cobounded]
\label{corollary: bounded annular combinatorics implies cobounded}
Let $\calG(t), \, t \in \RR$ be a stretch path in $\calT(S)$ with the horocyclic foliation $[\xi]$
such that $\calG(t) \to [\zeta] \in \PP \ML(S)$ as $t \to -\infty$. Suppose that the supports of $\xi$ and $\zeta$ are minimal and filling. If there exists a number $K \in \NN$ such that $d_{\alpha}(\xi,\zeta) \leqslant K$ for all curves $\alpha$ on $S$, then there exists $\varepsilon(K) > 0$ such that $\calG(t)$ lies in the thick part $\calT_{\varepsilon}(S)$ for all $t \in \RR$.
\end{corollary}
\begin{proof}
Suppose that there is a curve $\alpha$ on $S$ that gets shorter than $\varepsilon_0$ along the geodesic $\calG(t)$, where $\varepsilon_0 > 0$ is the constant in the statement of \thmref{Theorem: Short curves along a Thurston geodesic} --- otherwise there is nothing to prove. Since $\calG(t)$ is a stretch path,
\thmref{Theorem: Short curves along a Thurston geodesic} is applicable.
Let $[a,b]$ be the $\varepsilon_0$-active interval for $\alpha$. Indeed, this interval in bounded: for example, if there is a sequence $t_i \to \infty$ such that $\ell_{\alpha}(\calG(t_i)) \leqslant \varepsilon_0$, then by \lemref{lemma: short curves and convergent sequence in Teich space} we have $i(\alpha, \xi)=0$, which is impossible since $\xi$ is minimal and filling. 
By a similar argument it can be shown that there are infinitely many numbers $m \in \NN$ such that $\ell_{\alpha}(\calG(-m)) > B(S)$. By choosing large enough $n$, so that the interval $[-n,n]$ contains the interval $[a,b]$ and \lemref{lemma: convergence of the relative twisting} applies for $X_n = \calG(-n), Y_n=\calG(n)$, we conclude by combining \thmref{theorem: stretch paths}, \thmref{Theorem: Short curves along a Thurston geodesic} with the condition $d_{\alpha}(\xi,\zeta) \leqslant K$ that there is a lower bound on the minimal length of $\alpha$ along $\calG(t)$ that depends only on $K$.
\end{proof}
Finally, we prove our main result.
\begin{proof}[Proof of \thmref{theorem: main theorem}]
Let $[\lambda]$ be the projective class of some non-zero transverse measure on the non-uniquely ergodic ending lamination $\lambda$ constructed in \secref{section: construction of the lamination}. Let $\nu$ be the unstable or stable lamination of a pseudo-Anosov $\Psi$ map on $S$, and let $\widehat{\nu}$ be a maximal lamination on $S$ obtained from $\nu$ by adding finitely many leaves. Consider the projective measured foliation on $S$ that corresponds to $[\lambda]$ and that is standard near the cusps; we also denote it by $[\lambda]$. Since $\nu$ is minimal, filling and uniquely ergodic, the set of projective measured foliations transverse to $\widehat{\nu}$ 
contains $[\lambda]$. Thus there is a point $X \in \calT(S)$ such that $[\calF_{\widehat{\nu}}(X)] = [\lambda]$ (see \secref{subsecion: Thurston metric}). Since $\text{stump}(\widehat{\nu}) = \nu$, by \thmref{theorem: stretch paths} the stretch path $\text{stretch}(X, \widehat{\nu}, t)$ converges to $[\lambda]$ as $t \to \infty$ and to $[\nu]$ as $t \to -\infty$.

By \corref{corollary: bounded annular combinatorics implies cobounded}, to prove that $\text{stretch}(X, \widehat{\nu}, t)$ stays in the thick part, it is sufficient to show that the relative twisting coefficients $d_{\alpha}(\nu,\lambda)$ are uniformly bounded for all curves $\alpha$ on $S$.
Let $\mu_1$ be the marking on $S$ from \propref{prop: uniformly bounded annular projections}. By the triangle inequality, we have
\begin{equation}
\label{eq: last equation}
d_{\alpha}(\nu,\lambda) \leqslant d_{\alpha}(\nu,\mu_1) +
d_{\alpha}(\mu_1,\lambda).
\end{equation}
By \propref{Prop: pA} and \corref{corollary: uniformly bounded annular projection distances}, we have
$$
d_{\alpha}(\nu,\lambda) \leqslant C_{\Psi, \mu_1} + E',
$$
which completes the proof.
\end{proof}



\section{Appendix}
\label{Section: Appendix}

\subsection{Convergence Lemma}

Let $\|\cdot\|$ denote the operator norm. Then $\|Y\|\geqslant 1$ for any nontrivial idempotent matrix $Y$.
The following lemma is a slight improvement over Lemma 11.1 in \cite{BGT22}.

\begin{lemma}
\label{Lemma: convergence of matrices}
Let $Y$ be an idempotent matrix and let $\{ \Delta_i \}_{i=1}^{\infty}$ be a sequence of matrices such that $\sum_{i=1}^{\infty} \| \Delta_i \| < \infty$. Let $\varepsilon_j = \sum_{i=j}^{\infty} \| \Delta_i \|$ for $j \in \NN$. Then there is $j_0 \in \NN$ such that for every $j \geqslant j_0$, the infinite product 
$$\prod_{i=j}^\infty \,\, (Y+\Delta_i)$$
converges to a matrix $X_j$ with $\| X_j - Y \| \leqslant 2 \varepsilon_{j} \|Y\|^2$. Moreover, the kernel of $Y$ is contained in the kernel of $X_j$.
\end{lemma}

\begin{proof}
Let $j_0\in \NN$ be such that $\varepsilon_{j_0} \leqslant \frac{1}{2\|Y\|}$. Now fix some $j \geqslant j_0$. For $k \geqslant j$, write 
$$
Y + \Sigma_{k} = \prod_{i=j}^{k} \,\, (Y+\Delta_i).
$$
Then $(Y+\Sigma_k)(Y+\Delta_{k+1})=Y+\Sigma_{k+1}$ and since $Y^2=Y$
    it follows that
    \begin{equation}\label{Eq: Matrix mambo jambo 1} 
    \Sigma_{k+1}=\Sigma_k Y + Y \Delta_{k+1} + \Sigma_k \Delta_{k+1}.
    \end{equation}
Multiplying on the right by $Y$ and using $Y^2=Y$ we get
    $$\Sigma_{k+1}Y=\Sigma_kY+Y\Delta_{k+1}Y+\Sigma_k\Delta_{k+1}Y$$
and applying the norm    
$$
\| \Sigma_{k+1}Y \| \leqslant \| \Sigma_kY \| + \|\Delta_{k+1}\| \cdot \|Y\|^2 + \|\Sigma_k\| \cdot \| \Delta_{k+1} \| \cdot \|Y\|.
$$
For $m > j$, by applying these inequalities for $k=j, \dotsc, m-1$ we get 
$$
\| \Sigma_{m}Y \| \leqslant \| \Sigma_j Y \| + \biggl(\|\Delta_{j+1}\| + \dotsc + \|\Delta_{m}\| \biggr) \cdot \|Y\|^2 + \biggl(\|\Sigma_j\| \cdot \| \Delta_{j+1}\| + \dotsc + \|\Sigma_{m-1}\| \cdot \| \Delta_{m} \| \biggr) \cdot \|Y\|.
$$
Since $\| Y \| \leqslant \| Y \|^2$ and using $\Sigma_j = \Delta_j$ we can write 
$$
\| \Sigma_{m}Y \| \leqslant \biggl( \|\Delta_{j}\| + \dotsc + \|\Delta_{m}\| \biggr) \cdot \|Y\|^2 + \biggl( \|\Sigma_j\| \cdot \| \Delta_{j+1}\| + \dotsc + \|\Sigma_{m-1}\| \cdot \| \Delta_{m} \| \biggr) \cdot \|Y\|.
$$

Putting this together with \eqnref{Eq: Matrix mambo jambo 1} and using $\| Y \|^2 \geqslant\| Y \|, \| Y \| \geqslant 1$, we get 

\begin{equation}
\label{Eq: Matrix mambo jambo 2} 
\begin{split}    
\| \Sigma_{k+1} \| & \leqslant \| \Sigma_k Y \| + \| \Delta_{k+1} \| \cdot \| Y \| + \| \Sigma_k \| \cdot \| \Delta_{k+1} \| \\ & \leqslant \biggl( \|\Delta_{j}\| + \dotsc + \|\Delta_{k+1}\| \biggr) \cdot \|Y\|^2 + \biggl( \|\Sigma_j\| \cdot \| \Delta_{j+1}\| + \dotsc + \|\Sigma_{k}\| \cdot \| \Delta_{k+1} \| \biggr) \cdot \|Y\| \\
& \leqslant \varepsilon_j \|Y\|^2 + \biggl( \|\Sigma_j\| \cdot \| \Delta_{j+1}\| + \dotsc + \|\Sigma_{k}\| \cdot \| \Delta_{k+1} \| \biggr) \cdot \|Y\|.
\end{split}
\end{equation}

Now we show by induction that $\| \Sigma_k \| \leqslant \frac{\varepsilon_j \|Y\|^2}{1-\varepsilon_{j+1}\|Y\|}$ for all $k \geqslant j$. 

\textbf{Base: $k=j$}. Since $\Sigma_j = \Delta_j$, we have $\|\Sigma_j\| = \|\Delta_j\| = \varepsilon_j - \varepsilon_{j+1}$. Next, using $\|Y\|^2 \geqslant 1$ we trivially have
$$
(\varepsilon_j - \varepsilon_{j+1}) (1-\varepsilon_{j+1}\|Y\|) \leqslant \varepsilon_j - \varepsilon_{j+1} \leqslant \varepsilon_j \leqslant \varepsilon_j \|Y\|^2.
$$
By the choice of $j_0$, we have $1-\varepsilon_{j+1}\|Y\| >0$, hence by dividing both sides by $(1-\varepsilon_{j+1}\|Y\|) $, we obtain
$$
\|\Sigma_j\| \leqslant \frac{\varepsilon_j \|Y\|^2}{1-\varepsilon_{j+1}\|Y\|}
$$
as desired.

\textbf{Step.} By \eqnref{Eq: Matrix mambo jambo 2}, we have

$$
\| \Sigma_{k+1} \| \leqslant  \varepsilon_j \|Y\|^2 + \frac{\varepsilon_j \|Y\|^2}{1-\varepsilon_{j+1}\|Y\|} \biggl( \| \Delta_{j+1}\| + \dotsc + \| \Delta_{k+1} \| \biggr) \cdot \| Y \| \leqslant \varepsilon_j \|Y\|^2 + \frac{\varepsilon_j \|Y\|^2}{1-\varepsilon_{j+1}\|Y\|} \cdot \varepsilon_{j+1} \|Y\|  = \frac{\varepsilon_j \|Y\|^2}{1-\varepsilon_{j+1}\|Y\|}.
$$

By the choice of $j_0$, we also have $\frac{\varepsilon_j \|Y\|^2}{1-\varepsilon_{j+1}\|Y\|} \leqslant 2 \varepsilon_{j} \|Y\|^2$. This shows that 
\begin{equation}
\label{eq: general matrix norm bound}
    \| \Sigma_k \| \leqslant 2\varepsilon_j \|Y\|^2.
\end{equation} 
It also follows that $\| X_j - Y \| \leqslant 2 \varepsilon_{j} \|Y\|^2$ if we assume the convergence.

To prove the convergence, we show that the partial products form a Cauchy sequence. For $j < k < m,$
$$
 \prod_{i=j}^m \,\, (Y+\Delta_i) - \prod_{i=j}^k \,\, (Y+\Delta_i)=
      \prod_{i=j}^{k-1} (Y+\Delta_i) \bigg(\prod_{i=k}^m (Y+\Delta_i) -
      (Y+\Delta_k)\bigg)
$$
and applying the norm
$$
\left\lVert \prod_{i=j}^m \,\, (Y+\Delta_i) - \prod_{i=j}^k \,\, (Y+\Delta_i) \right\rVert 
\leqslant \left\lVert \prod_{i=j}^{k-1} (Y+\Delta_i) \right\rVert 
\Bigg(\left\lVert\prod_{i=k}^m (Y+\Delta_i) - Y \right\rVert + \|
\Delta_k \| \Bigg) \leqslant (\|Y\| + 2\varepsilon_{j}\|Y\|^2) (2\varepsilon_k \|Y\|^2+\|\Delta_k\|)
$$
which proves the sequence is Cauchy. 

For the last statement, let $v$ be a unit vector with $Yv = 0$. Then for $k \geqslant j$,
$$\| X_k v \| = \| (X_k-Y)v \| \leqslant \|X_k-Y\| \cdot \|v\| \leqslant 2\varepsilon_k \|Y \|^2.$$
Since $X_j = (Y+\Sigma_{k-1})X_k$, we have
$$
\|X_jv\| \leqslant \| (Y+\Sigma_{k-1}) \| \| X_k v \| \leqslant  (\|Y\| + 2\varepsilon_{j}\|Y\|^2) (2\varepsilon_k \|Y\|^2).
$$
Since this is true for all $k \geqslant j$, letting $k \to \infty$ yields $X_j v = 0$. 
\end{proof}

\subsection{Angle Estimate Lemma}
Let $V$ be an inner product space.
\begin{lemma}
\label{lemma: Angle Estimate Lemma}
Let $v_1,v_2,w_1,w_2 \in V$ be such that $\angle (v_1,v_2) = 0$. Then 
$$
1-\cos \angle (v_1+w_1,v_2+w_2) \leqslant \frac{2(\|v_1\| \cdot \|w_2\|+\|v_2\| \cdot \|w_1\|+\|w_1\| \cdot \|w_2\|)}{\|v_1\|\cdot\|v_2\|}.
$$
\end{lemma}

\begin{proof}
Writing the definition of the cosine of the angle, using the triangle inequality, the fact that $\langle v_1,v_2 \rangle = \|v_1\|\cdot\|v_2\|$ and that $\langle v,w \rangle \geqslant -\|v\|\cdot \|w\|$ for $v,w \in V$, we get
\begin{equation}
\label{eq: cos lower bound}
\begin{split}
\cos \angle (v_1+w_1,v_2+w_2) &= \frac{\langle v_1+w_1,v_2+w_2 \rangle}{\|v_1+w_1\| \cdot \|v_2+w_2\|} \geqslant \frac{\langle v_1+w_1,v_2+w_2 \rangle}{(\|v_1\|+\|w_1\|)(\|v_2\|+\|w_2\|)} \\ & = \frac{\langle v_1,v_2 \rangle + \langle v_1,w_2 \rangle+\langle v_2, w_1 \rangle+\langle w_1,w_2 \rangle }{(\|v_1\|+\|w_1\|)(\|v_2\|+\|w_2\|)} \\ 
& \geqslant \frac{\|v_1\|\cdot \|v_2\| - \| v_1\|\cdot \|w_2\| -  \| v_2 \| \cdot \|w_1\|  - \| w_1  \| \cdot \|w_2 \| }{(\|v_1\|+\|w_1\|)(\|v_1\|+\|w_1\|)}.
\end{split}
\end{equation}
Then by \eqnref{eq: cos lower bound} and since $\|w_1\|,\|w_2\|\geqslant0$,
\begin{equation}
\begin{split}
1-\cos \angle (v_1+w_1,v_2+w_2) & \leqslant \frac{2(\|v_1\| \cdot \|w_2\|+\|v_2\| \cdot \|w_1\|+\|w_1\| \cdot \|w_2\|) }{(\|v_1\|+\|w_1\|)(\|v_2\|+\|w_2\|)} \\ &\leqslant \frac{2(\|v_1\| \cdot \|w_2\|+\|v_2\| \cdot \|w_1\|+\|w_1\| \cdot \|w_2\|)}{\|v_1\|\cdot\|v_2\|}.
\end{split}
\end{equation}
\end{proof}

\subsection{Interval Neighborhood Lemma}

\begin{lemma}
\label{lemma: Interval Neighborhood Lemma}

Let $I \subset \RR^n, \, n\geqslant 2$ be a closed line segment. Let $I_r \subset \RR^n$ be the $r$-neighborhood of $I$ for $r>0$. Then for every $2$-dimensional closed disk $D_R \subset \RR^n$ of radius $R > 2\sqrt{2(n-1)} \, r$, $D_R \not\subset I_r$. 

\end{lemma}

\begin{proof}

Without loss of generality assume that $I = \{ (x_1, 0, \dotsc, 0) \,|\, -t \leqslant x_1 \leqslant t \}$ for some $t > 0$. Let $B_r = \{ (x_1, x_2, \dotsc, x_n) \,|\, -t-r \leqslant x_1 \leqslant t+r, \, -r \leqslant x_i \leqslant r, \, 2\leqslant i \leqslant n \}$. Notice that $ I_r \subset B_r$. We prove that $D_R \not\subset B_r$, hence $D_R \not\subset I_r$.

Assume on the contrary that $D_R \subset B_r$. Let $c \in D_R$ be the center of $D_R$ and $a,b \in \partial D_R$ be such that the vector $a-c$ is perpendicular to the vector $b-c$. Thus there are points $a,b,c \in B_r$ such that $\|a-c\| = R, \|b-c\| = R$ and $\langle a-c, b-c \rangle=0$. Let $v_i$ denote the $i$-th coordinate of a vector $v \in \RR^n$. Since $|(a-c)_i|, |(b-c)_i| \leqslant 2r$ for $2\leqslant i \leqslant n$, we have
$$
R^2 = \|a-c\|^2 \leqslant (a-c)_1^2 + (n-1)\cdot4r^2, \, R^2 = \|b-c\|^2 \leqslant (b-c)_1^2 + (n-1)\cdot4r^2.
$$
Hence $(a-c)_1^2 \geqslant R^2 - 4(n-1)r^2, \, (b-c)_1^2 \geqslant R^2 - 4(n-1)r^2$.
Then by the triangle inequality we have
$$
|\langle a-c, b-c \rangle| \geqslant |(a-c)_1|\cdot|(b-c)_1| - \sum_{i=2}^n |(a-c)_i|\cdot|(b-c)_i| \geqslant R^2 - 4(n-1)r^2 - 4(n-1)r^2 = R^2 - 8(n-1)r^2.
$$
Since $R > 2\sqrt{2(n-1)} \, r$, we have $|\langle a-c, b-c \rangle| > 0$, contradiction.
\end{proof}


\begin{thebibliography}{99}


\bibitem[Aoug13]{Aoug13}
T. Aougab,
{\em Uniform hyperbolicity of the graphs of curves},
Geometry \& Topology, vol. 17, 2855-2875 (2013).

\bibitem[Behr06]{Behr06}
J. Behrstock,
{\em Asymptotic Geometry of the Mapping Class Group and \Teich Space}, Geometry \&  Topology, vol. 10, 1523-1578, 2006.

\bibitem[BGT22]{BGT22}
M. Bestvina, R. Gupta, J. Tao,
{\em Limit sets of unfolding paths in Outer space}, 
{\tt arXiv:2207.06992}, (2022).

\bibitem[Bon88]{Bon88}
F. Bonahon,
{\em The geometry of \Teich space via geodesic currents},
Inventiones mathematicae 92, Issue: 1, page 139-162, 1988.

\bibitem[Bon96]{Bon96}
F. Bonahon,
{\em Shearing hyperbolic surfaces, bending pleated
surfaces and Thurston’s symplectic form},
Ann. Fac. Sci. Toulouse Math. (6), 5 (2), 1996, 233–297.

\bibitem[Bow06]{Bow06}
B. Bowditch,
{\em Intersection numbers and the hyperbolicity of the curve complex},
J. reine angew. Math. 598 (2006) 105-129.

\bibitem[Bow14]{Bow14}
B. Bowditch,
{\em Uniform hyperbolicity of the curve graphs},
Pacific J. Math. 269 (2014) 269--280. 

\bibitem[BH99]{BH99}
M. Bridson, A. Haefliger,
{\em Metric Spaces of Non-Positive Curvature},
Springer, 1999.



\bibitem[BLMR20]{BLMR20}
J. Brock, C. Leininger, B. Modami, K. Rafi,
{\em Limit sets of \Teich geodesics with minimal nonuniquely ergodic vertical foliation, II},
Journal f{\"u}r reine und angewandte Mathematik, 758 (2020), 1-66 (Crelle's Journal).

\bibitem[BM15]{BabBr15}
J. Brock, B. Modami,
{\em Recurrent Weil–Petersson geodesic rays with non-uniquely ergodic ending laminations},
Geometry \& Topology 19 (2015) 3565–3601.


\bibitem[CF21]{CF21}
A. Calderon, J. Farre,
{\em Shear-shape cocycles for measured laminations and ergodic theory of the earthquake flow}, {\tt arXiv:2102.13124}, (2021).

\bibitem[CEG86]{CEG86}
R. Canary, D. Epstein, P. Green, 
{\em Notes on notes of Thurston}, University of Warwick, 1986.

\bibitem[CR07]{CR07}
Y. Choi, K. Rafi,
{\em Comparison between \Teich and Lipschitz metrics},
J. Lond. Math. Soc (2) 76 (2007), no. 3, 739-756.

\bibitem[CRS14]{CRS14}
M. Clay, K. Rafi, S. Schleimer,
{\em Uniform hyperbolicity of the curve graph via surgery sequences},
Algebraic \& Geometric Topology 14-6 (2014), 3325-3344.


\bibitem[DLRT20]{DLRT20}
D. Dumas, A. Lenzhen, K. Rafi, J. Tao,
{\em Coarse and fine geometry of the Thurston metric},
Forum of Mathematics, Sigma (2020), Vol. 8, e28, 58 pages.

\bibitem[FM12]{FM12}
B. Farb, D. Margalit,
{\em A Primer on Mapping Class Groups}, Princeton Mathematical Series, 49, (2012).

\bibitem[FLP12]{FLP12}
A. Fathi, F. Laudenbach, V. Po\'{e}naru,
{\em Thurston's Work on Surfaces}, Translated by D. Kim, D. Margalit, Princeton Mathematical Series, (Mathematical Notes, 48), (2012).


\bibitem[Ham06]{Ham06}
U. Hamenst\"{a}dt,
{\em Train tracks and the Gromov boundary of the complex of curves},
Spaces of Kleinian groups, Lond. Math. Soc. Lec. Notes 329, 187-207, Cambridge Univ. Press (2006).

\bibitem[HP92]{PH92}
J. Harer, R. Penner, {\em Combinatorics of Train Tracks}, Princeton University Press, (1992).

\bibitem[Har81]{Harvey81}
W. Harvey,
{\em Boundary structure of the modular group},
Riemann surfaces and related topics:
Proceedings of the 1978 Stony Brook Conference, volume 97 of Ann. of Math. Stud., Princeton Univ. Press, 245–251, 1981.

\bibitem[Hem01]{Hem01}
J. Hempel,
{\em 3-manifolds as viewed from the curve complex},
Topology, 40 (3): 631–657.

\bibitem[HPW15]{HPW15}
S. Hensel, P. Przytycki, R. Webb,
{\em 1-slim triangles and uniform hyperbolicity for arc graphs and curve graphs}, J. Eur. Math. Soc. 17 (2015), 755–762.


\bibitem[Klar99]{Kl99}
E. Klarreich, {\em The boundary at infinity of the curve complex and the relative \Teich space}, {\tt arXiv:1803.10339}, (1999).

\bibitem[LLR18]{LLR18}
C. Leininger, A. Lenzhen, K. Rafi,
{\em Limit sets of Teichm\"uller geodesics with minimal non-uniquely ergodic vertical foliation}, J. Reine Angew. Math. 737 (2018), 1--32 (Crelle's Journal).

\bibitem[LRT12]{LRT12}
A. Lenzhen, K. Rafi, J. Tao,
{\em 
Bounded combinatorics and the Lipschitz metric on \Teich space}, Geom. Dedicata 159 (2012), 353-371. 

\bibitem[LRT15]{LRT15}
A. Lenzhen, K. Rafi, J. Tao,
{\em The shadow of a Thurston geodesic to the curve graph},
J. Topol. 8 (2015), no. 4, 1085--1118.

\bibitem[Li03]{Li03}
Z. Li, 
{\em Length spectrums of Riemann surfaces and the \Teich metric}, Bull. London Math. Soc. 35 (2003) 247–254.

\bibitem[Man13]{Man13}
J. Mangahas,
{\em A Recipe for Short-Word Pseudo-Anosovs},
American Journal of Mathematics 135 (2013), no. 4, 1087–1116.


\bibitem[Mar16]{Mar16}
B. Martelli,
{\em An Introduction to Geometric Topology}, \url{http://people.dm.unipi.it/martelli/Geometric_topology.pdf}, (2016).


\bibitem[Mas92]{Mas92} 
H. Masur, {\em Hausdorff dimension of the set of nonergodic foliations of a quadratic differential}, Duke Math. J. Volume 66, Number 3 (1992), 387-442.

\bibitem[MM99]{MasMin99}
H. Masur, Y. Minsky,
{\em Geometry of the complex of curves I: Hyperbolicity},
Inventiones mathematicae 138 (1), 103-149, 1999.


\bibitem[MM00]{MasMin00}
H. Masur, Y. Minsky,
{\em Geometry of the complex of curves II: Hierarchical structure},
GAFA, November 2000, Volume 10, Issue 4, 902–974.


\bibitem[Min96]{Min96}
Y. Minsky,
{\em Extremal length estimates and product regions in \Teich space}, Duke Math. J. Volume 83 (1996), 249-286.

\bibitem[Min00]{Min00}
Y. Minsky,
{\em Kleinian groups and the complex of curves},
Geometry \& Topology, Volume 4 (2000) 117–148.

\bibitem[Mir08]{Mir08}
M. Mirzakhani,
{\em 
Ergodic Theory of the Earthquake Flow},
International Mathematics Research Notices, Vol. 2008, Article ID rnm116, 39 pages.


\bibitem[Pap91]{Pap91}
A. Papadopoulos,
{\em On Thurston's boundary of \Teich space and the extension of earthquakes},
Topology and its Applications, 41 (1991) 147-177.

\bibitem[PT07]{PapThe07}
A. Papadopoulos, G. Th\'eret,
{\em On Teichm\"uller's metric and Thurston’s asymmetric
metric on \Teich space},
Handbook of \Teich Theory, Volume 1, 11, EMS Publishing House, (2007).

\bibitem[Raf05]{Raf05}
K. Rafi,
{\em A characterization of short curves of a \Teich geodesic}, Geometry \& Topology, Volume 9 (2005) 179–202.

\bibitem[Raf14]{Raf14}
K. Rafi,
{\em Hyperbolicity in \Teich space}, Geometry \& Topology, 18-5 (2014) 3025–3053.

\bibitem[Sm01]{Sm01} 
I. Smith,
{\em Geometric Monodromy and the Hyperbolic Disc}, The Quarterly Journal of Mathematics, Volume 52, Issue 2, July 2001, Pages 217–228.


\bibitem[Th07]{Th07}
G. Th\'eret,
{\em On the negative convergence of Thurston's stretch lines towards the boundary of \Teich space}, Ann. Acad. Sci. Fenn. Math. 32 (2007), 381-408.

\bibitem[Th14]{The14}
G. Th\'eret,
{\em Convexity of length functions and Thurston's shear coordinates}, {\tt arXiv:1408.5771}, (2014).

\bibitem[Thu86]{Th86}
W. Thurston, {\em Minimal stretch maps between hyperbolic surfaces}, {\tt arXiv:math.GT/9801039}, (1986).


\bibitem[Webb15]{Webb15}
R. Webb,
{\em Uniform bounds for bounded geodesic image theorems}, J. reine angew. Math. 709 (2015), 219–228.



\end{thebibliography}
\end{document}